\documentclass{ws-jta}

\usepackage{enumitem}
\usepackage{amsfonts}
\usepackage{blkarray}
\usepackage{bbm}
\usepackage{bigstrut}
\usepackage{cancel}
\usepackage{color}
\usepackage{dsfont}
\usepackage{extpfeil}
\usepackage{framed}
\usepackage{faktor}
\usepackage{graphicx}
\usepackage{latexsym}
\usepackage{mathtools}
\usepackage{mathabx}
\usepackage{mathrsfs}
\usepackage{multicol}
\usepackage{pb-diagram}
\usepackage{pgfplots}
\pgfplotsset{compat=newest}
\usepgfplotslibrary{polar}
\usepackage{setspace}
\usepackage{soul}
\usepackage{textcomp}
\usepackage{tikz}
\usepackage{tikz-cd}
\usepackage{tkz-euclide}
\usetikzlibrary{decorations.markings, arrows.meta}
\tikzset{
    marrow/.style={decoration={markings,mark=at position 0.5 with {\arrow{#1}}}, postaction=decorate}
}
\usepackage{url}
\usepackage[normalem]{ulem}
\usepackage{upgreek}
\usepackage{xfrac}
\usetikzlibrary{shapes,backgrounds,decorations.text,patterns,positioning,fit,calc,plotmarks,decorations.markings}
\DeclareMathAlphabet{\mymathbb}{U}{BOONDOX-ds}{m}{n}

\DeclarePairedDelimiter\floor{\lfloor}{\rfloor}
\usepackage[english]{babel}
\usepackage[utf8]{inputenc}
\usepackage[mathscr]{euscript}
\usepackage[all]{xy}
\usepackage{etoolbox}
\patchcmd{\paragraph}{\itshape}{\bfseries\boldmath}{}{}
\newcounter{claimcount}
\newenvironment{claim}{\refstepcounter{claimcount}\par\medskip
   \noindent\textit{Claim \arabic{claimcount}:}}{\vspace{1pt}}
\AtBeginEnvironment{proof}{\setcounter{claimcount}{0}}
\allowdisplaybreaks[1]
\usepackage{thmtools, thm-restate}
\newcommand{\bbmzero}{\mymathbb{0}}
\usepackage{accents}
\usepackage{hyperref}
\hypersetup{
    pdfauthor={Name},
	colorlinks  = true,
	citecolor   = gray,
	linkcolor   = blue}

\newenvironment{myproof}[1]{%
\par\addvspace{12pt plus3pt minus3pt}\global\logotrue%
\noindent{\bf Proof of #1.\hskip.5em}\ignorespaces
}{%
	\par\iflogo\vskip-\lastskip
	\vskip-\baselineskip\prbox\par 
	\addvspace{12pt plus3pt minus3pt}\fi}

\newcommand{\bfzero}{\boldsymbol{0}}
\newcommand{\bfone}{\boldsymbol{1}}
\newcommand{\bftwo}{\boldsymbol{2}}
\newcommand{\bfi}{\boldsymbol{i}}
\newcommand{\bfj}{\boldsymbol{j}}

\newcommand{\bfn}{\boldsymbol{n}}
\newcommand{\bfplus}{\boldsymbol{+}}

\newcommand{\Z}{\mathbb{Z}}
\newcommand{\R}{\mathbb{R}}
\newcommand{\N}{\mathbb{N}}
\newcommand{\Q}{\mathbb{Q}}
\newcommand{\C}{\mathbb{C}}
\newcommand{\F}{\Bbbk}

\newcommand{\bbG}{\mathbb{G}}
\newcommand{\bbH}{\mathbb{H}}
\newcommand{\bbS}{\mathbb{S}}
\newcommand{\bbT}{\mathbb{T}}
\newcommand{\bbU}{\mathbb{U}}
\newcommand{\bbV}{\mathbb{V}}
\newcommand{\bbW}{\mathbb{W}}
\newcommand{\bbX}{\mathbb{X}}
\newcommand{\bbY}{\mathbb{Y}}
\newcommand{\caC}{\mathcal{C}}
\newcommand{\caL}{\mathcal{L}}
\newcommand{\caV}{\mathcal{V}}

\newcommand{\card}{\operatorname{card}}

\newcommand{\im}{\operatorname{Im}}
\newcommand{\dis}{\operatorname{dis}}
\newcommand{\codis}{\operatorname{codis}}

\newcommand{\ppi}{\operatorname{P\Pi}}
\newcommand{\pH}{\operatorname{PH}}
\newcommand{\birth}{\operatorname{birth}}
\newcommand{\death}{\operatorname{death}}
\newcommand{\size}{\operatorname{size}}
\newcommand{\Span}{\operatorname{Span}}
\newcommand{\pt}{\operatorname{pt}}
\newcommand{\set}{\operatorname{\mathbf{Set}}}
\newcommand{\vs}{\operatorname{\mathbf{Vec}}}
\newcommand{\grp}{\operatorname{\mathbf{Grp}}}
\newcommand{\ab}{\operatorname{\mathbf{Ab}}}
\newcommand{\topo}{\operatorname{\mathbf{Top}}}
\newcommand{\Pset}{\operatorname{\mathbf{PSet}}}
\newcommand{\Pgrp}{\operatorname{\mathbf{PGrp}}}
\newcommand{\Pab}{\operatorname{\mathbf{PAb}}}
\newcommand{\Pvec}{\operatorname{\mathbf{PVec}}}
\newcommand{\Ptop}{\operatorname{\mathbf{PTop}}}
\newcommand{\PcaC}{\operatorname{\mathbf{P}\caC}}
\newcommand{\diam}{\operatorname{diam}}
\newcommand{\hyp}{\operatorname{hyp}}
\newcommand{\ecc}{\operatorname{ecc}}
\newcommand{\rad}{\operatorname{rad}}
\newcommand{\Id}{\operatorname{Id}}
\newcommand{\VR}{\operatorname{VR}}
\newcommand{\Part}{\operatorname{Part}}
\newcommand{\SubPart}{\operatorname{SubPart}}

\newcommand{\dgm}{\operatorname{dgm}}

\newcommand{\rmH}{\mathrm{H}}
\newcommand{\rmK}{\mathrm{K}}
\newcommand{\rmq}{\mathrm{q}}
\newcommand{\rms}{\mathrm{s}}

\newcommand{\dgh}{d_\mathrm{GH}}
\newcommand{\dhi}{d_\mathrm{HI}}
\newcommand{\di}{d_\mathrm{I}}
\newcommand{\dhaus}{d_\mathrm{H}}
\newcommand{\db}{d_\mathrm{B}}

\definecolor{darkblue}{rgb}{0.0, 0.0, 0.8}
\definecolor{darkred}{rgb}{0.8, 0.0, 0.0}
\definecolor{darkgreen}{rgb}{0.0, 0.8, 0.0}

\begin{document}

\markboth{Authors' Names}{Persistent Homotopy Groups of Metric Spaces}


\title{PERSISTENT HOMOTOPY GROUPS OF METRIC SPACES
}

\author{FACUNDO M\'EMOLI
}

\address{Department of Mathematics and Department of Computer Science and Engineering,\\ 
The Ohio State University, 
Columbus, Ohio 43210,
United States\\
memoli@math.osu.edu}

\author{LING ZHOU}

\address{Department of Mathematics, The Ohio State University, 
\\ 
Columbus, Ohio 43210,
United States\\
zhou.2568@osu.edu}

\maketitle


\begin{abstract}
We study notions of persistent homotopy groups of compact metric spaces together with their stability properties in the Gromov-Hausdorff sense. We pay particular attention to the case of fundamental groups, for which we obtain a more precise description. 
Under fairly mild assumptions on the spaces, we proved that the classical fundamental group has an underlying tree-like structure (i.e. a dendrogram) and an associated ultra-metric.

\end{abstract}



\tableofcontents

\section{Introduction}

In \cite{gromov1999quantitative}, Gromov motivated the study of quantitative homotopy theory through an anecdote: while attending a conference on cosmology he got different responses from two topologists about whether the universe is simply-connected or not. This discussion inspired Gromov to think that, instead of asking a `yes' or `no' question, it can be more interesting to ask HOW simply-connected a space is. In Figure \ref{fig:sphere with handle}, the space $X$ is given by a $2$-sphere with a tiny handle attached to it, which is clearly not simply-connected. However, noticing how small the handle is compared with the whole space $X$, one may think that $X$ is `almost' simply-connected in a reasonable sense. In this paper, we study a way of quantifying simply-connectedness of a space, by introducing a tree-like structure (i.e. a dendrogram) underlying the fundamental group of the space. For instance, the second figure in Figure \ref{fig:sphere with handle} is the dendrogram over $\pi_1(X)$ that we construct in this paper.

\begin{figure}[ht!]
    \centering
    \includegraphics[scale=.075]{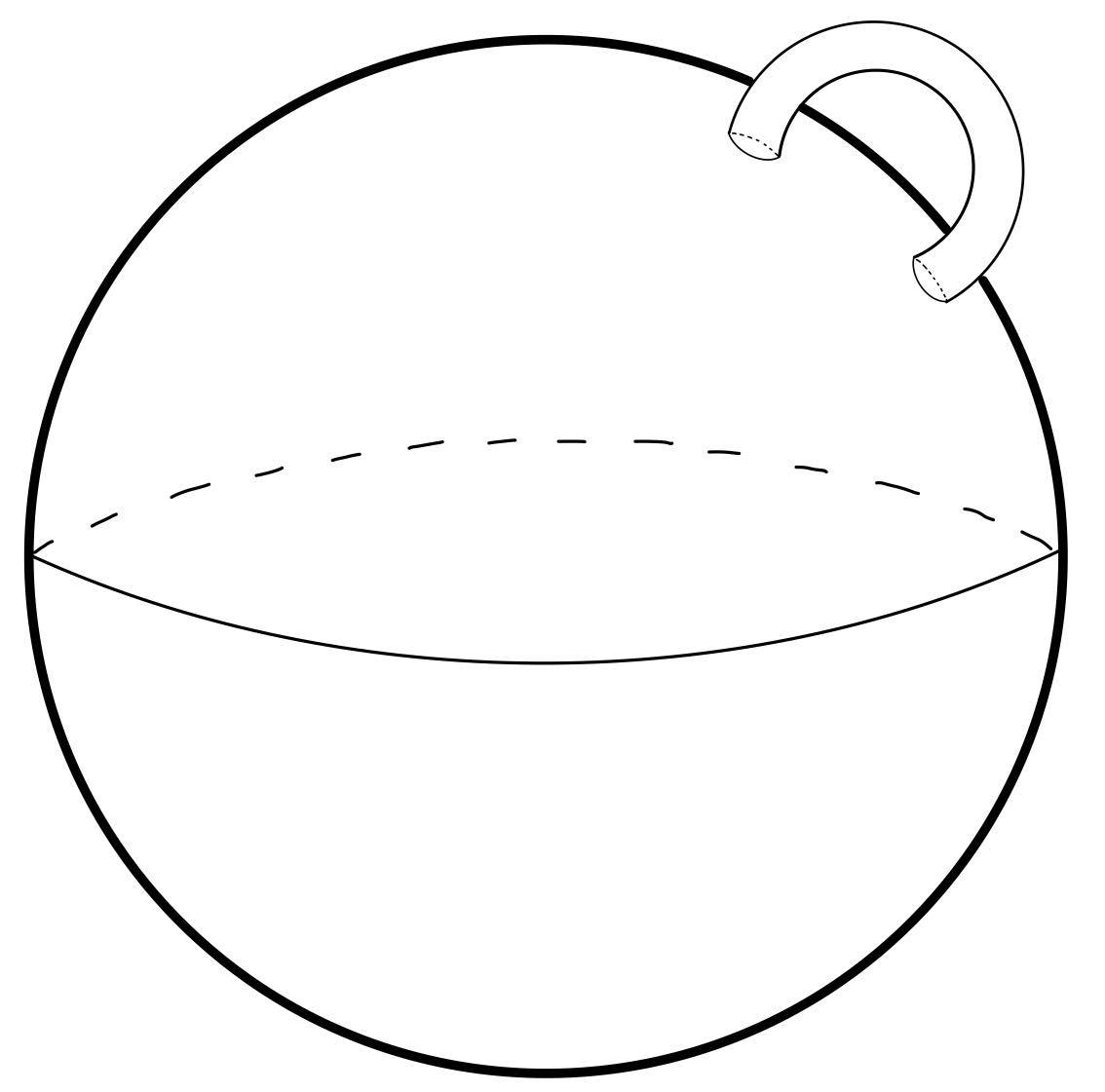}
    \hspace{2em}
    	\begin{tikzpicture}
    \begin{axis} [ 
    height=5cm,
    axis y line=left, 
    axis x line=none,
    xlabel=$\epsilon$,
    ylabel=$\pi_1(X)$,
    ytick={-1,0,1},
    yticklabels={},
    xtick={0,1,1.5},
    xmin=0, xmax=3,
    ymin=-1.8, ymax=1.8,]
    \addplot +[mark=none] coordinates {(.3,0) (.3,0.07)};
    \addplot +[mark=none,color=blue, ultra thick, opacity=0.4] coordinates {(.5,1.7) (.5,-1.7)};
    \addplot +[mark=none,color=blue,ultra thick, opacity=0.4] coordinates {(.3,-0.37) (.3,0.37)};
    \addplot +[mark=none,color=blue,ultra thick, opacity=0.4] coordinates {(.3,0.63) (.3,1.37)};
    \addplot +[mark=none,color=blue,ultra thick, opacity=0.4] coordinates {(.3,-0.63) (.3,-1.37)};
    \addplot[domain=0:.5,color=blue,ultra thick,opacity=0.4]{1};
    \addplot[domain=0:.5,color=blue,ultra thick,opacity=0.4]{-1};
    \addplot[domain=0:.3,color=blue,thick,opacity=0.4]{0.37};
    \addplot[domain=0:.3,color=blue,thick,opacity=0.4]{0.3};
    \addplot[domain=0:.3,color=blue,thick,opacity=0.4]{0.25};
    \addplot[domain=0:.3,color=blue,thick,opacity=0.4]{0.2};    \addplot[domain=0:.3,color=blue,thick,opacity=0.4]{-0.37};
    \addplot[domain=0:.3,color=blue,thick,opacity=0.4]{-0.3};
    \addplot[domain=0:.3,color=blue,thick,opacity=0.4]{-0.25};
    \addplot[domain=0:.3,color=blue,thick,opacity=0.4]{-0.2};
    \addplot[domain=0:.3,color=blue,thick,opacity=0.4]{1.37};
    \addplot[domain=0:.3,color=blue,thick,opacity=0.4]{1.3};
    \addplot[domain=0:.3,color=blue,thick,opacity=0.4]{1.25};
    \addplot[domain=0:.3,color=blue,thick,opacity=0.4]{1.2};    \addplot[domain=0:.3,color=blue,thick,opacity=0.4]{0.63};
    \addplot[domain=0:.3,color=blue,thick,opacity=0.4]{0.7};
    \addplot[domain=0:.3,color=blue,thick,opacity=0.4]{0.75};
    \addplot[domain=0:.3,color=blue,thick,opacity=0.4]{0.8};
    \addplot[domain=0:.3,color=blue,thick,opacity=0.4]{-1.37};
    \addplot[domain=0:.3,color=blue,thick,opacity=0.4]{-1.3};
    \addplot[domain=0:.3,color=blue,thick,opacity=0.4]{-1.25};
    \addplot[domain=0:.3,color=blue,thick,opacity=0.4]{-1.2};    \addplot[domain=0:.3,color=blue,thick,opacity=0.4]{-0.63};
    \addplot[domain=0:.3,color=blue,thick,opacity=0.4]{-0.7};
    \addplot[domain=0:.3,color=blue,thick,opacity=0.4]{-0.75};
    \addplot[domain=0:.3,color=blue,thick,opacity=0.4]{-0.8};
    \addplot[domain=0.15:0.183,color=blue,dotted,thick,opacity=0.4]{0.1};
    \addplot[domain=0.15:0.183,color=blue,dotted,thick,opacity=0.4]{-0.1};
    \addplot[domain=0.15:0.183,color=blue,dotted,thick,opacity=0.4]{0.9};
    \addplot[domain=0.15:0.183,color=blue,dotted,thick,opacity=0.4]{-0.9};
    \addplot[domain=0.15:0.183,color=blue,dotted,thick,opacity=0.4]{-1.1};
    \addplot[domain=0.15:0.183,color=blue,dotted,thick,opacity=0.4]{1.1};
    \addplot[domain=0.23:0.263,color=blue,dotted, ultra thick,opacity=0.4]{1.6};
    \addplot[domain=0.23:0.263,color=blue,dotted,ultra thick,opacity=0.4]{-1.6};
    \addplot[color=blue,ultra thick,opacity=0.4]{0};
    \end{axis} 
    \end{tikzpicture}
    \caption{Left: the space $X$ given by a $2$-sphere with a tiny handle attached to it. Right: dendrogram over the fundamental group $\pi_1(X)$, which is how we quantify the simply-connectedness of spaces in this paper.}
    \label{fig:sphere with handle}
\end{figure}

\subsection{Related work}
We overview related work.

\paragraph*{Quantitative topology and persistence.}

The idea of \emph{topological persistence} has independently appeared in different settings. In a nutshell, via some functorial constructions one assigns to a topological space $X$: (1) for every $\epsilon>0$ some space $S_\epsilon(X)$ and (2) for every $\epsilon'\geq \epsilon$ an equivalence relation $R_{\epsilon'}(X)$ on $S_\epsilon(X).$ For example, when $X$ is a metric space, then $S_\epsilon(X)$ could be chosen to be the collection of $k$-cycles of its Vietoris-Rips complex $\mathrm{VR}_\epsilon(X)$ and $R_{\epsilon'}(X)$ the equivalence relation arising from two cycles $(c,c')\in R_{\epsilon'}(X)$ whenever there exists a $(k+1)$-chain in $R_{\epsilon'}(X)$ such that its boundary is the sum $c+c'$.

Ideas of this type can already be found in the work of Borsuk in the 1940s. For instance, in \textsection 3 of \cite{borsuk1955some} Borsuk studies the transfer of certain scale-dependent topological properties of compacta under a suitable metric.

Frosini and collaborators \cite{frosini1990distance,frosini1992measuring} and Robins \cite{robins1999towards} identified ideas related to concepts nowadays known under the term \emph{persistence} in the context of applications to imaging and physics, respectively. From an algorithmic perspective, the study of topological persistence was initiated by Edelsburnner et al. in \cite{edelsbrunner2000topological}. See the overviews \cite{edelsbrunner2008persistent,carlsson2009topology,ghrist2008barcodes,weinberger2011persistent}, for more information.

Scale-dependent invariants that blend topology with geometry have also been considered by Gromov \cite{gromov1999quantitative,gromov1999metric}.
For example, in the late 1990s Gromov suggested the study of homotopy invariants in a scale-dependent manner. One question posed by Gromov was: given a $\lambda$-Lipschitz and contractible map $f:X\rightarrow Y$ between metric spaces $X$ and $Y$, determine the function $\lambda\mapsto\Lambda(\lambda)$ such that there exists a null-homotopy of $f$ given by $\Lambda(\lambda)$-Lipschitz maps. This question can be interpreted as a homotopical version of the homological construction described above.

\paragraph*{(Discrete) homotopy groups and persistence.}
In the late 1990s Frosini and Mulazzani \cite{frosini1999size} considered a construction that combines ideas related to persistence with homotopy groups. Their construction dealt with pairs $(M,\varphi)$ where $M$ is a closed manifold and $\varphi = (\varphi_1,\ldots,\varphi_k):M\rightarrow \mathbb{R}^k$ is a continuous function. In this setting, for $\zeta,\eta \in \mathbb{R}^k$ with $\zeta_i\leq \eta_i$ for all $i$ they considered a certain $(\varphi(\zeta),\eta)$-parametrized version of the fundamental group arising from considering equivalence classes of pointed loops in sublevel sets $M_\zeta:=\{x\in M|\varphi_i(x)\leq \zeta_i,\,\forall i\}$ under the equivalence relation: for $\alpha$ and $\beta$ pointed loops in $M_\zeta,$ $\alpha\sim_\eta \beta$ iff there is a (pointed) homotopy whose tracks remain in $M_\eta$. These groups are then used to obtain lower bounds for the so-called natural pseudo distance between two manifolds. A more comprehensive look at several related ideas is carried out by Frosini in \cite{frosini1999metric}.

In \cite{Letscher:2012:PHK:2090236.2090270} in his study of knots, Letscher considered the scenario in which one is given a filtration $\{X_\epsilon\}_{\epsilon>0}$ of a path connected topological space $X$. For each $\epsilon,\delta>0$, he considered the group $\pi_1^{(\delta)}(X_\epsilon)$ consisting of equivalence classes of (based) loops contained in $X_\epsilon$ under the equivalence relation stating that two loops in $X_\epsilon$ are equivalent iff there is a homotopy between them which is fully contained in $X_{\epsilon+\delta}$. 

In \cite{barcelo2014discrete} Barcelo et al. introduced the notion of \emph{discrete} homotopy groups associated to metric spaces at each given scale $\epsilon$, via the study of $1$-Lipschitz maps from $n$-cubes into the metric space in question. Barcelo et al. explicitly studied the behavior of their invariants under scale coarsening.

In his dissertation \cite{wilkins2011discrete}, Wilkins studies a certain notion of the discrete fundamental group, 
which differs from the one introduced in \cite{barcelo2014discrete} in that the respective notions of homotopies between discrete loops are different. In \cite{conant2012discrete} Conant et al. also study the notion of (homotopy) critical spectrum, which is the set of parameter values at which the discrete fundamental groups change up to isomorphism.

The notion studied by Wilkins originated in work by Berestovskii and Plaut \cite{berestovskii2001covering} in the context of topological groups and subsequently extended in \cite{berestovskii2007uniform} to the setting of uniform spaces. In \cite{berestovskii2007uniform}, the authors also consider discrete fundamental groups as an inverse system and study its inverse limit. 

In \textsection 8.1 of \cite{blumberg2017universality} Blumberg and Lesnick define persistent homotopy groups of $\R$-spaces (in a way somewhat different from Letscher's) and formulate certain persistent Whitehead conjectures. In \cite{batan2019persistent} Batan, Pamuk, and Varli study a certain persistent version of Van-Kampen's theorem, excision, and Hurewicz theorem for Letscher's definition of persistent homotopy groups of $\R$-spaces.

In \cite{Jardine2019DataAH} Jardine utilizes poset theoretical ideas to study Vietoris-Rips complexes (and some variants) associated with metric spaces and provides explicit conditions for the poset morphism induced by an inclusion of metric spaces to be quantifiably close to a homotopy equivalence. Jardine also considers persistent fundamental groupoids and higher-dimensional persistent homotopy in the lecture notes \cite{Jardine2020Persistent}. 

In \cite{virk20201}, Virk studies a notion of critical points of the persistent fundamental group and the degree one persistent homology group of Vietoris-Rips or \v{C}ech filtration of compact geodesic spaces. In particular, he shows that a critical point $c$ associated with the Vietoris-Rips filtration corresponds to an isometrically embedded circle of length $3c$.

In \cite{rieser2021vcech} Rieser studies a certain version of discrete homotopy groups of \v{C}ech closure spaces. A \v{C}ech closure space is a set equipped with a \v{C}ech closure operator, which is a concept similar to that of a topological closure operator except that it is not required to be idempotent. Given a finite metric space $X$, a way of obtaining \v{C}ech closure operators is the following: for each $r\geq 0$ let $c_r$ be a map between subsets of $X$ such that $c_r(A)=\{x\in X:d_X(x,A)\leq r\}$ for each $A\subset X$. The author shows that each $c_r$ is a \v{C}ech closure operator and the identity map $(X,c_q)\to (X,c_r)$ induces maps $\iota_n^{q,r}:\pi_n(X,c_q)\to \pi_n(X,c_r)$ for every $q\leq r$ and $n\geq 0$. Then, he defines the persistent homotopy groups of $X$ to be the collection  $\{\im(\iota_n^{q,r})\}_{r\geq q \geq 0}$. Along this line, in \cite{bubenik2021homotopy} Bubenik and Milićević develop a uniform framework using Čech closure spaces to encompass the discrete homotopy of metric spaces, the homotopy of topological spaces, and the homotopy of (directed) clique complexes.

\subsection{Overview of our results}

We consider three different definitions of the persistent fundamental group. Persistent fundamental groups are examples of persistent groups, which are functors from a poset category (e.g., $(\R,\leq)$) to the category $\grp$ of groups. 
One of the three constructions, denoted by $\ppi_1^{\rmK,\bullet}$ arises from isometrically embedding a given metric space $X$, via the so-called Kuratowski embedding, into $L^\infty(X)$ and then applying the $\pi_1$ functor to successive \emph{thickenings} $X^\epsilon$ of $X$ inside $L^\infty(X)$, cf. Definition \ref{def:ppi^K}.\footnote{See \cite{adamaszek2018metric} and \cite{adams2021persistent} for optimal transport notions of thickening that have similar properties.} This definition is similar in spirit to the construction used by Gromov for defining the so-called \emph{filling radius invariant} \cite{gromov1983filling}. Another definition, denoted $\ppi_1^{\VR,\bullet}$, stems from simply considering the geometric realization of nested Vietoris-Rips simplicial complexes for different choices of the scale parameter and then applying the $\pi_1$ functor, cf. Definition \ref{def:ppi^VR}. The third, more combinatorial, definition which we consider is one coming from the work of Plaut and Wilkins \cite{plaut2013discrete}, see Definition \ref{def:persistent f. group}. 
This construction, denoted simply $\ppi_1^\bullet$, turns out to be isomorphic to $\ppi_1^{\VR,\bullet}$ and to $\ppi_1^{\rmK,\bullet}$ up to a scaling
(see Theorem \ref{thm:iso-ppi_1's}).

These three definitions can all be constructed in two versions: the open version and the closed version, which will be specified with the superscripts $<$ and $\leq $, respectively. Throughout this paper, we mainly work on the closed version and will omit the superscripts unless necessary. Now we observe that these three definitions give rise to isomorphic persistent fundamental groups.

\begin{restatable}[Isomorphisms of persistent fundamental groups]{theorem}{isoofPPI}\label{thm:iso-ppi_1's} Given a pointed compact metric space $(X,x_0)$, we have
\begin{equation} \label{eq:iso of p.f.g}
    \ppi_1^{\rmK,\bullet}(X,x_0) \cong\ppi_1^{\VR,2\bullet}(X,x_0)\cong  \ppi_1^{2\bullet}(X,x_0),
\end{equation}
where the rightmost isomorphism is true in either the closed or the open version, and the leftmost isomorphism is true for the open version and the closed version when $X$ is finite. In either version, $\ppi_1^{\rmK,\bullet}(X,x_0) $ and $\ppi_1^{\VR,2\bullet}(X,x_0)$ are at zero homotopy-interleaving distance.
\end{restatable}

For a given compact metric space $X$ and a base point $x_0\in X$, the persistent fundamental group is defined as 
$$\ppi_1(X,x_0) = \left\{\pi_1^\epsilon(X,x_0)\stackrel{\Phi_{\epsilon,\epsilon'}}{\longrightarrow}\pi_1^{\epsilon'}(X,x_0)\right\}_{\epsilon'\geq \epsilon\geq 0},$$
where $\pi_1^0(X,x_0):=\pi_1(X,x_0)$ and $\pi_1^\epsilon(X,x_0)$ is the discrete fundamental group at scale $\epsilon$ for $\epsilon>0$ (see Definition \ref{def:discrete f. group}).
Under fairly mild assumptions on $X$, the structure maps $\Phi_{\epsilon,\epsilon'}$ 
are surjective and for sufficiently small $\epsilon$ the discrete fundamental group $\pi_1^{\epsilon}(X)$ is isomorphic to $\pi_1(X)$ via $\Phi_\epsilon:=\Phi_{0,\epsilon}$. When combined, these two factors suggest that as the scale parameter increases homotopy classes in $X$ can only become equal (no new classes can appear), thus suggesting the existence of a tree-like structure (i.e., a dendrogram) underlying $\ppi_1(X)$, by viewing each
$\pi_1^{\epsilon}(X)\cong \pi_1(X)/\ker(\Phi_\epsilon)$ 
as a partition of $\pi_1(X)$. In summary, we prove:

\begin{restatable}[Dendrogram over $\pi_1(X)$]{theorem}{dendrogram}\label{thm:dendrogram} Let a compact geodesic space $X$ be uniformly locally path connected (u.l.p.c.) and semi-locally simply connected (s.l.s.c.). Then, the following map defines a dendrogram $\theta_{\pi_1(X)}$ over $\pi_1(X)$:
\begin{equation*}
\theta_{\pi_1(X)}(\epsilon):=\begin{cases} \pi_1^{\epsilon}(X), & \mbox{if $\epsilon>0$}\\
\pi_1(X), & \mbox{if $\epsilon=0$,}
\end{cases}
\end{equation*}
where we view $\pi_1(X)$ as a set and each $\pi_1^{\epsilon}(X)$ as a partition of $\pi_1(X)$.
\end{restatable}
\begin{figure}[ht!]
\centering
	\begin{tikzpicture}
    \begin{axis} [ 
    height=5.5cm,
    axis y line=left, 
   axis x line=none,
    xlabel=$\epsilon$,
    ylabel=$\pi_1(\mathbb{S}^1)$,
    ytick={-3,-2,...,3},
    yticklabels={-3,-2,-1,0,1,2,3},
    xtick={0,1},
    xticklabels={0,$\tfrac{2\pi}{3}$,},
    xmin=0, xmax=2,
    ymin=-3.7, ymax=3.7,]
    \addplot +[mark=none,color=blue,ultra thick, opacity=0.4] coordinates {(1,3.5) (1,-3.5)};
    \addplot +[mark=none,color=blue,densely dotted, ultra thick, opacity=0.4] coordinates {(0.5,3.3) (0.5,3.7)};
    \addplot +[mark=none,color=blue,densely dotted, ultra thick, opacity=0.4] coordinates {(0.5,-3.3) (0.5,-3.7)};
    \addplot[domain=0:1,color=blue,ultra thick,opacity=0.4]{1};
    \addplot[domain=0:1,color=blue,ultra thick,opacity=0.4]{-1};
    \addplot[domain=0:1,color=blue,ultra thick,opacity=0.4]{2};
    \addplot[domain=0:1,color=blue,ultra thick,opacity=0.4]{-2};
    \addplot[domain=0:1,color=blue,ultra thick,opacity=0.4]{3};
    \addplot[domain=0:1,color=blue,ultra thick,opacity=0.4]{-3};
    \addplot[color=blue,ultra thick,opacity=0.4]{0};
    \end{axis} 
    \node[below right=0.4pt of {(2.5,1.9)}, outer sep=1.5pt,fill=white] {$\tfrac{2\pi}{3}$};
    \end{tikzpicture}
    \caption{The dendrogram associated to $\ppi_1(\bbS^1)$. The $y$-axis represents elements of $\pi_1(\mathbb{S}^1)\cong \Z$, i.e., homotopy classes of continuous loops in $\mathbb{S}^1$ (each corresponds to an integer).} \label{fig:dendrogram-pi1}
\end{figure}
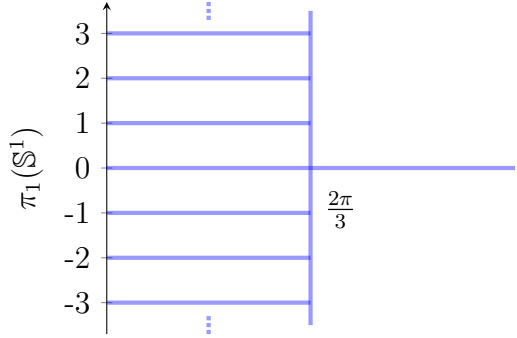 

For instance, when $X$ is $\bbS^1$, the geodesic unit circle, associated to $\ppi_1(\bbS^1)$ we have a \emph{dendrogram} over $\pi_1(\bbS^1)\cong \Z$ shown in Figure \ref{fig:dendrogram-pi1}.

Dendrograms over a given set can be bijectively associated with an ultra-metric on the set. In this sense, one can view the map $X\mapsto \theta_{\pi_1(X)}$ described in the theorem above as a map from Gromov-Hausdorff space into itself. We prove that this endomorphism is $2$-Lipschitz.

\begin{restatable} [$\dgh$-stability of $\theta_{\pi_1(\bullet)}$]{theorem}{GHstabofpi}\label{thm:stab-pi_1} 
Given compact geodesic s.l.s.c. spaces $X$ and $Y$, we have
$$\dgh\left( \left( \pi_1(X),\mu_{\theta_{\pi_1(X)}}\right),\left( \pi_1(Y),\mu_{\theta_{\pi_1(Y)}}\right)\right)\leq 2\cdot \dgh(X,Y).$$
\end{restatable}

In Proposition \ref{prop:cr=im}, we prove that the \emph{critical spectrum} $\mathrm{Cr}(X)$ of a compact geodesic s.l.s.c. space $X$ \cite{wilkins2011discrete,conant2012discrete} coincides with the image of the ultra-metric $\mu_{\theta_{\pi_1(X)}}$ on $\pi_1(X)$. In other words, the dendrogram $\mu_{\theta_{\pi_1(X)}}$ captures the critical spectrum as the parameter values where it changes. 
Combining this result with Theorem \ref{thm:stab-pi_1}, we obtain the stability of $X\mapsto \mathrm{Cr}(X)$: i.e. that the Hausdorff distance between critical spectra of two metric spaces is upper bounded by the Gromov-Hausdorff distance between the two spaces; see Proposition \ref{prop:stab-cr}.

\medskip

Of the three definitions of persistent fundamental groups given previously, $\ppi_1^{\rmK}$ and $\ppi_1^{\VR}$ can be generalized to high-dimensional by replacing $\pi_1$ with the $n$-th homotopy group $\pi_n$ in their definitions. As for the case of $n=1$ described in Theorem \ref{thm:iso-ppi_1's}, we still have that $\ppi_n^{\rmK,\bullet}\cong \ppi_n^{\VR,2\bullet}$, as well as the isomorphism of persistent homology groups $\pH_n^{\rmK,\bullet}\cong \pH_n^{\VR,2\bullet}$ (cf. Corollary \ref{cor:iso-ppi_n}). We furthermore have that, under the interleaving distance (cf. Definition \ref{def:interleaving}), the map from compact metric spaces into persistent groups $X\mapsto \ppi_n^\rmK(X)$ is $1$-Lipschitz.

\begin{restatable}
[$\di$-stability of $\ppi_n^{\rmK}(\bullet)$ and $\ppi_n^{\VR}(\bullet)$]{theorem}{stabofppiKn}
\label{thm:stab-ppi_n} 
Let $(X,x_0)$ and $(Y,y_0)$ be pointed compact metric spaces. Then, for each $n\in\Z_{\geq 1}$,
$$\di\left(\ppi_n^{\rmK}(X,x_0),\ppi_n^{\rmK}(Y,y_0)\right)\leq  \dgh^{\pt}((X,x_0),(Y,y_0)).$$
If $X$ and $Y$ are connected, then
$\di\left(\ppi_n^{\rmK}(X),\ppi_n^{\rmK}(Y)\right)\leq  \dgh(X,Y).$
Via the isomorphism $\ppi_n^{\rmK,\bullet}\cong \ppi_n^{\VR,2\bullet}$, analogous inequalities also hold for $\ppi_n^{\VR}(\bullet)$, up to a factor $\tfrac{1}{2}$.
\end{restatable}

We relate persistent fundamental groups to the persistent homology groups via a suitable version of the Hurewicz Theorem.

\begin{restatable} [Persistent Hurewicz theorem]{theorem}{persistentHurewicz}
\label{thm:hurewicz-persistent} Let $X$ be a connected metric space. Then, there exists a natural transformation 
$\ppi_1^{\rmK}(X)\xRightarrow{\rho} \pH_1^{\rmK}(X),$
such that for each $\epsilon>0$, $\rho_{\epsilon}$ is surjective and $\ker(\rho_{\epsilon})$ is the commutator group of $\ppi_1^{\rmK,\epsilon}(X).$
\end{restatable}

Like the static case where homotopy groups are stronger invariants than homology, persistent homotopy groups are more informative than persistent homology. For the case of dimension $1$, see Theorem \ref{thm:stab-di-ppi_1} and Example \ref{ex:torus-S1S1S2}, where in the example we compare the torus $\bbS^1\times\bbS^1$ with the wedge sum $\bbS^1\vee \bbS^1\vee\bbS^2$. It is well-known that these two spaces have the same homology groups in all dimensions, but their fundamental groups are different. In the persistent setting, we see that the persistent fundamental group provides better estimates of the Gromov-Hausdorff distance between these two spaces than persistent homology over all dimensions. For higher dimensions, we consider the pair $\bbS^m\times\bbS^m$ and $\bbS^m\vee \bbS^m\vee\bbS^{2m}$ in Example \ref{ex:rational_homotopy}, and see that the persistent ration homotopy group in dimension $3m-1$ tells these two spaces apart, while persistent homology fails to.

\subsection{Organization of the paper}

In \S \ref{sec:notation-key} we provide a notation key to facilitate reading the paper.

In \S\ref{sec:background} we review elements from category theory, metric geometry, and topology. In particular, \S\ref{subsec:persistent theory} provides the necessary background about persistent homology, the interleaving distance, and related concepts.

In \S\ref{sec:p-set}, we review $0$-dimensional persistent homotopy, which 
due to its natural `tree structure', serves as a precursor to our construction of dendrograms for the persistent fundamental group. In \S\ref{subsec:persistent fundamental group}, we review discrete fundamental groups from 
\cite{plaut2013discrete,barcelo2014discrete}, and use it to define a notion of persistent fundamental group. 
In \S\ref{subsec:persistent n-homotopy}, we study the relationship between different definitions of persistent fundamental groups and discuss higher-dimensional persistent homotopy groups. 
In \S\ref{sec:rational homotopy groups}, to circumvent the difficulty in computing the high-dimensional homotopy groups of spheres that appear in the calculation of $\ppi_n^{\rmK}(\bbS^1)$, we study persistent rational homotopy groups as a weaker but computable\footnote{Here the word `computable' is used in a general way. For example, higher-dimensional persistent homotopy groups of $\bbS^1$ are not completely known, but its rationalization can be obtained in all dimensions, see page \pageref{cor:ppi_n rational}. Thus, we say the latter one is computable.} alternative. In \S \ref{sec:properties}, we consider persistent homotopy groups under point-wise products and wedge sums and prove a persistent version of the Hurewicz theorem, cf. Theorem \ref{thm:hurewicz-persistent}.

In \S \ref{subsec:discretization}, we show that the limit of discrete fundamental groups is isomorphic to the classical fundamental group for tame enough spaces (cf. Theorem \ref{thm:discretization}). Moreover, in \S \ref{sec:treegram-pseudo-metric} we construct a stable (under the Gromov-Hausdorff distance of metric spaces) pseudo-metric $\mu_X^{(1)}$ on the set of $\caL(X,x_0)$ of discrete loops in a metric space $X$, and identify 
a \emph{treegram} structure over $\caL(X,x_0)$.
In \S \ref{subsec:dendrogram}, we identify some conditions on $X$ such that the treegram structure over $\caL(X,x_0)$ induces a dendrogram over the fundamental group $\pi_1(X)$ and prove Theorem \ref{thm:dendrogram}. In addition, we show that the ultra-metric on $\pi_1(X)$ associated with the dendrogram $\theta_{\pi_1(X)}$ is stable when the underlying spaces are homotopy equivalent, cf. Theorem \ref{thm:stab-l^infty-mu}, together with some geometric applications. 

In \S \ref{sec:stability-ppi_n} we address the question of stability of persistent homotopy groups and establish Theorems \ref{thm:stab-pi_1} and \ref{thm:stab-ppi_n}. We exhibit pairs of filtrations confounded by persistent homology but distinguished by their persistent (rational) homotopy groups.
In \textsection \ref{sec:finiteness}, we establish a notion of finiteness theorem for fundamental groups of locally geometrically connected spaces. 
In \S \ref{sec:stability-second-proof}, an alternative proof, more constructive and independent of the notion of Vietoris-Rips or Kuratowski filtrations, of the case $n=1$ in Theorem \ref{thm:stab-ppi_n} is given.

This paper focuses primarily on geodesic spaces that satisfy some mild conditions. Readers interested in the case of persistent fundamental groups of finite metric spaces will find the appendix of the preprint of this paper to be helpful, see \cite{memoli2019persistent}. 

\section{Notation} 

\label{sec:notation-key}
{\small  \begin{spacing}{1.5} 
 \begin{tabbing} 

 \= Simbolo \= Separador \= Significado \+ \kill 
 {\bf Symbol} \> \> {\bf Meaning} \\ \\
 
  $\F$ \> \> A field. \\
  
$\caC$\>\> A category. \\

$\card(A)$\>\> Cardinality of a set $A$. \\

$\db$\>\> Bottleneck distance between multisets of points in $\Bar{\R}^2$, page \pageref{para:db}.\\

$\dhaus^X$ \> \> Hausdorff distance between subsets of the metric space $X$, page \pageref{para:dh}.\\

$d_{\mathrm{GH}}^{(\pt)}$ \> \> (Pointed) Gromov-Hausdorff distance, 
 page \pageref{para:dgh} (\pageref{para:pt dgh}).\\

$\di$\>\> Interleaving distance between functors, Definition \ref{def:interleaving}.\\
 
 $\dhi$\>\> Homotopy interleaving distance between functors, Definition \ref{def:HI}.\\

$\PcaC^{(\R_+,\leq)}$\>\> The category of functors from $(\R_+,\leq)$ to $\caC$, page \pageref{para:category PC}.\\

$\simeq$\>\> Weak homotopy equivalence, 
page \pageref{para:weak h.e.}.\\

$\cong$\>\> 
Homotopy equivalence of spaces; isomorphism in $\PcaC^{(\R_+,\leq)}$ or $\grp$. \\

$\VR_{\bullet}(X)$\>\> Vietoris–Rips filtration of a metric space $X$, page \pageref{para:VR}.\\


$\pH_k^{\VR}(X)$\>\> The $k$-th persistent homology of $X$, page \pageref{para:PH}.\\



$\mu_{\theta_{A}}$\>\> ultra-metric induced by a dendrogram $\theta_A$ on a set $A$, page \pageref{para:metric by dendrogram}.\\


$\mu_{\theta_{A}}^{\rms}$\>\> Pseudo-ultra-metric induced by a treegram $\theta_A^{\rms}$ on a set $A$, page \pageref{def:treegram}.\\



$\ppi_0(X)$\>\> The persistent set induced by a metric space $X$ on page \pageref{para:ppi_0}.\\



$[n]$\>\> The set $\{0,\cdots,n-1\}$ for a non-negative set $n$, page \pageref{para:n set}.\\



$\caL(X,x_0)$\>\> The set of discrete loops on $(X,x_0)$, Proposition \ref{prop:equivalence}.\\

$\caL^\epsilon(X,x_0)$\>\> The set of $\epsilon$-loops on $(X,x_0)$, page \pageref{para:epsilon loop}.\\


$\pi_1^\epsilon(X,x_0)$\>\> Discrete fundamental group, also as $\ppi_1^{\epsilon}(X,x_0)$, Definition \ref{def:discrete f. group}.\\





$\ppi_1(X,x_0)$\>\> Persistent fundamental group of $(X,x_0)$, Definition \ref{def:persistent f. group}.\\



$X^\epsilon$\>\> The $\epsilon$-thickening of a compact metric space $X$, Definition \ref{def:epsilon thickening}.\\

$\ppi_n^\rmK(X,x_0)$\>\> The $n$-th persistent $\rmK$-homotopy group, Definition \ref{def:ppi^K}.\\

$\pH_n^\rmK(X)$\>\> The $n$-th persistent $\rmK$-homology group, Remark \ref{rmk:persistent-K-homology}.\\

$\ppi_n^{\VR}(X,x_0)$\>\> The $n$-th persistent $\VR$-homotopy group, Definition \ref{def:ppi^VR}.\\

 \end{tabbing} 
 \end{spacing}
}

\section{Background} \label{sec:background}

In this section, we recall some basic notions from geometry, topology, persistent theory, dendrograms, and so on, which will be used in later sections.

\subsection{Metric spaces and general topology} Given a set $X$, an (extended) \emph{pseudo-metric} $d_X$ on $X$ is a function $d:X\times X\to [0,+\infty]$ such that for any $x,y,z\in X$, the following holds: 
	\begin{itemize}
		\item $d_X(x,x)=0$;
		\item $d_X(x,y)=d_X(y,x)$;
		\item $d_X(x,z)\leq d_X(x,y)+d_X(y,z)$.
	\end{itemize}
In this paper, the term pseudo-metric will be abused so that when the first condition is not satisfied, we still call it a pseudo-metric. A \emph{metric} $d_X$ on $X$ is a pseudo-metric such that $d_X(x,y)=0$ if and only if $x=y$. A \emph{metric space} is a pair $(X,d)$ where $X$ is a set and $d_X$ is a metric on $X$. A \emph{(pseudo) ultra-metric} $d_X$ on $X$ is a (pseudo) metric satisfying the strong version of the triangle inequality:
	$$d_X(x,z)\leq \max\left\{d_X(x,y),d_X(y,z)\right\},\forall x,y,z\in X.$$
	
\begin{definition}[Length and geodesics]
In a metric space $(X,d_X)$, the \emph{length} of a curve $\gamma:[0,1]\to X$ is defined as $L(\gamma):=\sup \sum_{i=1}^kd_X(\gamma(t_{i-1}),\gamma(t_i))$, where the supremum is taken over all finite sequences $0=t_0\leq t_1\leq \dots\leq t_k=1$.

A curve $\gamma:[0,1]\to X$ is said to be \emph{geodesic} if $L(\gamma([t,t']))=d_X(\gamma(t),\gamma(t'))$ for all $t,t'\in [0,1]$. The space $(X,d_X)$ is said to be \emph{geodesic} if any two points in $X$ is connected by a geodesic.
\end{definition}

\begin{definition}[$\delta$-hyperbolic spaces] 
\label{def:hyp}
A metric space $(X,d_X)$ is said to be \emph{$\delta$-hyperbolic} for some $\delta\geq 0$, if for any $x_1,x_2,x_3,x_4\in X,$ 
$$d_X(x_1,x_3)+d_X(x_2,x_4)\leq \max\{d_X(x_1,x_2)+d_X(x_3,x_4),d_X(x_3,x_2)+d_X(x_1,x_4)\}+\delta.$$

The largest $\delta$ such that $X$ is $\delta$-hyperbolic is called the \emph{hyperbolicity of $X$}, denoted by $\hyp(X)$.
A $0$-hyperbolic space is also called a \emph{tree-like metric space}, in which case the above formula (with $\delta=0$) is called the \emph{four-point condition}.
\end{definition}	

Let $(X,d_X)$ be a metric space. For $A\subset X$, we define 
\begin{itemize}
    \item the \emph{separation} of $A$ to be $\operatorname{sep}(A):=\inf_{x\neq y}d_X(x,y);$ \label{para:sep}
    \item the \emph{diameter} of $A$ to be $\diam(A):=\sup\{d_X(x,y):x,y\in A\}.$
\end{itemize}

For a compact metric space $(X,d_X)$, we denote by $B(x,r)$ the open ball of radius $r$ centered at $x\in X$,\label{para:open-ball} and denote by $A^r$ the \emph{$r$-neighborhood} of a set $A$ in $X$, i.e., 
$A^r:=\bigcup_{x\in A} B(x,r).$

As a topological space with the metric topology, $X$ is said to be \emph{locally path-connected (l.p.c.)} if, for each $\epsilon>0$ and $x\in X$, there exists an open path-connected neighborhood $V\subset B(x,\epsilon)$ of $x$.
The space $X$ is said to be \emph{uniformly locally path connected} (u.l.p.c.) if for each $\epsilon>0$ there exists $\delta>0$ such that for every $x\in X$ any two points in $B(x,\delta)$ are connected by a path $\gamma$ with the image completely contained in $B(x,\epsilon)$. 
The space $X$ is \emph{uniformly semi-locally simply connected} (u.s.l.s.c.) if there exists $\epsilon>0$ small enough so that for every $x\in X$, a loop with image contained in $B(x,\epsilon)$ is null-homotopic in $X$ (but is not necessarily null-homotopic in $B(x,\epsilon)$). We also recall the following notion from \cite{peter1990finiteness}, which describes a condition stronger than being u.s.l.s.c.
\begin{definition}[Locally geometrically $1$-connected spaces] 
\label{def:lgc}
Let $R\geq 0$ and $\rho:[0,R]\to [0,\infty)$ be a function (not necessarily continuous) such that $\rho(\epsilon)\geq \epsilon$ for all $\epsilon$ and $\rho(\epsilon)\to 0$ as $\epsilon\to 0$. The space $X$ is said to be \emph{locally geometrically $1$-connected of size $\rho$}, if for every $x\in X$ and $r\in [0,R]$ the ball $B(x,r)$ is both $0$-connected and $1$-connected inside $B(x,\rho(r))$. In short, we write $X$ is $\operatorname{LGC_1}(\rho,R).$
\end{definition}

\subsection{Gromov-Hausdorff distance} 

We recall the definition of the Gromov-Hausdorff distance from \textsection 7.3.2 of \cite{burago2001course}. Let $Z_1$ and $Z_2$ be subspaces of a metric space $(X,d)$. The \emph{Hausdorff distance} between $Z_1$ and $Z_2$ is \label{para:dh}
\[d_\mathrm{H}^X(Z_1,Z_2):=\inf\left\{r>0: Z_1\subset Z_2^r\text{ and }  Z_2\subset Z_1^r\right \}.\]
\begin{definition}\label{para:dgh}
The \emph{Gromov-Hausdorff distance} between metric spaces $(X,d_X)$ and $(Y,d_Y)$ is the infimum of $r>0$ for which there exist a metric spaces $Z$ and two distance preserving maps $\psi_X:X\to Z$ and $\psi_Y:Y\to Z$ such that $\dhaus^Z(\psi_X(X),\psi_Y(Y))<r$, i.e., 
\begin{equation}\label{eq:dgh}
    \dgh(X,Y):=\inf_{Z,\psi_X,\psi_Y}\dhaus^Z(\psi_X(X),\psi_Y(Y)).
\end{equation}
\end{definition}

For metric spaces $(X,d_X)$ and $(Y,d_Y)$, the \emph{distortion} of a map $\varphi:X\rightarrow Y$ is 
$$\dis(\varphi):=\sup_{x,x'\in X}|d_X(x,x')-d_Y(\varphi(x),\varphi(x'))|.$$
For maps $\varphi:X\rightarrow Y$ and $\psi:Y\rightarrow X$, their \emph{co-distortion} is defined to be $$\codis(\varphi,\psi):=\sup_{x\in X,y\in Y}|d_X(x,\psi(y))-d_Y(\varphi(x),y)|.$$
It follows from Theorem 2.1 of \cite{kalton1999distances} that 
\begin{equation}\label{eq:dgh-maps}
    \dgh(X,Y)=\inf_{\substack{\varphi:X\rightarrow Y\\ \psi :Y\rightarrow X}}\tfrac{1}{2}\max\{\dis(\varphi),\dis(\psi),\codis(\varphi,\psi)\}.
\end{equation}

A \emph{tripod} between two sets $X$ and $Y$ is a pair of surjections from another set $Z$ to $X$ and $Y$, respectively. We will express this by the diagram \label{para:tripod} $$R:X\xtwoheadleftarrow{\phi_X}Z\xtwoheadrightarrow{\phi_Y}Y.$$ 
For $x\in X$ and $y\in Y$, by $(x,y)\in R$ we mean there exists $z\in Z$ such that $\phi_X(z)=x$ and $\phi_Y(z)=y$. When $(X,d_X)$ and $(Y,d_Y)$ are pseudo-metric spaces, the \emph{distortion} of a tripod $R$ between $X$ and $Y$ is defined to be:	\label{para:dis of tripod}
$$\dis(R):=\sup_{z,z'\in Z}\left|  d_X(\phi_X(z),\phi_X(z'))-d_Y(\phi_Y(z),\phi_Y(z'))\right| .$$
A \emph{correspondence} is a subset $R'$ of $X\times Y$ such that for any $x\in X$ there exists at least one $y\in Y$ such that $(x,y)\in R'$ and for any $y\in Y$ there exists at least one $x\in X$ such that $(x,y)\in R'$. Notice that a tripod induces a correspondence and vice versa.
Let $\mathfrak{R}(X,Y)$ denote the collection of all tripods between $X$ and $Y$. 

\begin{theorem}[{Theorem 7.3.25, \cite{burago2001course}}] Given bounded metric spaces $X$ and $Y$, 
\begin{equation}\label{eq:dgh-tripod}
    \dgh(X,Y)=\tfrac{1}{2}\inf_{R\in \mathfrak{R}(X,Y)}\dis(R).
\end{equation}
\end{theorem}


\begin{remark} In this theorem, the formula on the right-hand side applies to pseudo-metric spaces (see \cite{chowdhury2017distances,chowdhury2018metric}) as a generalization of the Gromov-Hausdorff distance. For this reason, we still use the symbol $\dgh$ to denote this generalized distance. 
\end{remark}

\begin{proposition}[p. 255, \cite{burago2001course}]\label{prop:property-GH} For bounded metric spaces $(X,d_X)$ and $(Y,d_Y)$,
 $$\tfrac{1}{2}|\diam(X)-\diam(Y)|\leq \dgh(X,Y)\leq \tfrac{1}{2}\max\{\diam(X),\diam(Y)\}.$$ 
 In particular, if $Y$ is the one-point metric space $*$, then $\dgh(X,\ast)= \tfrac{1}{2}\diam(X).$
\end{proposition}

A \emph{pointed metric space} $(X,x_0,d_X)$ is a metric space $(X,d_X)$ together with a distinguished basepoint $x_0\in X$. For the sake of simplicity, we will often omit the distance function $d_X$ and denote a pointed metric space by $(X,x_0)$. Given basepoints $x_0\in X$ and $y_0\in Y$, a \emph{pointed tripod} is a tripod $R:X\xtwoheadleftarrow{\phi_X}Z\xtwoheadrightarrow{\phi_Y}Y$ such that $\phi_X^{-1}(x_0)\cap\phi_Y^{-1}(y_0)$ is non-empty. Let $\mathfrak{R}^{\pt}((X,x_0),(Y,y_0))$ denote the collection of pointed tripods between $(X,x_0)$ and $(Y,y_0)$. The \emph{pointed Gromov-Hausdorff distance} between $X$ and $Y$ is defined to be  \label{para:pt dgh}
	$$\dgh^{\pt}((X,x_0),(Y,y_0)):=\tfrac{1}{2}\inf_{R\in \mathfrak{R}^{\pt}((X,x_0),(Y,y_0))}\dis(R).$$ 
It is clear that
\begin{equation*}\label{eq:dgh-dgh^pt}
    \dgh(X,Y)=\inf_{x_0\in X,\,y_0\in Y}\dgh^{\pt}((X,x_0),(Y,y_0)).
\end{equation*}

\subsection{Persistence theory}\label{subsec:persistent theory} 

We recall the definitions of persistence modules and the interleaving distance, as well as the stability theorem for the interleaving distance and the bottleneck distance, from \cite{oudot2015persistence,blumberg2017universality}.  

Let $(\R_+,\leq)$ be a poset category, i.e., a category whose objects are real numbers and the set of morphisms from an object $a$ to an object $b$ consists of a single morphism if $a\leq b$ and is otherwise empty.
Let $\delta\geq0$. We define $S_{\delta}:(\R_+,\leq)\to (\R_+,\leq)$ to be the functor given by $S_{\delta}(t)=t+\delta$ and define $\eta_{\delta}:\Id_{(\R_+,\leq)}\Rightarrow S_{\delta}$ to be the natural transformation given by $\eta_{\delta}(t):t\leq t+\delta$. Note that $S_{\delta}\circ S_{\delta'}=S_{\delta+\delta'}$ and $\eta_{\delta}\circ\eta_{\delta'}=\eta_{\delta+\delta'}.$

Let $\caC$ be any arbitrary category. The collections of functors $$\bbV:(\R_+,\leq)\to \caC,$$ and natural transformations between functors forms a category denoted by $\PcaC^{(\R_+,\leq)}$\label{para:category PC} (see \textsection 2 of \cite{Bubenik2014}). 
Let $\bbV$ and $\bbW$ be objects in $\PcaC^{(\R_+,\leq)}$.
\begin{itemize}
    \item A natural transformation $f:\bbV\Rightarrow\bbW,$
also called a \emph{homomorphism} from $\bbV$ to $\bbW$, is a family of morphisms in $\caC$: $\{f_t:V_t\to W_t\}_{t\in \R_+ }$ such that for any $t\leq t'$ the following diagram commutes:
	\begin{center}
		\begin{tikzcd} 
			V_t \ar[d, "f_t" left] 
			\ar[r,"v_{t,t'}"]
			& 
			V_{t'}			
			\ar[d,"f_{t'}"]
			\\
			W_{t}
			\ar[r,"w_{t,t'}" below]
			& 
			W_{t'}.
		\end{tikzcd}
	\end{center} 
    \label{para:Hom(V,W)}  
    If $f_t$ is an isomorphism in $\caC$ for each $t\in \R_+ $, then $f=\{f_t\}$ is called a \emph{(natural) isomorphism} between $\bbV$ and $\bbW$, in which case we write $\bbV\cong\bbW.$
    \item  A \emph{homomorphism of degree $\delta$} from $\bbV$ to $\bbW$ is a natural transformation 
\label{para:Hom^delta(V,W)} 
$$f:\bbV\Rightarrow\bbW\circ S_{\delta}.$$ 
\end{itemize}

For an object $A\in \caC$ and an interval $I\subset \R_+ $, the \emph{interval (generalized) persistence module} $A[I]$ is defined as follows: $A[I]=\bbmzero$ if $I$ is empty and otherwise,   \label{para:interval g.p.m.}
\begin{equation*}
\left( A[I]\right)(t)=\begin{cases}A,&\mbox{if $t\in I$,}\\
		0,&\mbox{otherwise,}
		\end{cases} 
\end{equation*} 
where $(A[I])(s\leq t)=\Id_{A}$ when $s,t\in I$, and $(A[I])(s\leq t)=0$ otherwise.

In this paper will consider the following choices of $\caC$. Let $\F$ be a field. We denote the category of sets by $\set$, the category of vector spaces over $\F$ by $\vs$, and the category of groups by $\grp$. Also, let $\topo$ be the category of compactly generated weakly Hausdorff topological spaces \cite{blumberg2017universality}, and let $\topo^*$ be the category of pointed compactly generated weakly Hausdorff topological spaces. 
\begin{itemize}
    \item $\caC=\set$ yields $\Pset^{(\R_+,\leq)}$, whose objects are called \emph{persistent sets};
    \item $\caC=\vs$ yields $\Pvec^{(\R_+,\leq)}$, whose objects are called \emph{persistence modules};
    \item $\caC=\grp$ yields $\Pgrp^{(\R_+,\leq)}$, whose objects are called \emph{persistent groups};
    \item $\caC=\topo$ yields $\Ptop^{(\R_+,\leq)}$, whose objects are called \emph{$\R_+ $-spaces}.
\end{itemize}

\begin{definition}[Definition 3.1, \cite{Bubenik2014}]\label{def:delta-interleaving} A $\delta$-interleaving of $\bbV$ and $\bbW$ consists of natural transformations $f:\bbV\Rightarrow \bbW\circ S_{\delta}$ and $g:\bbW\Rightarrow \bbV\circ S_{\delta}$ 
such that the following diagrams commute for all $t\in \R_+ $: 
\begin{equation*}
    (g S_{\delta}) f=\bbV \eta_{2\delta} \text{ and } (f S_{\delta}) g=\bbW \eta_{2\delta}. 
\end{equation*}
If there exists a $\delta$-interleaving $(\bbV,\bbW,f,g)$, we say that $\bbV$ and $\bbW$ are \emph{$\delta$-interleaved}. 
\end{definition}

In other words, a $\delta$-interleaving $(\bbV,\bbW,f,g)$ consists of families of morphisms $\{f_t:V_t\to W_{t+\delta}\}_{t\in \R_+ }$ and $\{g_t:W_t\to V_{t+\delta}\}_{t\in \R_+ }$ such that the following diagrams commute for all $t\leq t'$:
	\begin{center}
		\begin{tikzcd}[column sep={6em,between origins}]
			V_t \ar[dr, "f_t" below left] 	
			\ar[r,"v_{t,t'}"]
			& 
			V_{t'}			
			\ar[dr,"f_{t'}" ]&
			\\
			&W_{t+\delta}
			\ar[r,"w_{t+\delta,t'+\delta}" below]
			& 
			W_{t'+\delta}
		\end{tikzcd}
\hspace{0.6cm}
	\begin{tikzcd}[column sep={6em,between origins}]
			& V_{t+\delta}
			\ar[r,"v_{t+\delta,t'+\delta}"]
			& 
			V_{t'+\delta}
			\\
			W_t
			\ar[ur, "g_t"] 
			\ar[r,"w_{t,t'}" below]
			& 
			W_{t'}
			\ar[ur,"g_{t'}" below right]&
		\end{tikzcd}
	\end{center} 
and
\begin{center}
		\begin{tikzcd}
			V_t
			\ar[dr, "f_t" below left ] 
			\ar[rr,"v_{t,t+2\delta}"]%
			&& 
			V_{t+2\delta}		
			\\
			& W_{t+\delta}
			\ar[ur,"g_{t+\delta}" below right ]
			& 
		\end{tikzcd}
\hspace{0.6cm}
		\begin{tikzcd}
			& V_{t+\delta}
			\ar[dr,"f_{t+\delta}"]
			& 
			\\
			W_t
			\ar[ur, "g_t"] 
			\ar[rr,"w_{t,t+2\delta}" below]
			&& 
			W_{t+2\delta}.	
		\end{tikzcd}
	\end{center}

\begin{definition}\label{def:interleaving}
Let $\bbV$ and $\bbW$ be objects of $\PcaC^{(\R_+,\leq)}$. The \emph{interleaving distance} between $\bbV$ and $\bbW$ is 
		$$\di(\bbV,\bbW):=\inf\{\delta\geq0: \bbV \text{ and }\bbW\text{ are }\delta\text{-interleaved}\}.$$
Here we follow the convention that $\inf \emptyset=+\infty.$ A quick fact is that $\di$ descends to a distance on isomorphism classes of objects in $\PcaC^{(\R_+,\leq)}$. 
\end{definition}

\paragraph*{Persistence modules and persistence diagrams.} Let $\caC=\vs$. Recall from \cite{oudot2015persistence} that a persistence module $\bbV$ is \emph{$\rmq$-tame} if $\mathrm{rank}(v_{t,t'}:V_t\to V_{t'})<\infty$ whenever $t<t'$. If a persistence module $\bbV$ can be decomposed as a direct sum of interval modules (e.g. when $\bbV$ is $\rmq$-tame), say $\bbV\cong \bigoplus_{l\in L}\F(p_l^*,q_l^*)$ where $*$ indicates whether the interval is half-open or not (see \cite{CSGO16}), then its \emph{(undecorated) persistence diagram} is the multiset $$\dgm(\bbV):=\{ (p_l,q_l):l\in L\} -\Delta,$$ where $\Delta:=\{(r,r):r\in \R\}$ is the diagonal in the real plane.

The \emph{bottleneck distance} between persistence diagrams, and more generally between multisets $A$ and $B$ of upper-diagonal points in $\Bar{\R}^2$, where $\Bar{\R}$ are the extended real numbers $\R\cup\{\pm\infty\}$, is defined as follows: \label{para:db} 
    $$\db(A,B):=\inf\left\{ \sup_{a\in A}\|a-\phi(a)\|_{\infty}:\phi: A\cup \Delta^{\infty}\to B\cup \Delta^{\infty}\text{ a bijection }\right\}.$$
Here $\|(p,q)-(p',q')\|_{\infty}:=\max\{|p-p'|,|q-q'|\}$ for each $p,q,p',q'\in \R$ and $\Delta^{\infty}$ is the multiset consisting of each point on the diagonal $\{(r,r):r\in \bar{\R}\}$ in the extended real plane, taken with infinite multiplicity (see \cite{chowdhury2018functorial}).
\begin{theorem}[Theorem 3.4 of \cite{lesnick2015theory}]
\label{thm:di-db}
Let $\bbV$ and $\bbW$ be $\rmq$-tame persistence modules. Then,
		$$\di(\bbV,\bbW)=\db(\dgm(\bbV),\dgm(\bbW)),$$
where $\di(\bbV,\bbW)$ is defined in Definition \ref{def:interleaving} with $\caC=\vs$.
\end{theorem} 
	
\paragraph*{Vietoris-Rips complexes.} By first constructing a simplicial filtration out of a metric space and then applying the simplicial homology functor, one obtains a persistence module that encodes computable invariants of the original space. One such example is the Vietoris-Rips filtration. For a compact metric space $(X,d_X)$ and $\epsilon\geq 0$, the \emph{Vietoris–Rips complex} $\VR_{\epsilon }(X)$ is the simplicial complex with vertex set $X$, where 
\begin{center}
    a finite subset $\sigma\subset X$ is a simplex of $\VR_{\epsilon }(X)$ iff $\diam(\sigma)\leq\epsilon$.
\end{center} 
The collection $\{\VR_{\epsilon }(X)\}_{\epsilon\geq 0}$ together with the natural simplicial inclusions forms a simplicial filtration of the power set of $X$, denoted by $\VR_{\bullet}(X)$.\label{para:VR} 

Applying the $k$-th homology $\rmH_k(\cdot;G)$ with coefficients in an Abelian group $G$ to the filtration $\VR_{\bullet}(X)$, we obtain a persistent object $\pH_k^{\VR}(X;G)$ given by
$$\pH_k^{\VR,\epsilon }(X;G):=\rmH_k(\VR_{\epsilon }(X);G),$$
together with the maps induced by natural simplicial inclusions. 
The collection of these modules for $k$ ranging over all dimensions is called the \emph{persistent homology of $X$}. This paper considers $G=\Z$ or a field $\F$. We will omit the coefficient group when $G=\Z$. 
When $G=\F$, $\pH_k^{\VR}(X;\F)$ is a persistence module, and it is shown in \cite{CSO14} that if $X$ is compact then $\pH_k^{\VR}(X;\F)$ is $\rmq$-tame for all $k\in \Z_{\geq 0}$. The persistence diagram corresponding to $\pH_k^{\VR}(X;\F)$ is denoted by $\dgm_k(X)$ for each $k\in \Z_{\geq 0}$.\label{para:PH} 
\begin{theorem}[{Stability Theorem $\db$, \cite{CCSGMO09,CSO14}}]\label{thm:stab-bot} Let $(X,d_X)$ and $(Y,d_Y)$ be compact metric spaces. Then, for any $k\in \Z_{\geq 0},$
\[\di (\pH_k^{\VR}(X;\F),\pH_k^{\VR}(Y;\F)) = \db(\dgm_k(X),\dgm_k(Y))\leq 2\cdot \dgh(X,Y).\]
\end{theorem}

\subsection{Dendrograms and treegrams} For this section, we refer to \cite{carlsson2010characterization,smith2016hierarchical} for further details. Let $A$ be a set, and let $\Part(A)$ be the set of all partitions of $A$. 
\begin{definition}\label{def:dendrogram} 
A \emph{dendrogram} over $A$ is a pair $(A,\theta_{A})$, where $\theta_A:\R_{\geq 0}\rightarrow \Part(A)$ satisfies:\label{para:dendrogram of A} 
	\begin{itemize}
		\item [(1)] If $t\leq s$, then $\theta_A(t)$ refines $\theta_A(s)$.
		\item [(2)] For all $r$ there exists $\delta>0$ s.t. $\theta_A(r)=\theta_A(t)$ for all $t\in [r,r+\delta]$.
		\item [(3)] There exists $t_0$ such that $\theta_A(t)$ is the single block partition for all $t\geq t_0$.
		\item [(4)] $\theta_A(0)$ is the partition into singletons.
	\end{itemize}
If all conditions except for (4) are satisfied, we call $(A,\theta_{A})$ a \emph{generalized dendrogram}.
\end{definition}
The following function on $A\times A$ defines an ultra-metric on $A$: for any $x,x'\in A$, \label{para:metric by dendrogram} 
	\[\mu_{\theta_A}(x,x'):=\min\{\delta:x\text{ and }x'\text{ belong to the same block of }\theta_A(\delta)\}.\]

Let $(A,\theta_A)$ and $(B,\theta_B)$ be two dendrograms. Given $\delta\geq 0$, we say that the set maps $\phi:A\to B$ and $\psi:B\to A$ provide a \emph{$\delta$-interleaving} between $\theta_A$ and $\theta_B$ iff for all $x\in A, y\in B$ and $t\geq 0$,
\begin{center}
    $\phi(\langle x\rangle_t^A)\subset \langle\phi(x)\rangle_{t+\delta}^B $ and     $\psi(\langle y\rangle_t^B)\subset \langle\psi(y)\rangle_{t+\delta}^A $,
\end{center}
where $\langle x\rangle_t^A$ represents the subset of $A$ containing $x$ in the partition $\theta_A(t)$. If such a $\delta$-interleaving exists, we say that $\theta_A$ and $\theta_B$ are $\delta$-interleaved. The \emph{interleaving distance} between dendrograms $\theta_A$ and $\theta_B$ is defined to be
	$$\di(\theta_A,\theta_B):=\inf \left\{\delta> 0:A\text{ and }B \text{ are }\delta\text{-interleaved}\right\}.$$ 
	
\begin{remark}Each dendrogram $\theta_A$ can be viewed as an element of $\Pset$, following from the definition of $\theta_A$ and the fact that $\Part(A)$ is a subcategory of $\set$. Using Definition \ref{def:interleaving}, we can construct the interleaving distance, denoted by $\di^{\set}$, between dendrograms $(A,\theta_A)$ and $(B,\theta_B)$. It turns out that $\di(\theta_A,\theta_B)=\di^{\set}(\theta_A,\theta_B)$. Indeed, each $\delta$-interleaving $(\phi:A\to B,\psi:B\to A)$ clearly induces a $\Pset$-$\delta$-interleaving. Conversely, given a $\Pset$-$\delta$-interleaving $(\Phi:\Part(A)\to \Part(B),\Psi:\Part(B)\to (A))$, we can define map $\phi:A\to B$ with $a\mapsto b\in \Phi(0)(\{a\})$, and similarly define $\psi:B\to A$. Although $\phi$ and $\psi$ are not uniquely defined, it is not hard to see that the resulting $(\phi,\psi)$ is always a $\delta$-interleaving regardless of the choices.
\end{remark}

The functor $\caV_{\F}\circ \theta_A:(\R_+,\leq)\rightarrow \vs$, with $\caV_{\F}\circ \theta_A(t)=\caV_{\F}(\theta_A(t))$ and $\caV_{\F}\circ \theta_A(t\leq s)=\caV_{\F}(\theta_A(t))\rightarrow \caV_{\F}(\theta_A(s))$ induced by refinement of partitions, is then a persistence module. 

\begin{proposition}[\textsection3.5, \cite{carlsson2010characterization}]\label{prop:stab-dendrogram} Let $(A,\theta_A)$ and $(B,\theta_B)$ be two dendrograms. Then,
$$\tfrac{1}{2}\cdot\di (\caV_{\F}\circ\theta_A,\caV_{\F}\circ\theta_B)\leq  \dgh((A,\mu_{\theta_{A}}),(B,\mu_{\theta_B}))\leq \di (\theta_A,\theta_B) .$$ 
\end{proposition}

\begin{proof} 
The leftmost inequality follows from an argument in \cite{carlsson2010characterization}. Assume $\dgh((X,\mu_{\theta_A}),(Y,\mu_{\theta_B}))\leq \delta/2$ for some $\delta\geq 0$. Then, there are maps $\phi:A\to B$ and $\psi:B\to A$ such that
	\begin{itemize}
		\item if $x$ and $x'$ are in the same block of $\theta_A(t)$, then $\phi(x)$ and $\phi(x')$ belong to the same block of $\theta_B(t+\delta)$;
		\item if $y$ and $y'$ are in the same block of $\theta_B(t)$, then $\psi(y)$ and $\psi(y')$ belong to the same block of $\theta_A(t+\delta)$.
	\end{itemize} 
In other words, $\phi$ and $\psi$ induce homomorphisms of degree $\delta$ on persistence modules $\caV_{\F}\circ \theta_A$ and $\caV_{\F}\circ \theta_B$, denoted by $\phi_{\delta}$ and $\psi_{\delta}$, respectively. Given $t>0$, if $x$ and $x'$ fall into the same block of $\theta_A(t)$, then $\psi\circ\phi(x),\psi\circ\phi(x')$ belong to the same block of $\theta_A(t+2\delta)$. Thus, $\psi_{\delta}\circ\phi_{\delta}=1_{\caV_{\F}\circ \theta_A}^{2\delta}$. Similarly, we have $\phi_{\delta}\circ\psi_{\delta}=1_{\caV_{\F}\circ \theta_B}^{2\delta}$. Therefore, $\caV_{\F}\circ \theta_A$ and $\caV_{\F}\circ \theta_B$ are $\delta$-interleaved. Letting $\delta \searrow 2\dgh((A,\mu_{\theta_A}),(B,\mu_{\theta_B}))$, the leftmost inequality follows.

Suppose $(\phi,\psi)$ is a $\delta$-interleaving between $\theta_A$ and $\theta_B$. Clearly, $$R:A\xtwoheadleftarrow{(\Id_A,\phi)}A\sqcup B\xtwoheadrightarrow{(\psi,\Id_B)}B$$ forms a tripod between $A$ and $B$, with distortion no larger than $2\delta$. Therefore, the rightmost inequality is true.
\end{proof}

Given two sets $A$ and $B$, let $P_A:=\{A_1,\cdots,A_k\}$ and $P_B:=\{B_1,\cdots,B_l\}$ be partitions of $A$ and $B$, respectively. We define the product of $P_A$ and $P_B$ to be
\[P_A\times P_B:=\{(A_i,B_j):1\leq i\leq k, 1\leq j\leq l\},\]
which is clearly a partition of $A\times B$. In fact, the map $(P_A,P_B)\mapsto P_A\times P_B$ gives an embedding $\Part(A)\times\Part(B)\hookrightarrow \Part(A\times B)$. For two dendrograms $(A,\theta_A)$ and $(B,\theta_B)$, we define their product to be $(A\times B,\theta_A\times \theta_B)$, where \[\theta_A\times \theta_B:\R_{\geq 0}\to \Part(A\times B)\text{  with  }t\mapsto \theta_A(t)\times \theta_B(t).\]
It follows directly from Definition \ref{def:dendrogram} that $(A\times B,\theta_A\times \theta_B)$ forms a dendrogram. \label{para:prod of dendrogram}

A \emph{subpartition} of a set $A$ is a partition of one of its subsets $A'\subset A$. Let $\SubPart(A)$ be the set of subpartitions of $A$. \label{para:subpartition}
In \cite{smith2016hierarchical}, the authors introduce a generalization of dendrograms, by allowing subpartitions of the underlying set. This generalization again has a tree-like structure and is called a treegram: 

\begin{definition}\label{def:treegram}
A \emph{treegram} of $A$ is a pair $(A,\theta_A^\rms)$, where $\theta_A^\rms:\R_+\rightarrow \SubPart(A)$ satisfies: 
	\begin{itemize}
	    \item [(1')] If $t\leq s$, then $\theta_A^\rms(t)$ and $\theta_A^\rms(s)$ are partitions of $A_t$ and $A_s$, respectively, where $A_t\subset A_s\subset A$ and $\theta_A^\rms(t)$ refines $\theta_A^\rms(s)|_{A_t}:=\left\{B\cap A_t:B\in \theta_A^\rms(s)\right\}$. 
		\item [(2)] For all $r$ there exists $\delta>0$ s.t. $\theta_A^\rms(r)=\theta_A^\rms(t)$ for all $t\in [r,r+\delta]$.
		\item [(3)] There exists $t_0$ such that $\theta_A^\rms(t)$ is the single block partition for all $t\geq t_0$.
	\end{itemize} 
\end{definition}
	
Similarly, we consider the following function on $A\times A$: for any $x$ and $x'$ in $A$,
\label{para:treegram}
\[\mu_{\theta_A}^\rms(x,x'):=\min\{\delta\geq 0:x,x'\text{ belong to the same block of }\theta_A(\delta)\},\]
which turns out to be a pseudo-ultra-metric via a proof similar to Proposition \ref{prop:stab-dendrogram}:

\begin{proposition}\label{prop:treegram} Let $(A,\theta_A^\rms)$ and $(B,\theta_B^\rms)$ be two treegrams. Then
	$$\tfrac{1}{2}\cdot\di (\caV_{\F}\circ\theta_A^\rms,\caV_{\F}\circ\theta_B^\rms)\leq \dgh((A,\mu^\rms_{\theta_{A}}),(B,\mu^\rms_{\theta_B})).$$ 
\end{proposition}


\section{Persistent Homotopy Groups} \label{sec:persistent homotopy}

In \textsection \ref{sec:p-set}, we study the $0$-th persistent homotopy, which are persistent sets, and see that there is a dendrogram structure associated with it. In later sections, we study the $n$-th persistent homotopy arising from applying the $n$-th homotopy group functor to a filtration for $n\geq 1$. For $n=1$, we give an alternative construction of the persistent fundamental group using discrete fundamental groups.

\subsection{Persistent sets}
\label{sec:p-set} 

\label{para:pi_0} Given a compact metric space $(X,d_X)$ and $\epsilon\geq0$, let $\pi^{\epsilon}_0(X)=X/\sim_0^{\epsilon}$ where $x\sim_0^{\epsilon} x'$ if there exists $\{x_0,\cdots,x_n\}\subset X$ such that $x_0=x,x_n=x'$ and $d_X(x_i,x_{i+1})\leq \epsilon$. For simplicity, we will write the sequence $\{x_0,\cdots,x_n\}$ as $x_0 \cdots x_n$. For $\epsilon'\geq\epsilon\geq0$, there is a natural map from $\pi^{\epsilon}_0(X)$ to $\pi^{\epsilon'}_0(X)$ via $[x]_{\epsilon}\mapsto [x]_{\epsilon'}$. The collection $\{\ppi^{\epsilon}_0(X):=\pi^{\epsilon}_0(X)\}_{\epsilon\geq0}$ together with the natural maps forms a persistent set, i.e., an object in $\Pset^{(\R_{\geq0},\leq)}$, denoted by $\ppi_0(X)$.\label{para:ppi_0}

\begin{remark} Let $\mathcal{V}_\F$ be the functor from $\set$ to $\vs$ given by $\caV_{\F}:A\mapsto\Span_{\F}(A),$ the linear space spanned by $A$ over $\F$. Then
		$\caV_{\F} (\ppi_0(X))\cong \pH_0(X;\F).$ Indeed, fix an $\epsilon\geq0$ and see that the boundary operator
$$\partial_1:C_1(\VR_{\epsilon }(X);\F)=\caV_\F(\{(x,x')\in X\times X: d_X(x,x')\leq \epsilon\})\to C_0(\VR_{\epsilon }(X);\F)=\caV_\F(X)$$
is given by $(x,x')\mapsto x-x'.$ Thus, $\im(\partial_1)=\caV_{\F}\{x'-x: d_X(x,x')\leq \epsilon\}$ and \begin{eqnarray*}\pH_0^{\epsilon}(X;\F)&=&\rmH_0(\VR_{\epsilon }(X);\F)\\
			&=&\caV_{\F}(X)/\caV_{\F}\{x'-x: d_X(x,x')\leq \epsilon\}\\
			&\cong&\caV_{\F}(X/\sim_0^{\epsilon})=\caV_{\F} (\ppi_0^{\epsilon}(X)).
		\end{eqnarray*}
Since the above isomorphisms form a natural isomorphism between $\caV_{\F} (\ppi_0(X))$ and $\pH_0(X;\F)$, $\caV_{\F} (\ppi_0(X))\cong \pH_0(X;\F).$ 
\end{remark}
	
\begin{proposition}
\label{prop:mu^0} 
The following defines an ultra-metric on $X$: for any $x, x'\in X$,
		\begin{eqnarray*}\mu_X^{(0)}(x,x')&:=&\inf\{\epsilon: x\sim_0^{\epsilon} x'\}\\
			&=&\inf\{\epsilon:\exists \{x_0,\cdots,x_n\}\subset X\text{ s.t. }x_0=x,x_n=x',d_X(x_i,x_{i+1})\leq \epsilon\}.
		\end{eqnarray*}
\end{proposition}
	
\begin{proof} Clearly, $\mu_X^{(0)}(x,x)=0$ and $\mu_X^{(0)}(x,x')=\mu_X^{(0)}(x',x)$ for all $x$ and $x'\in X$. It remains to prove the strong triangle inequality. Given arbitrary $x,y,z\in X$, suppose $\epsilon_1>\mu_X^{(0)}(x,y)$ and $\epsilon_2>\mu_X^{(0)}(y,z)$. Let $x_0^1\cdots x^1_{n_1}$ and $ x_0^2\cdots x^2_{n_2}$ be the sequences to realize $x\sim_{\epsilon_1}y$ and $y\sim_{\epsilon_2}z$, respectively. Then, $x_0^1\cdots x^1_{n_1} x^2_{n_2}\cdots x_0^2$ is such that $x_0^1=x,x_0^2=z$ and the distance between adjacent points in that sequence is no larger than $\max\{\epsilon_1,\epsilon_2\}$. By the minimality, $\mu_X^{(0)}(x,z)\leq \max\{\epsilon_1,\epsilon_2\}.$ letting $\epsilon_1 \searrow \mu_X^{(0)}(x,y)$ and $\epsilon_2 \searrow \mu_X^{(0)}(y,z)$, we obtain the strong triangle inequality.
\end{proof}
	
A stability theorem for $\mu^{(0)}_{\bullet}$ can be found in \cite{carlsson2010characterization}. We include it here together with a proof for pedagogical reasons: our proof of Theorem \ref{thm:stab-mu1} will exhibit a similar pattern. 
	
\begin{theorem}[Stability theorem for $\mu^{(0)}_{\bullet}$]\label{thm:stab-pi0} Given compact metric spaces $(X,d_X)$ and $(Y,d_Y)$, 
		$$\dgh\left(\left( X,\mu_X^{(0)}\right),\left( Y,\mu_Y^{(0)}\right)\right)\leq \dgh(X,Y).$$
\end{theorem}
	
\begin{proof} Let $R:X\xtwoheadleftarrow{\phi_X}Z\xtwoheadrightarrow{\phi_Y}Y$ be an arbitrary tripod between $X$ and $Y$. Given $(x,y)$ and $(x',y')$ in $R$, suppose $\epsilon>\mu_X^{(0)}(x,x')$ and $x_0 x_1\cdots x_n$ is a sequence to realize $\epsilon$. For each $0\leq i\leq n-1,$ there exists some $z_i\in Z$ such that $\phi_X(z_i)=x_i$. Let $y_i=\phi_Y(z_i)$. For $0\leq i\leq n-1,$ 
		\begin{eqnarray*}
			d_Y(y_i,y_{i+1})&\leq & d_X(x_i,x_{i+1})+|d_Y(y_i,y_{i+1})- d_X(x_i,x_{i+1})|\\
			&\leq &\epsilon+\dis(R),
		\end{eqnarray*} where $\dis(R)$ denotes the distortion induced by $d_X$ and $d_Y$. Thus, $\{y=y_0,y_1,\cdots,y_{n-1},y_n=y'\}$ is an $ (\epsilon+\dis(R))$-homotopy between $y$ and $y'$. As $\epsilon \searrow \mu_X^{(0)}(x,x')$, it follows that $\mu_Y^{(0)}(y,y')\leq \mu_X^{(0)}(x,x')+\dis(R).$ Similarly, we can prove $\mu_X^{(0)}(x,x')\leq \mu_Y^{(0)}(y,y')+\dis(R).$ Therefore,
		\begin{eqnarray*}
			\dgh((X,\mu_X^{(0)}),(Y,\mu_Y^{(0)}))&=&\tfrac{1}{2}\inf_{R\in\mathfrak{R}(X,Y)} \max_{(x,y),(x',y')\in R}|\mu_X^{(0)}(x,x')-\mu_Y^{(0)}(y,y')|\\
			&\leq &\tfrac{1}{2}\inf_{R\in\mathfrak{R}(X,Y)} \dis(R)\\
			&= & \dgh(X,Y).
		\end{eqnarray*}
\end{proof}

The ultra-metric $\mu^{(0)}_{X}$ induces a dendrogram structure over $X$, which characterizes how bars in the $0$-th barcode of $X$ merge to each other and thus provides additional information of $X$ than the $0$-th barcode. For example, below is a case when the ultra-metric $\mu^{(0)}_{\bullet}$ yields a better approximation to the Gromov-Hausdorff distance of two metric spaces compared to the $0$-th barcode. This phenomenon motivates us to seek a similar dendrogram structure in the case of persistent fundamental groups.
	
\begin{example}\label{ex:twopointspace} Let $X_{\epsilon}:=(\{0,1\},d_{\epsilon})$ be a metric space consisting of two points, where $d_{\epsilon}(0,1)=1+\epsilon$. Since $\mu^{(0)}_{X_{\epsilon}}=d_{\epsilon}$, we have
\[\dgh\left(\left( X_{\epsilon},\mu^{(0)}_{X_{\epsilon}}\right),\left( X_{0},\mu^{(0)}_{X_{0}}\right)\right)=\dgh(X_{\epsilon},X_{0})=\tfrac{\epsilon}{2},\]
for all $\epsilon\geq0$. Hence, using $\mu^{(0)}_{\bullet}$, we are able to recover the Gromov-Hausdorff distance of the pair $(X_{\epsilon},X_{0})$ for any $\epsilon$. However, using the $0$-th barcode, we recover $\dgh(X_{\epsilon},X_{0})$ only when $\epsilon\leq 1$. This is because $\dgm_0(X_{\epsilon})=\{(0,1+\epsilon),(0,+\infty)\}$, and thus by Theorem \ref{thm:stab-bot},
\[\tfrac{1}{2}\db(\dgm_0(X_{\epsilon}),\dgm_0(X_{0}))= \min\left\{\tfrac{\epsilon}{2},\tfrac{1+\epsilon}{4}\right\}\leq \tfrac{\epsilon}{2},\]
where the inequality is strict when $\epsilon>1$.
\end{example}

\subsection{Persistent fundamental groups}
\label{subsec:persistent fundamental group}

In this section, we first review the definition and properties of discrete fundamental groups from \cite{wilkins2011discrete}, and utilize it to construct the first persistent homotopy group $\ppi_1(\bullet)$, also called the persistent fundamental group.
	
\subsubsection{Discrete fundamental groups} Let us fix a parameter value $\epsilon>0$, and adopt the construction of discrete fundamental groups from \cite{wilkins2011discrete}. An \emph{$\epsilon$-chain} in $(X,d_X)$ is a finite sequence of points $\gamma=x_0 x_1\cdots x_n$ such that $d_X(x_i,x_{i+1})\leq \epsilon$ for $i=0,\cdots,n-1.$ When $x_0=x_n$, we say that $\gamma$ is an \emph{$\epsilon$-loop (based at $x_0$)}. The integer $n$ is called the \emph{size}\label{para:size} of $\gamma$, and written as $\size(\gamma)$. The \emph{reversal} of $\gamma$ is the $\epsilon$-chain $\gamma^{-1}:=x_nx_{n-1}\cdots x_0$. The terms \emph{discrete chains} (resp. \emph{discrete loops}) will refer to all $\epsilon$-chains (resp. $\epsilon$-loops) for any $\epsilon>0$. A \emph{subdivision} of a discrete chain $\gamma$ is a discrete chain that can be written as
	$$\alpha'=(x_{0}^0\cdots x_{0}^{m_0} )\cdots ( x_{i}^{0}\cdots x_{i}^{m_i})\cdots ( x_{n-1}^{0}\cdots x_{n-1}^{m_{n-1}}),$$ 
where $x_i^0=x_i, x_i^{m_i}=x_{i+1}$ and $x_{i}^{j_i}\in X$ for any $i$ and $j_i=1,\cdots,m_i-1$. In this case, we denote $\alpha'\supset \alpha$ or $\alpha\subset \alpha'$. 
	
A space $X$ is \emph{$\epsilon$-connected} if any two points $x,y\in X$ can be joined by an $\epsilon$-chain. If $X$ is $\epsilon$-connected for all $\epsilon>0$, we say that $X$ is \emph{chain-connected}. \label{para:epsilon-connected}

Recall from \cite{wilkins2011discrete} that a \emph{basic move} on an $\epsilon$-chain is defined as the addition or removal of a point that is not an endpoint, so that the resulting path is still an $\epsilon$-chain. In our paper, to avoid tedious arguments, basic moves also include the case when there are no changes.

\begin{definition}[$\epsilon$-homotopy]
\label{def:epsilon-homotopy} Two chains $\alpha$ and $\beta$ are \emph{$\epsilon$-homotopic}, denoted by $\alpha\sim_1^{\epsilon}\beta$, if there is a finite sequence of $\epsilon$-chains $$H=\{\alpha=\gamma_0,\gamma_1,\cdots,\gamma_{k-1},\gamma_k=\beta\}$$ such that each $\gamma_i$ differs from $\gamma_{i-1}$ by a basic move. We call $H$ an \emph{$\epsilon$-homotopy} and denote the $\epsilon$-homotopy class of an $\epsilon$-chain $\alpha$ by $[\alpha]_{\epsilon}.$ 
\end{definition}

The following reformulation (see \cite{barcelo2014discrete}) will be useful in the sequel. For a non-negative integer $n$, let $[n]:=\{0,\cdots,n-1\}.$\label{para:n set} Note that an $\epsilon$-chain in $X$ can also be regarded as an $\epsilon$-Lipschitz map $\gamma:([n],\ell^{\infty})\rightarrow (X,d_X)$ for some $n\in \Z_{\geq 0}$, i.e., $d_X(\gamma(i),\gamma(i+1))\leq \epsilon$ for all $i\in [n-1].$ A \emph{lazification} of an $\epsilon$-path $\gamma:[n]\rightarrow X$ is an $\epsilon$-path $\tilde{\gamma}:[m]\rightarrow X$ such that $m \geq n$ and $\tilde{\gamma}=\gamma\circ p$ where $p:[m]\rightarrow [n]$ is surjective and monotone. 
	
\begin{proposition}[Reformulation of $\epsilon$-homotopy]
\label{prop:Lipdef} Two chains $\alpha$ and $\beta$ are \emph{$\epsilon$-homotopic} if and only if there exists a triple $(\tilde{\alpha},\tilde{\beta},\tilde{H})$ where
		\begin{itemize}
			\item $\tilde{\alpha}$ and $\tilde{\beta}$, of the same size $n$, are lazifications of $\alpha$ and $\beta$, respectively;
			\item $\tilde{H}:([n]\times [m],\ell^{\infty})\rightarrow (X,d_X)$ is an $\epsilon$-Lipschitz map  such that $\tilde{H}(\cdot,0)=\alpha,\tilde{H}(\cdot,m)=\beta$, $\tilde{H}(0,t)=\alpha(0)=\beta(0)$ and $\tilde{H}(n,t)=\alpha(n)=\beta(n), \forall t$.
		\end{itemize} 
\end{proposition}

\begin{proof} We prove the `if' part by induction on the number of basic moves. For the base case, we suppose $\alpha=x_0 x_1\cdots x_n$ and $\{\alpha=\gamma_0,\gamma_1=\beta\}$ is an $\epsilon$-homotopy. If there is no change between $\alpha$ and $\beta$, we take $\tilde{H}(\cdot,0)=\gamma_0 \text{ and } \tilde{H}(\cdot,1)=\gamma_1.$

Otherwise, there is a removal or addition of a point from $\alpha$ to $\beta$. Define two maps $p_i:[n]\to [n-1]$ and $p^i:[n+1]\to [n]$ as
		\[p_i(j)=\begin{cases}j,&\mbox{if $j\neq i$}\\
		i-1,&\mbox{if $j=i$}
		\end{cases} \text{ and } p^i(j)=\begin{cases}j&\mbox{if $j< i$,}\\
		i-1,&\mbox{if $j\geq i$}
		\end{cases}.\]
If $\gamma_1$ is obtained from $\gamma_0$ by removing some point $x_i$, we set $\tilde{\gamma_1}=\gamma_1\circ p_i,\tilde{H}(\cdot,0)=\gamma_0 \text{ and }\tilde{H}(\cdot,1)=\tilde{\gamma_1}.$
If $\gamma_1$ is obtained from $\gamma_0$ by adding some point between $x_{i-1}$ and $x_i$, we take $\tilde{\gamma_0}=\gamma_1\circ p^i,\tilde{H}(\cdot,0)=\tilde{\gamma_0} \text{ and }\tilde{H}(\cdot,1)=\gamma_1.$

Suppose the statement is true for any two $\epsilon$-chains which differ by $k-1$ basic moves. Assume that $H=\{\alpha=\gamma_0,\gamma_1,\cdots,\gamma_{k-1},\gamma_k=\beta\}$ is an $\epsilon$-homotopy between $\alpha$ and $\beta$. By the induction hypothesis, there exists a required triple $(\tilde{\gamma_0},\tilde{\gamma_{k-1}},\tilde{H}_1)$. Since $\gamma_{k-1}$ and $\gamma_k$ differ by one basic move, we need to discuss three cases as before. Notice that $\tilde{H}_1$ is essentially a matrix with entries from $X$. If there is no change, we append one more row $\gamma_k$ to $\tilde{H}_1$ to obtain $\tilde{H}$. If $\gamma_{k-1}$ and $\gamma_k$ differ by the removal of the point $\gamma_{k-1}(i)$, we compose $\gamma_k$ with some $p_i:[\size(\gamma_{k-1})]\to [\size(\gamma_{k})]$ and append $\gamma_k\circ p_i$ as a new row to $\tilde{H}_1$. If they differ by the addition of the point $\gamma_{k}(i)$, we compose each row of $\tilde{H}_1$ with some $p^i:[\size(\gamma_{k})]\to [\size(\gamma_{k-1})]$ and append one more row $\gamma_k$ to $\tilde{H}_1$.
		
For the `only if' part, it suffices to establish it for the case $m=1$. Suppose $\beta=x_0'x_1'\cdots x_{n-1}'x_n'$, where $x_0'=x_0$ and $x_n'=x_n$. For each $i=0,\cdots,n-1$, we have $d_X(x_i,x_{i+1}')\leq \epsilon.$ Thus, $x_i'$ can be inserted between $x_{i-1}$ and $x_i$ into $\alpha$ for each $i=1,\cdots,n$, after which $x_i$ can be removed for each $i=1,\cdots,n$. Therefore, we can obtain $\beta$ from $\alpha$ via $2n$ basic moves.
\end{proof}

Fix a basepoint $x_0\in X$ and $\epsilon>0$. 
\begin{proposition}\label{prop:equivalence} The relation $\sim_1^{\epsilon}$ defines an equivalence relation on the set of discrete chains based at $x_0$, denoted by $\caL(X,x_0)$.
\end{proposition}
	
\begin{proof} Clearly, $\sim_1^{\epsilon}$ is reflexive and symmetric. Let $\alpha,\beta$ and $\gamma$ be any three chains such that $\alpha\sim_1^{\epsilon}\beta$ and $\beta\sim_1^{\epsilon}\gamma$. Then, there exists two sequences of $\epsilon$-chains $H_1=\{\alpha=\gamma_0,\gamma_1,\cdots,\gamma_{k-1},\gamma_k=\beta\}$ and $H_2=\{\beta=\gamma_{k+1},\gamma_{k+2},\cdots,\gamma_{k+l-1},\gamma_{k+l}=\gamma\}$ such that each $\gamma_i$ differs from $\gamma_{i-1}$ by a basic move. Then, $H=\{\alpha=\gamma_0,\gamma_1,\cdots,\gamma_{k+l-1},\gamma_{k+l}=\gamma\}$ is a sequence of $\epsilon$-chains, where each adjacent two differ by a basic move. Thus, $ \alpha\sim_1^{\epsilon} \gamma$.
\end{proof}
	
Let $\caL^{\epsilon}(X,x_0)$\label{para:epsilon loop} be the collection of all $\epsilon$-loops based at $x_0$. Given two $\epsilon$-loops $\gamma=x_0x_1\cdots x_{n-1}x_0$ and $\gamma'=x_0y_1\cdots y_{m-1} x_0$, define their \emph{concatenation} by \label{para:concatenation}
$$\gamma\ast \gamma'=x_0x_1\cdots x_{n-1}x_0y_1\cdots y_{m-1} x_0.$$  The \emph{birth time}\label{para:birth} of a discrete loop $\gamma=x_0x_1\cdots x_n $ is defined to be $$\birth(\gamma):=\max_{0\leq i\leq n-1} d_X(x_i,x_{i+1}).$$
The \emph{death time}\label{para:death} of $\gamma$ is defined to be $$\death(\gamma):=\inf\{\epsilon>0:\gamma\sim_1^{\epsilon}\{x_0\}\}.$$ 
Clearly, both $\birth(\gamma)$ and $\death(\gamma)$ are no larger than $\diam(X)$.

\begin{remark}Because $\caL^{\epsilon}(X,x_0)$ is a subset of $\caL(X,x_0)$, it follows from Proposition \ref{prop:equivalence} that $\sim_1^{\epsilon}$ defines an equivalence relation on $\caL^{\epsilon}(X,x_0)$. By Lemma 2.1.2 of \cite{wilkins2011discrete}, $\caL^{\epsilon}(X,x_0)/\sim_1^{\epsilon}$ is a group under the operation of concatenation, where the identity element is $[x_0]_{\epsilon}$ and $[\gamma]_{\epsilon}^{-1}=[\gamma^{-1}]_{\epsilon}$.
\end{remark}
	
\begin{definition}[Discrete fundamental group, \cite{wilkins2011discrete}]\label{def:discrete f. group} The \emph{discrete fundamental group at scale $\epsilon$} of a metric space $X$, based at $x_0$, is
		$$\pi_1^{\epsilon}(X,x_0):=\caL^{\epsilon}(X,x_0)/\sim_1^{\epsilon}.$$
\end{definition}

By Lemma 2.2.9 of \cite{wilkins2011discrete}, if $X$ is an $\epsilon$-connected metric space, then $\pi_1^{\epsilon}(X,x_0)\cong \pi_1^{\epsilon}(X,x_1)$ for any two points $x_0$ and $x_1$ in $ X$. In this case, the discrete fundamental groups at scale $\epsilon$ are independent of choices of basepoints, up to isomorphism, so we can omit the basepoints and simply write $\pi_1^{\epsilon}(X)$. 
	
\begin{example}[{Example 2.2.6, \cite{wilkins2011discrete}}]\label{ex:circle-discrete} Let $\mathbb{S}^1(r)$ be a circle of radius $r>0$, equipped with the geodesic metric. When $r=1$, we simply write $\bbS^1$ for $\mathbb{S}^1(1)$. Then, 
		\[\pi_1^{\epsilon}(\mathbb{S}^1(r))=\begin{cases}\Z,&\mbox{if $0<\epsilon< \tfrac{2\pi}{3}r$,}\\ 
		0,&\mbox{if $\epsilon\geq \tfrac{2\pi}{3}r$.}
		\end{cases}\] 
\end{example}

In Theorem \ref{thm:iso of d.f.g}, we will see that discrete fundamental groups of a compact metric space are in fact isomorphic to fundamental groups of its successive Vietoris-Rips complexes.

\paragraph*{Edge-path groups of Vietoris-Rips complexes.} Given a simplicial complex $K$, an \emph{edge} in $K$ is an ordered pair $e=v_1v_2$ for vertices of $K$, such that $v_1$ and $v_2$ are in the same simplex. An \emph{edge path} in $K$ is a sequence of vertices connected by edges. The concatenation of two edge paths can be defined in a natural way and called the \emph{product} of two edge paths. If $e=v_1v_2$ and $f=v_2v_3$ are such that $v_1,v_2$ and $v_3$ are vertices of a simplex, then the product $ef$ is \emph{edge equivalent} to $v_1v_3$. Two edge paths are \emph{edge equivalent} if one can be obtained from another by a sequence of such elementary edge equivalences. Let $v_0$ be a vertex of $K$. When an edge starts and ends at the same vertex $v_0$, we call it an \emph{edge loop at $v_0$}. We define $\pi_1^{\mathrm{E}}(K,v_0)$ to be the set of edge equivalence classes of edge loops at $v_0$, called the \emph{edge-path group} of $K$ (see page 87 of \cite{singer1967}). Let $|K|$ denote the geometric realization of $K$. \label{para:geo-realization}

\begin{theorem} [{Theorem 4 and 5, \textsection 4.4, \cite{singer1967}}]
\label{thm:edgepath} With the notations above, $\pi_1^{\mathrm{E}}(K,v_0)$ is a group, with identity $v_0v_0$, under the operation of product defined above. Furthermore,
$$\pi_1^{\mathrm{E}}(K,v_0)\cong \pi_1(|K|,v_0).$$
\end{theorem}

We see that the following groups are isomorphic.
\begin{restatable}[Isomorphisms]{theorem}{isoofPI1} 
\label{thm:iso of d.f.g}
Let $(X,x_0)$ be a pointed compact metric space and let $\epsilon\geq 0$. Then
    $$\pi^{\epsilon}_1(X,x_0)\cong\pi^{\mathrm{E}}_1(\VR_{\epsilon }(X),x_0)\cong \pi_1\left( \left|  \VR_{\epsilon }(X)\right| ,x_0\right).$$
\end{restatable}

\begin{proof} The rightmost isomorphism follows directly from Theorem \ref{thm:edgepath}. The leftmost isomorphism is straightforward from the definitions. Indeed, if we write $\caL^{\mathrm{E}}(\VR_{\epsilon }(X),x_0)$ as the set of loops at $x_0$ in $\VR_{\epsilon }(X)$, it is clear that $\caL^{\epsilon}(X,x_0)= \caL^{\mathrm{E}}(\VR_{\epsilon }(X,x_0))$ as sets. It remains to check that $\epsilon$-homotopy equivalence is the same as the edge equivalence in $\caL^{\mathrm{E}}(\VR_{ \epsilon}(X,x_0))$, which can be done by proving that basic moves are equivalent to elementary edge equivalences. Indeed, for $x_1,x_3\in X$ such that $d_1(x_1,x_3)\leq \epsilon$, $x_1x_2x_3\sim_1^{\mathrm{E}} x_1x_3$ iff $x_1$, $x_2$ and $x_3 $ form a $2$-simplex in $\VR_{\epsilon}(X,x_0)$, iff $x_1x_2x_3\sim^{\epsilon}_1 x_1x_3$. 
\end{proof}

Given $\epsilon'>\epsilon$, an $\epsilon$-chain is also an $\epsilon'$-chain and an $\epsilon$-homotopy is also an $\epsilon'$-homotopy. Thus, there is a natural group homomorphism 
\begin{equation}\label{eq:Phieep}\Phi_{\epsilon,\epsilon'}: \pi_1^{\epsilon}(X,x_0)\rightarrow\pi_1^{\epsilon'}(X,x_0)\text{ with }[\alpha]_{\epsilon}\mapsto [\alpha]_{\epsilon'}.
\end{equation}
The collection $\{\ppi_1^{\epsilon}(X,x_0):=\pi_1^{\epsilon}(X,x_0)\}_{\epsilon>0}$ together with the natural group homomorphisms $\left\{\Phi_{\epsilon,\epsilon'}\right\}_{\epsilon'\geq \epsilon>0}$ forms a persistent group, denoted by $\ppi_1^{\bullet}(X,x_0)$, or $\ppi_1(X,x_0)$ for simplicity. \label{para:ppi_1^epsilon}

\begin{remark} The leftmost isomorphism in Theorem \ref{thm:iso of d.f.g} was first established in page 599 of \cite{plaut2009equivalent}. 
\end{remark}

\subsubsection{Persistent fundamental groups} 

Recall from \textsection \ref{subsec:persistent theory} the general construction of persistent groups. 
By \emph{isomorphisms} between persistent groups, we will mean isomorphisms in the category $\Pgrp^{(\R_+,\leq)}$, denoted by $\cong$. We denote by $\bbmzero$ the trivial persistent group. 
For a group $G$ and an interval $I\subset \R_+$, the interval generalized persistence module $G[I]$ (see page \pageref{para:interval g.p.m.}) is also called the \emph{interval persistent group}. 

We immediately obtain the following proposition from the definitions of product and coproduct in a category.

\begin{proposition}[Products and coproducts]\label{prop:prod-coprod} In $\Pgrp^{(\R_+,\leq)}$, for any two objects $\bbG$ and $\bbH$,
		\begin{itemize}
			\item the product is given by component-wise direct products and denoted by $\bbG\times\bbH$;
			\item the coproduct is given by component-wise free products and denoted by $\bbG\ast\bbH$.
		\end{itemize}
\end{proposition}

Given a pointed metric space $(X,x_0)$, there are different approaches to construct persistent groups using homotopy information from the metric space, such as the persistent group $\ppi_1(X,x_0)$ constructed on page \pageref{para:ppi_1^epsilon}.

\begin{definition}[Persistent fundamental group]\label{def:persistent f. group} Given a pointed metric space $(X,x_0)$, the persistent group $\ppi_1(X,x_0)$, given by $\{\pi_1^{\epsilon}(X,x_0)\}_{\epsilon>0}$ together with the natural group homomorphisms $\left\{\Phi_{\epsilon,\epsilon'}\right\}_{\epsilon'\geq \epsilon>0}$, is called the \emph{persistent fundamental group} of $X$.
\end{definition}

Let us examine some properties of persistent fundamental groups here.

\begin{proposition}[Induced homomorphisms] \label{prop:induced} Let $X$ and $Y$ be chain-connected metric spaces, and let $\varphi:(X,x_0)\to (Y,y_0)$ be any pointed set map.
		\begin{romanlist}
			\item For any $\delta\geq \dis(\varphi)$ and $\epsilon>0$, the map \begin{center}
			    $\varphi_{\epsilon}^{\epsilon+\delta}:\pi_1^{\epsilon}(X)\to \pi_1^{\epsilon+\delta}(Y)$ with $[\alpha]_{\epsilon}\mapsto [\varphi(\alpha)]_{\epsilon+\delta}$
			\end{center} is a well-defined group homomorphism, and $(\varphi_{\epsilon}^{\epsilon+\delta})_{\epsilon>0}$ is a homomorphism of degree $\delta$ from $\ppi_1(X)$ to $\ppi_1(Y)$ (see page \pageref{para:Hom^delta(V,W)}).
			
			\item If $\varphi$ is $1$-Lipschitz, then for each $\epsilon>0$, the map  \begin{center}
			    $\varphi_{\epsilon}:\pi_1^{\epsilon}(X)\to \pi_1^{\epsilon}(Y)$ with $[\alpha]_{\epsilon}\mapsto [\varphi(\alpha)]_{\epsilon}$
			\end{center}is a well-defined group homomorphism, and $(\varphi_{\epsilon})_{\epsilon>0}$ defines a homomorphism from $\ppi_1(X)$ to $\ppi_1(Y)$ (see page \pageref{para:Hom(V,W)}).
		\end{romanlist}
	\end{proposition}
	
\begin{proof} Given an $\epsilon$-loop $\alpha=x_0 x_1\cdots x_n$ in X, $\varphi(\alpha)=\varphi(x_0)\cdots \varphi(x_n)$ is an $(\epsilon+\delta)$-loop in $Y$, because for $0\leq i\leq n-1$,
		$$d_Y(\varphi(x_i),\varphi(x_{i+1}))\leq d_X(x_i,x_{i+1})+|d_Y(\varphi(x_i),\varphi(x_{i+1}))-d_X(x_i,x_{i+1})|\leq \epsilon+\delta.$$
Since $\varphi$ preserves basic moves, given an $\epsilon$-homotopy $H=\{\alpha=\alpha_0,\cdots,\alpha_k=\beta\}$, $\varphi(H):=\{\varphi(\alpha)=\varphi(\alpha_0),\cdots,\varphi(\alpha_k)=\varphi(\beta)\}$ is an $(\epsilon+\delta)$-homotopy in $Y$. Thus, $\varphi_{\epsilon}^{\epsilon+\delta}$ is well-defined. In addition, $\varphi_{\epsilon}^{\epsilon+\delta}$ is a group homomorphism, because for any $\alpha,\beta\in \caL(X,x_0)$,
\[ \varphi_{\epsilon}^{\epsilon+\delta}([\alpha\ast \beta]_{\epsilon})
    =[\varphi(\alpha)\ast \varphi(\beta)]_{\epsilon+\delta}
    =[\varphi(\alpha)]_{\epsilon+\delta}[\varphi(\beta)]_{\epsilon+\delta}=\varphi_{\epsilon}^{\epsilon+\delta}([\alpha]_{\epsilon})\varphi_{\epsilon}^{\epsilon+\delta}([\beta]_{\epsilon+\delta}).\]
Here $\ast$ denotes the concatenation of discrete chains.

The statement that $(\varphi_{\epsilon}^{\epsilon+\delta})_{\epsilon>0}$ is a homomorphism of degree $\delta$ from $\ppi_1(X)$ to $\ppi_1(Y)$ follows from the following commutative diagram (for all $\epsilon'>\epsilon$):
		\begin{center}
			\begin{tikzcd} 
				{[\alpha]}_{\epsilon} \ar[r, mapsto] 
				\ar[d,mapsto]
				& 
				{[\varphi(\alpha)]}_{\epsilon+\delta} \ar[d, mapsto] 
				\\
				{[\alpha]}_{\epsilon'}  \ar[r, mapsto] 
				& 
				{[\varphi(\alpha)]}_{\epsilon'+\delta}.
			\end{tikzcd}
		\end{center} 
		
Similar arguments can be applied to prove Part (2), so we omit them.
\end{proof}

\subsection{Persistent homotopy groups} \label{subsec:persistent n-homotopy}

In this section, we relate $\ppi_1(\bullet)$ with persistent fundamental groups arising from the Vietoris-Rips filtration and the Kuratowski filtration, and study persistent homotopy groups in dimensions $n\geq 2$, as well as persistent rational homotopy groups.

\subsubsection{Persistent \texorpdfstring{$\rmK$}{K}-homotopy groups and persistent \texorpdfstring{$\VR$}{VR}-homotopy groups}

For $n\in\Z_{\geq 1}$, assigning the $n$-th homotopy group to a pointed topological space forms a functor from the category of pointed spaces to the category of groups, called the \emph{homotopy group functor} and denoted by $\pi_n:\topo^*\to \grp$. In particular, when $n=1$, we call it the \emph{fundamental group functor}. In analogy to persistent homology theory, one could apply the homotopy group functor to filtrations of a given metric space to obtain a persistent group.
\medskip

\paragraph*{Kuratowski filtration and persistent $\rmK$-homotopy groups.} Let $(X,d_X)$ be a bounded metric space. We recall from \cite{borsuk1967retracts} a method to embed $X$ isometrically as a subset of some Banach space, which will eventually allow us to enlarge $X$. Let $K(X):=(L^{\infty}(X),\|\cdot\|_{\infty})$, where $L^{\infty}(X)$ is the space of bounded functions $f:X\to \R$ and $\|f\|_{\infty}:=\sup_{x\in X}|f(x)|$. For a bounded metric space $(X,d_X)$, the \emph{Kuratowski embedding} is given by 
\begin{align*}
    k_X: X&\to L^\infty(X), \\
    x&\mapsto d_X(x,\cdot).
\end{align*}
By the Kuratowski–Wojdysławski Theorem (see Theorem III.8.1 of \cite{borsuk1967retracts}), the Kuratowski embedding $k_X$ is distance-preserving. 
\begin{definition}[$\epsilon$-thickenings]\label{def:epsilon thickening} The \emph{$\epsilon$-thickening} of a compact metric space $(X,d_X)$ is the closed (resp. open) $\epsilon$-neighborhood of $k_X(X)$ in $K(X)$ , henceforth denoted by $X^{\epsilon}$ (resp. $X^{<\epsilon}$).
\end{definition}

Let $(X,x_0)$ be a pointed compact metric space. For any $0\leq \epsilon<\epsilon'$, there is a natural embedding $i_{\epsilon,\epsilon'}^X:X^{\epsilon}\to X^{\epsilon'}$, which forms a filtration in $\topo$, called the \emph{Kuratowski filtration} and denoted by $\bbX$. When there is no danger of confusion, we write $i^X$ for $i_{\epsilon,\epsilon'}^X$. Under the map $k_X$, $x_0$ can be viewed as a point in $X^{\epsilon}$ for each $\epsilon>0$, which induces a functor $(\bbX,x_0):(\R_{\geq 0},\leq)\to \topo^*$ such that $(\bbX,x_0)(\epsilon)=(X^{\epsilon},x_0)$ and $(\bbX,x_0)(\epsilon\leq \epsilon')=i_{\epsilon,\epsilon'}^X$.

\begin{definition}[Persistent $\rmK$-homotopy groups\footnote{$\rmK$ is for Kuratowski.}]\label{def:ppi^K} Let $(X,x_0)$ be a pointed compact metric space. By $(\bbX,x_0)$, composing the homotopy group functor $\pi_n$ with $(\bbX,x_0)$ induces a persistent group  
$$\ppi_n^{\rmK}(X,x_0):=\left\{\pi_n(X^{\epsilon},x_0)\right\}_{\epsilon\geq 0}, $$
together with the induced homomorphisms on homotopy groups. We call it the \emph{$n$-th persistent $\rmK$-homotopy group}.
\end{definition}

\begin{remark}\label{rmk:persistent-K-homology} Composing the homology functor $\rmH_n$ with $\bbX$ results into the \emph{persistent $\rmK$-homology group} $\pH_n^{\rmK}(X,x_0)$.
\end{remark}

Similarly, we can apply homotopy group functors to the geometric realizations of Vietoris-Rips complexes.
\begin{definition}[Persistent $\VR$-homotopy groups]\label{def:ppi^VR} Let $(X,x_0)$ be a pointed compact metric space. For each $n\in \Z_{\geq 1}$, we have the persistent group \label{para:ppi_1^K}  $$\ppi_n^{\VR}(X,x_0):=\left\{\pi_n\left(\left|  \VR_{\epsilon}(X)\right| ,x_0\right)\right\}_{\epsilon\geq 0},$$ 
together with the induced homomorphisms. We call it the \emph{$n$-th persistent $\VR$-homotopy group}. 
\end{definition}

As Vietoris-Rips complexes are simplicial complexes, one can also apply the edge-path group (see page \pageref{para:geo-realization}) functor to obtain a persistent group \label{para:ppi_1^E} 
    $$\ppi_1^{\mathrm{E}}(X,x_0):=\left\{\pi_1^{\mathrm{E}}(\VR_{\epsilon}(X),x_0)\right\}_{\epsilon\geq 0},$$
    together with the induced homomorphisms on edge-path groups. We have seen in Theorem \ref{thm:iso of d.f.g} that
$$\ppi_1^{\VR}(X,x_0)\cong \ppi_1^{\mathrm{E}}(X,x_0)\cong \ppi_1(X,x_0).$$
It will be seen in Theorem \ref{thm:iso-ppi_1's} that Definition $\ppi_1^{\rmK}(X,x_0)$ is isomorphic to the others as well.
 
\begin{remark} One advantage of persistent $\rmK$-homotopy groups and persistent $\VR$-homotopy groups is that they can be defined for $n$-th homotopy groups. However, it does not appear to be trivial to generalize the idea of persistent fundamental group or $\ppi_1^{\mathrm{E}}(X,x_0)$ to higher dimensions.
\end{remark}

\begin{theorem}[Theorem 3.1 of \cite{lim2020vietorisrips}] \label{thm:X*-VR(X)}Let $X$ be a compact metric space and $\epsilon\geq0$. Then,
$$\bbX^{<\bullet}\cong \left|  \VR_{<2\bullet}(X)\right| ,$$
as functors from $(\R_{\geq 0},\leq)$ to the homotopy category of $\topo$.
\end{theorem}

\begin{remark} By a very similar argument with Theorem \ref{thm:X*-VR(X)}, one can check that if $X$ is a finite metric space, then 
$$\bbX^{\leq\bullet}\cong \left|  \VR_{\leq 2\bullet}(X)\right| .$$
If $X$ is not finite, it does not appear to be a trivial question whether the above isomorphism still holds. But we always have that 
\[\dhi \left( \bbX^{\leq\bullet}, \left|  \VR_{\leq 2\bullet}(X)\right| \right) = 0,\]
by approximating $X$ with a sequence of finite metric spaces under $\dgh$ and applying the triangle inequality, the stability of Kuratowski filtration and the Vietoris-Rips filtration. Indeed, let $\epsilon>0$ be arbitrarily small, and let $X_{\epsilon}$ be a finite metric space such that $\dgh(X,X_\epsilon)\leq \epsilon$. Then
\begin{align*}
    \dhi \left( \bbX^{\leq\bullet}, \left|  \VR_{\leq 2\bullet}(X)\right| \right) 
    & \leq \dhi \left( \bbX^{\leq\bullet}, \bbX_{\epsilon}^{\leq\bullet}\right)
    +\dhi \left( \bbX_{\epsilon}^{\leq\bullet}, \left|  \VR_{\leq 2\bullet}(X_{\epsilon})\right| \right)\\
    & +\dhi \left( \left|  \VR_{\leq 2\bullet}(X_{\epsilon})\right| , \left|  \VR_{\leq 2\bullet}(X)\right| \right)\\
    &\leq \epsilon + 0 + \epsilon = 2\epsilon.
\end{align*}
\end{remark}

\begin{corollary}\label{cor:iso-ppi_n} Let $(X,x_0)$ be a pointed compact metric space. Then,
$$\ppi_n^{\rmK,\bullet}(X,x_0)\cong \ppi_n^{\VR,2\bullet}(X,x_0)\text{  and  }\pH_n^{\rmK,\bullet}(X,x_0)\cong \pH_n^{\VR,2\bullet}(X,x_0),$$
holds for the open version and for the closed version with $X$ a finite metric space.
\end{corollary}

\subsubsection{Isomorphism of persistent fundamental groups}
In this section, we prove the following isomorphism theorem of persistent fundamental groups and study the persistent fundamental group of the unit circle and tree-like metric spaces.
\isoofPPI*

\begin{proof} The first isomorphism in Eq. (\ref{eq:iso of p.f.g}) follows from Theorem \ref{thm:X*-VR(X)} directly. The second isomorphism is derived from Theorem \ref{thm:iso of d.f.g}, by checking the following collection of group isomorphisms  
    $$\pi^{\epsilon}_1(X,x_0)\cong \pi_1\left( \left|  \VR_{\epsilon }(X)\right| ,x_0\right), \forall\epsilon\geq 0 $$
forms an isomorphism between persistent groups.
\end{proof}

\begin{example}[Tree-like metric spaces] \label{ex:finite-tree}
Recall from page \pageref{def:hyp} that a tree-like metric space $(T,d_T)$ is a finite metric space where $d_T$ satisfies the four-point condition. Corollary 2.13 of \cite{lim2021some} shows that for any $\epsilon$, the Vietoris-Rips complex $\VR_\epsilon(T)$ is homotopy equivalent to a disjoint union of points. Thus, 
for each $x_0\in T$, 
\[\ppi_n(T,x_0)=\bbmzero.\]

For $n=1$, there is a more constructive proof for $\ppi_1(T,x_0)=\bbmzero$. Take arbitrary $\epsilon>0$ and any $\epsilon$-loop $\alpha=x_0x_1\dots x_nx_0$. We claim that for any $i$, if $d_T(x_i,x_0)\geq \max\{d_T(x_{i-1},x_0),d_T(x_{i+1},x_0)\}$, then $d_T(x_{i-1},x_{i+1})\leq \epsilon$. Indeed, the claim follows directly from the four-point condition:
\begin{align*}
    &d_T(x_{i-1},x_{i+1})\\
    \leq&\max\{ d_T(x_{i-1},x_0)+d_T(x_i,x_{i+1}),d_T(x_{i+1},x_0)+d_T(x_i,x_{i-1}) \}-d_T(x_i,x_0)\\
    \leq&\max\{ d_T(x_{i-1},x_0)+\epsilon,d_T(x_{i+1},x_0)+\epsilon \}-d_T(x_i,x_0)\\
    \leq&\max\{ d_T(x_{i-1},x_0),d_T(x_{i+1},x_0)\}-d_T(x_i,x_0)+\epsilon\leq \epsilon.
\end{align*}
By the claim, we can inductively remove the furthest point to $x_0$ in $\alpha$ to obtain an $\epsilon$-homotopy between $\alpha$ and the trivial loop $x_0.$ Thus, $\ppi_1(T,x_0)=\bbmzero$.

\end{example}

\begin{example}[Unit circle $\bbS^1$]\label{ex:circle-VR-ppi_n} Recall from Theorem 7.6 of \cite{AdamA15} that we have homotopy equivalence
 \begin{equation*}
 \left| \VR_{ r}(\bbS^1)\right| \cong \begin{cases} 
 \bbS^{2l+1},&\mbox{if $\tfrac{l}{2l+1}\, 2\pi<r<\tfrac{l+1}{2l+3}\, 2\pi$ for some $l=0,1,\cdots$,}\\
 \bigvee^{\mathfrak{c}}\bbS^{2l},&\mbox{if $r=\tfrac{l}{2l+1}\, 2\pi$ for some $l=0,1,\cdots$,}\\
 \ast, &\mbox{if $r\geq \pi.$}
 \end{cases}
 \end{equation*}
Here $\mathfrak{c}$ is the cardinality of the continuum (i.e. the cardinality of $\R$), and $\ast$ is the one-point space.

Let $k\in\Z_{\geq 1}$. As $\bigvee^{\mathfrak{c}}\bbS^{2k}$ is $(2k-1)$-connected, it follows from the Hurewicz theorem (cf. Theorem 4.32 of \cite{Hatcher01(AT)}) that
\[\pi_{2k} \left(\bigvee^{\mathfrak{c}}\bbS^{2k}\right)\cong \rmH_{2k} \left(\bigvee^{\mathfrak{c}}\bbS^{2k}\right)\cong\Z^{\times \mathfrak{c}}.\]
Then, the persistent homology groups of $\bbS^1$ are
 \begin{equation*} 
 \label{eq:pH of S1}
\pH_{n}^{\VR}(\bbS^1)\cong \begin{cases} 
 \Z\left(\tfrac{k-1}{2k-1}\, 2\pi,\,\tfrac{k}{2k+1}\, 2\pi\right),&\mbox{if $n=2k-1,$}\\[6pt] 
 \Z^{\times \mathfrak{c}}\left[\tfrac{k}{2k+1}\, 2\pi,\,\tfrac{k}{2k+1}\, 2\pi\right],&\mbox{if $n=2k$.}
 \end{cases}
\end{equation*}
Because $\pi_n(\bbS^m)$ is not totally known for the case $n>m$, the calculation of $\ppi_{n}^{\VR}(\bbS^1)$ can only be done for some choices of $n$. For example, we have
 \begin{equation*} \label{eq:ppi of S1}
\ppi_{n}^{\VR}(\bbS^1)\cong \begin{cases} 
 \Z\left( 0,\,\tfrac{2\pi}{3}\right),&\mbox{if $n=1$,}\\[6pt] 
 \Z^{\times \mathfrak{c}}\left[\tfrac{2\pi}{3},\,\tfrac{2\pi}{3}\right],&\mbox{if $n=2$.}
 \end{cases}
\end{equation*}
\end{example}

To circumvent the difficulty in computing $\ppi_n^{\VR}(\bbS^1)$, in the next subsection we study persistent rational homotopy groups, which is a weaker invariant compared to persistent homotopy groups but is easier to compute. In particular, we compute the $n$-th persistent rational homotopy groups of $\bbS^1$ in Remark \ref{rmk:p.rational.S1} for any $n\geq 1$. 

\subsubsection{Persistent rational homotopy groups} 
\label{sec:rational homotopy groups}

The \emph{$n$-th rational homotopy group} $\pi_n(X,x_0)\otimes_{\Z}\Q$ of a pointed topological space $(X,x_0)$ is the homotopy group $\pi_n(X,x_0)$ tensored with the rational numbers $\Q$. We assume that all topological spaces are path-connected and denote rational homotopy groups by $\pi_n(X)\otimes_{\Z}\Q$ for simplicity of notation. Compared with the difficulty of determining homotopy groups of spheres, the rational homotopy groups of spheres is substantially easier to compute and was done by Serre in 1951 (see \cite{serre1951homologie}):
\[\pi_n(\bbS^{2k-1})\otimes_{\Z}\Q\cong \begin{cases}
	\Q,&\mbox{$n=2k-1$,}\\ 
	0,&\mbox{otherwise,}
	\end{cases} \text{ and }
	\pi_n(\bbS^{2k})\otimes_{\Z}\Q\cong \begin{cases}
	\Q,&\mbox{$n=2k$ or $n=4k-1$,}\\ 
	0.&\mbox{otherwise.}
	\end{cases}\]
This inspires us to consider the notion of \emph{persistent rational homotopy groups}, by tensoring persistent homotopy groups with $\Q$.

Let $\vs$ be the category of vector spaces over $\Q$, $\ab$ be the category of Abelian groups, and $\Pab^{(\R_+,\leq)}$ be the category of functors $\bbG:(\R_+,\leq)\to \ab$. 
Let $-\otimes_{\Z} \Q$ be the tensor product functor from $\ab$ to $\vs$, which naturally 
induces a functor from $\Pab^{(\R_+,\leq)}$ to $\Pvec^{(\R_+,\leq)}$ such that for each $\bbG\in \Pab^{(\R_+,\leq)}$, $\bbG\otimes_{\Z} \Q$ is the composition of the following two functors:
\[ (\R_+,\leq)\xrightarrow{\bbG} \ab \xrightarrow{-\otimes_{\Z} \Q} \vs. \]

\begin{remark} \label{rmk:p.rational.S1}
With the above construction and Example \ref{ex:circle-VR-ppi_n}, we can fully characterize the persistent rational homotopy groups of the unit circle $\bbS^1$:
 \begin{equation*}
\ppi_{n}^{\VR}(\bbS^1)\otimes_{\Z} \Q\cong \begin{cases} 
 \Q\left(\tfrac{2k-1}{4k-1}\, 2\pi,\,\tfrac{2k}{4k+1}\, 2\pi\right) \oplus \Q^{\times \mathfrak{c}}\left[\tfrac{2k}{4k+1}\, 2\pi,\,\tfrac{2k}{4k+1}\, 2\pi\right] ,&\mbox{if $n=4k-1,$}\\[6pt]
 \Q\left(\tfrac{2k}{4k+1}\, 2\pi,\,\tfrac{2k+1}{4k+3}\, 2\pi\right),&\mbox{if $n=4k+1,$}\\[6pt] 
 \Q^{\times \mathfrak{c}}\left[\tfrac{k}{2k+1}\, 2\pi,\,\tfrac{k}{2k+1}\, 2\pi\right],&\mbox{if $n=2k$,}
 \end{cases}
\end{equation*}
\end{remark}
On the other hand, we have the persistent rational homology groups of $\bbS^1$:
 \begin{equation} \label{eq:prH of S1}
\pH_{n}^{\VR}(\bbS^1;\Q)\cong\pH_{n}^{\VR}(\bbS^1)\otimes_{\Z} \Q\cong \begin{cases} 
 \Q\left(\tfrac{k-1}{2k-1}\, 2\pi,\,\tfrac{k}{2k+1}\, 2\pi\right),&\mbox{if $n=2k-1,$}\\[6pt] 
 \Q^{\times \mathfrak{c}}\left[\tfrac{k}{2k+1}\, 2\pi,\,\tfrac{k}{2k+1}\, 2\pi\right],&\mbox{if $n=2k$.}
 \end{cases}
\end{equation}
The leftmost isomorphism in Eq. (\ref{eq:prH of S1}) follows from the universal coefficient theorem for homology (cf. Theorem 3A.3 of \cite{Hatcher01(AT)}). In addition, it can be directly checked that \[\di\left(\ppi_{n}^{\VR}\left(\bbS^1\right)\otimes_{\Z} \Q,\pH_{n}^{\VR}\left(\bbS^1;\Q\right)\right)=0.\]
Note that in the case of $n=4k-1$, because the persistence modules $\ppi_{n}^{\VR}\left(\bbS^1\right)\otimes_{\Z} \Q$ and $\pH_{n}^{\VR}\left(\bbS^1;\Q\right)$ contain different types of indecomposables, they are not isomorphic to each other. 

We see that persistent rational homotopy groups are stable, following from the stability of persistent homotopy groups.
Indeed, any $\delta$-interleaving between two persistent Abelian groups $\bbG$ and $\bbH$ induces a $\delta$-interleaving between the persistence modules $\bbG\otimes_{\Z}\Q$ and $\bbH\otimes_{\Z}\Q$. Thus, 
\[\di\left(\bbG\otimes_{\Z}\Q,\bbH\otimes_{\Z}\Q\right)\leq \di\left(\bbG,\bbH\right).\]
So we have the following corollary of Theorem \ref{thm:stab-ppi_n}.
\begin{corollary} \label{cor:ppi_n rational}
Let $X$ and $Y$ be compact chain-connected metric spaces. Then, for each $n\in \Z_{\geq 2}$,
\[\di\left(\ppi_n^{\rmK}(X)\otimes_{\Z}\Q,\ppi_n^{\rmK}(Y)\otimes_{\Z}\Q\right)\leq  \dgh(X,Y).\]
When $\ppi_1^{\rmK}(X),\ppi_1^{\rmK}(Y)\in \Pab$, we also have
\[\di\left(\ppi_1^{\rmK}(X)\otimes_{\Z}\Q,\ppi_1^{\rmK}(Y)\otimes_{\Z}\Q\right)\leq  \dgh(X,Y).\]
\end{corollary}

Remark \ref{rmk:p.rational.S1} is an example where persistent rational homotopy groups contain slightly more information than persistent rational homology groups. 
There are also cases when the distinguishing power of persistent rational homotopy groups is superior to that of persistent rational homology groups, see Example \ref{ex:rational_homotopy} in \textsection \ref{sec:stability-ppi_n}.

\subsection{Some properties of persistent homotopy groups}
\label{sec:properties}

In this section, we study persistent homotopy groups under products or wedge sums, and establish a persistent Hurewicz theorem to describe the relation between persistent homotopy and persistent homology.

\subsubsection{Persistent homotopy groups under products}\label{sec:ppi_n-product}
	
Let the Cartesian product $X\times Y$ of two metric spaces be equipped with the $\ell^{\infty}$ product metric:
	$$d_{X\times Y}((x,y),(x',y')):=\max\{d_X(x,x'),d_Y(y,y')\},\forall x,x'\in X; y,y'\in Y.$$
Note that when $X$ and $Y$ are chain-connected, $X\times Y$ is also chain-connected.

\begin{proposition}[Proposition 10.2, \cite{AdamA15}]\label{prop:product of VR} Let $X$ and $Y$ be pointed metric spaces. For each $\epsilon>0$, we have the basepoint preserving homotopy equivalence $$\left| \VR_{ \epsilon}(X)\right|  \times \left|  \VR_{ \epsilon}(Y)\right| \xrightarrow{\cong} \left|  \VR_{ \epsilon}(X\times Y)\right| .$$
Furthermore, for $0<\epsilon\leq\epsilon'$, we have the following commutative diagram:
	\begin{center}
		\begin{tikzcd}
			\left| \VR_{ \epsilon}(X)\right|  \times \left|  \VR_{ \epsilon}(Y)\right|  \ar[d, "\cong" left] 
			\ar[r, hook]
			& 
			\left| \VR_{ \epsilon'}(X)\right|  \times \left|  \VR_{ \epsilon'}(Y)\right| 	
			\ar[d,"\cong"]
			\\
			\left|  \VR_{ \epsilon}(X\times Y)\right| 
			\ar[r,hook]
			& 
			\left|  \VR_{ \epsilon'}(X\times Y)\right| .
		\end{tikzcd}
	\end{center} 
\end{proposition}

Recall from Proposition 4.2 of \cite{Hatcher01(AT)} the fact that $\pi_n(X\times Y,(x_0,y_0))\cong \pi_n(X,x_0)\times \pi_n(Y,y_0).$ Thus, by applying the homotopy group functor to the above commuting diagram where all maps preserve basepoints, we obtain the following corollary:

\begin{corollary}\label{cor:prod of ppi} Let $(X,x_0)$ and $(Y,y_0)$ be pointed metric spaces. There is a natural isomorphism of persistent groups:
$$\ppi_n^{\VR}(X\times Y,(x_0,y_0))\cong \ppi_n^{\VR}(X,x_0)\times \ppi_n^{\VR}(Y,y_0),$$ 
the product of $\ppi_n^{\VR}(X,x_0)$ and $\ppi_n^{\VR}(Y,y_0)$. 
\begin{equation}\label{eq:prod of ppi_1}
    \ppi_1(X\times Y,(x_0,y_0))\cong \ppi_1(X,x_0)\times \ppi_1(Y,y_0).
\end{equation}
\end{corollary}

\begin{example}\label{ex:product of S1} Consider the torus $\mathbb{S}^1(r_1)\times \mathbb{S}^1(r_2):=\mathbb{S}^1(r_1)\times \mathbb{S}^1(r_2)$, where $0<r_1\leq r_2$ and $\mathbb{S}^1(r_1)$ and $ \mathbb{S}^1(r_2)$ are both endowed with their corresponding geodesic metrics. Then,  
\[\ppi_1\left( \mathbb{S}^1(r_1)\times \mathbb{S}^1(r_2)\right)=\Z\left( 0,\tfrac{2\pi}{3}\,r_1\right) \times \Z\left( 0,\tfrac{2\pi}{3}\,r_2\right)=(\Z\times\Z)\left( 0,\tfrac{2\pi}{3}\,r_1\right) \times \Z\left[ \tfrac{2\pi}{3}\,r_1 ,\tfrac{2\pi}{3}\,r_2\right),.\]

When $r_1=r_2=1$, we write $\bbT^2:=\bbT^2(1,1)$ and apply Example \ref{ex:circle-VR-ppi_n} to compute: 
\[\ppi_n^{\VR}(\bbT^2)\cong\begin{cases} 
 (\Z\times\Z)\left(0,\,\tfrac{2\pi}{3}\right),&\mbox{if $n=1,$}\\[6pt]
 (\Z^{\times \mathfrak{c}}\times\Z^{\times \mathfrak{c}})\left[\tfrac{2\pi}{3},\,\tfrac{2\pi}{3}\right],&\mbox{if $n=2$.}
 \end{cases}\]
 
As for the persistent homology groups of $\bbS^1\times \bbS^1$, because $\pH_i^{\epsilon}(\bbS^1)$ is torsion free for each $i\geq 0$ and $\epsilon>0$, the K\"{u}nneth Theorem (cf. Theorem 3B.5 of \cite{Hatcher01(AT)}) can be applied to calculate $\pH_n^{\VR,\epsilon}(\bbT^2)\cong \bigoplus_{i=0}^{n}\pH_i^{\epsilon}(\bbS^1)\times \pH_{n-i}^{\epsilon}(\bbS^1)$. On the other hand, if the homology groups $\pH_n^{\VR,\epsilon}(\bbT^2;\R)$ are computed with coefficients in $\R$, then for all integers $k\geq 0$, the corresponding undecorated persistence diagrams of $\bbT^2$ are (cf. Example 4.3 of \cite{lim2020vietorisrips}):
\begin{align*}
    \dgm_{2k+1}(\bbT^2)&=\left\{\left(\tfrac{k}{2k+1}\,2\pi,\,\tfrac{k+1}{2k+3}\,2\pi\right),\left(\tfrac{k}{2k+1}\,2\pi,\,\tfrac{k+1}{2k+3}\,2\pi\right)\right\},\\
    \dgm_{4k+2}(\bbT^2)&=\left\{\left(\tfrac{k}{2k+1}\,2\pi,\,\tfrac{k+1}{2k+3}\,2\pi\right)\right\},\\   
    \dgm_{4k+4}(\bbT^2)&=\emptyset.
\end{align*}
\end{example}

\subsubsection{Persistent fundamental groups under wedge sums}\label{sec:ppi_1-wedge}
		
The wedge sum $X\vee Y$ of two pointed metric spaces $X$ and $Y$ is the quotient space of the disjoint union of $X$ and $Y$ by the identification of basepoints $x_0\sim y_0$: ${\displaystyle X\vee Y=(X\amalg Y)\;/{\sim }}$. Denote the resulting basepoint of $X\vee Y$ by $z_0$. Let the wedge product be equipped with the \emph{gluing metric} (see \cite{AAGGPSWWZ17}):
	$$d_{X\vee Y}(x,y):= d_X(x,x_0)+d_Y(y,y_0),\forall x\in X, y\in Y$$ 
and $d_{X\vee Y}|_{X\times X}=d_X,d_{X\vee Y}|_{Y\times Y}=d_Y$. Notice that the above definition can be generalized to the case when we glue $n$ different pointed metric spaces $(X_1,x_1),\cdots,(X_n,x_n)$ by identifying $x_i\sim x_j$ for all $1\leq i,j\leq n$. We denote the resulting space by $\bigvee_{i=1}^n (X_i,x_i)$, or $\bigvee_{i=1}^n X_i$ for simplicity, with the gluing metric:
	$$d_{\bigvee_{i=1}^n X_i}(y_i,y_j):= d_{X_i}(y_i,x_i)+d_{X_j}(y_j,x_j),\forall i\neq j, y_i\in X_i, y_j\in X_j$$ 
and $d_{X_i\vee X_i}|_{X_i\times X_i}=d_{X_i}$ for $1\leq i\leq n$.

\begin{proposition}[Proposition 3.7, \cite{AAGGPSWWZ17}]\label{prop:wedge sum of VR} Let $X$ and $Y$ be pointed metric spaces. For each $\epsilon>0$, we have the basepoint preserving homotopy equivalence $$\left| \VR_{ \epsilon}(X)\right|  \vee \left|  \VR_{ \epsilon}(Y)\right| \xrightarrow{\cong} \left|  \VR_{ \epsilon}(X\vee Y)\right| .$$
Furthermore, for $0<\epsilon\leq\epsilon'$, we have the following commutative diagram:
	\begin{center}
		\begin{tikzcd}
			\left| \VR_{ \epsilon}(X)\right|  \vee \left|  \VR_{ \epsilon}(Y)\right|  \ar[d, "\cong" left] 
			\ar[r, hook]
			& 
			\left| \VR_{ \epsilon'}(X)\right|  \vee \left|  \VR_{ \epsilon'}(Y)\right| 	
			\ar[d,"\cong"]
			\\
			\left|  \VR_{ \epsilon}(X\vee Y)\right| 
			\ar[r,hook]
			& 
			\left|  \VR_{ \epsilon'}(X\vee Y)\right| .
		\end{tikzcd}
	\end{center} 
\end{proposition}

By applying the fundamental group functor to the above commuting diagram, where all maps preserve basepoints, we obtain a corollary about the persistent fundamental group of the wedge sum of metric spaces.
\begin{corollary}\label{cor:join of ppi}  
Let $(X,x_0)$ and $(Y,y_0)$ be pointed metric spaces, and let $z_0$ be the basepoint in $X\vee Y$ obtained by identifying $x_0$ and $y_0$. There is a natural isomorphism of persistent groups:
$$\ppi_1(X\vee Y,z_0)\cong \ppi_1(X,x_0)\ast \ppi_1(Y,y_0),$$ 
the coproduct of $\ppi_1(X)$ and $\ppi_1(Y)$ given by componentwise free product. 
\end{corollary}

\begin{proof} This follows immediately from Theorem \ref{thm:iso of d.f.g}, Proposition \ref{prop:wedge sum of VR} and the van-Kampen theorem (see Theorem 1.20 of \cite{Hatcher01(AT)}).
\end{proof}

\begin{remark} Via the isomorphism between $\ppi_n^{\VR}$ and $\ppi_n^{\rmK}$, Corollary \ref{cor:prod of ppi} and Corollary \ref{cor:join of ppi} can both be derived in a much simpler way using Kuratowski filtrations. See the proof of Theorem 4.1 of \cite{lim2020vietorisrips} for details.
\end{remark}

\begin{example}[Bouquets of circles]\label{ex:bouquet of circles} Let $0=:r_0<r_1< r_2<\cdots< r_n$ and let $k_1,k_2,\cdots,k_n$ be positive integers. Consider a \emph{bouquet of circles} given as follows: let 
\[C = \left(\bigvee^{k_1}\mathbb{S}^1(r_1)\right)\vee\left(\bigvee^{k_2}\mathbb{S}^1(r_2)\right)\vee\cdots \vee \left(\bigvee^{k_n}\mathbb{S}^1(r_n)\right)\] by gluing one point from each circle together. Then, 
\[\ppi_1^{\epsilon}(C)=\begin{cases}
		\Z^{\ast (k_i+\cdots+k_n)},&\mbox{if $ \tfrac{2\pi}{3}r_{i-1} \leq \epsilon< \tfrac{2\pi}{3}r_{i}$, for $i=1,\cdots,n$.}\\ 
		0,&\mbox{if $\epsilon\geq \tfrac{2\pi}{3}r_n$.}
		\end{cases},\]
where $\ppi_1^{0}(C):=\pi_1(C)=\Z^{\ast (k_1+\cdots+k_n)}$. In addition,
\[\ppi_1(C)=\coprod_{i=1}^{n} \Z^{\ast k_i}\left[0,\tfrac{2\pi}{3}r_{i}\right) =\coprod_{i=1}^{n} \Z^{\ast (k_i+\cdots+k_n)}\left[\tfrac{2\pi}{3}r_{i-1} ,\tfrac{2\pi}{3}r_{i}\right).\]
Notice that the number of free generators of $\ppi_1^{\epsilon}(C)$ can be represented by a staircase function:
	\[f_C(\epsilon):=\begin{cases}
	k_i+\cdots+k_n,&\mbox{if $\epsilon\in \left[\tfrac{2\pi}{3}r_{i-1}, \tfrac{2\pi}{3}r_{i}\right)$, for $i=1,\cdots,n$.}\\ 
	0,&\mbox{if $\epsilon\in \left[\tfrac{2\pi}{3}r_n,+\infty\right)$.}
	\end{cases}\] 

Taking $r_1=1,r_2=2$ and $r_3=3$, we obtain the function $f_C$ shown in Figure. \ref{fig:bounquet of circles}
\begin{figure}[ht!]	  
\centering
	\begin{tikzpicture}[scale=0.4]
    \draw[line width=0.5mm] (0,3) circle [blue, radius=1];
    \draw[line width=0.5mm] (0,4) circle [blue, radius=2];
    \draw[line width=0.5mm] (0,5) circle [blue, radius=3];
    \filldraw (0,2) [color=blue] circle[radius=1.5pt];
    \node[below=1pt of {(0,1.5)}, outer sep=1.5pt,fill=white] {$C$};
    \end{tikzpicture} \hspace{2cm}
    \begin{tikzpicture}
    \begin{axis} [ 
    height=4cm,
    axis y line=left, 
    axis x line=middle,
    xlabel=$\epsilon$,
    ytick={0,1,2,3},
    xtick={0,1,2,3},
    xticklabels={0,$\tfrac{2\pi}{3}$,$\tfrac{4\pi}{3}$,$\tfrac{6\pi}{3}$},
    xmin=0, xmax=4.5,
    ymin=0, ymax=3.5,]
    \addplot +[mark=none,color=blue,dashed] coordinates {(1,3) (1,-3)};
    \addplot +[mark=none,color=blue,dashed] coordinates {(2,2) (2,-2)};
    \addplot +[mark=none,color=blue,dashed] coordinates {(3,1) (3,-1)};
    \addplot[domain=2:3,color=blue,ultra thick]{1};
    \node[mark=none,color=blue,right] at (axis cs:2.5,2.5){$f_C$};
    \addplot[domain=1:2,color=blue,ultra thick]{2};
    \addplot[domain=0:1,color=blue,ultra thick]{3};
    \addplot[domain=3:4.5, color=blue,ultra thick]{0};
    \end{axis}
    \end{tikzpicture}
    \caption{Bouquets of three circles with $r_1=1,r_2=2$ and $r_3=3$ (left) and the corresponding function $f_C$ (right) representing the number of free generators of $\ppi_1^{\epsilon}(C)$.} \label{fig:bounquet of circles}
\end{figure}
\end{example}
	
\subsubsection{Relation to persistent homology groups: a persistent Hurewicz theorem}\label{sec:persistent-hurewicz}
	
In \cite{barcelo2014discrete}, Barcelo et al. established a slightly different definition of $\epsilon$-homotopy, where an $\epsilon$-homotopy between two $\epsilon$-chains $\gamma_1$ and $\gamma_2$ (of the same size $n$) is an $\epsilon$-Lipschitz map $H:([n]\times [m],\ell^{1})\rightarrow (X,d_X)$, instead of equipping $[n]\times [m]$ with the $\ell^{\infty}$ metric as we do in this paper. They proved a discrete version of the Hurewicz theorem under the $\ell^1$ metric in \cite{barcelo2014discrete}. Via a similar argument, we obtain the following:

\begin{restatable} [Discrete Hurewicz theorem]{theorem}{discreteHurewicz}
\label{thm:hurewicz-discrete} Let $\epsilon>0$. Let $(X,x_0)$ be a pointed $\epsilon$-connected metric space. Then, there is a surjective group homomorphism 
	$$\rho_{\epsilon}: \pi_1^{\epsilon}(X,x_0)\twoheadrightarrow \pH_1^{\VR,\epsilon}(X):=\pH_1^{\VR,\epsilon}(X;\Z)$$ such that $\ker(\rho_{\epsilon})=[\pi_1^{\epsilon}(X,x_0),\pi_1^{\epsilon}(X,x_0)]$, the commutator group of $\pi_1^{\epsilon}(X,x_0)$.
\end{restatable}

By checking that $\rho:=\{\rho_{\epsilon}\}_{\epsilon>0}$ is a homomorphism from $\ppi_1(X)$ to $\pH_1^{\VR}(X)$ (see page \pageref{pf:hurewicz-persistent}) and applying the isomorphism theorem of persistent fundamental groups, we obtain the persistent Hurewicz theorem:
\persistentHurewicz*
 
\begin{myproof}{Theorem \ref{thm:hurewicz-discrete}} 
\label{pf:hurewicz-persistent}
Here we adapt the proof of Theorem 4.1 from \cite{barcelo2014discrete}. For simplicity, we fix $\epsilon$ and write $\rho$ for $\rho_{\epsilon}.$ Also, we write $d$ for $d_X$ and omit $\ast$ for the concatenation of discrete loops.

Let $[\gamma]\in \pi_1^{\epsilon}(X,x_0)$ and choose a representative $\gamma=x_0x_1\cdots x_{n+1}$ with $x_{n+1}=x_0.$ For each $i\in [n]$, since $d_X(x_i,x_{i+1})\leq \epsilon$, $\sigma_i:=\{x_i,x_{i+1}\}$ is an $1$-simplex in $\VR_{\epsilon }(X)$, implying that $\sum_i\sigma_i\in C_1(\VR_{\epsilon }(X);\Z).$ Note that $\sum_i\sigma_i\in \ker(\partial_1)$, because $$\partial_1\left( \sum_i\sigma_i\right)=(x_1-x_0)+(x_2-x_1)+\cdots+(x_{n+1}-x_n)=0.$$
Define $\tilde{\rho}(\gamma)=\sum_i\sigma_i$. Let $p:\ker(\partial_1)\rightarrow\pH_1^{\VR,\epsilon}(X)$ be the canonical projection, and define $$\rho([\gamma]):=p(\tilde{\rho}(\gamma)).$$ 

\begin{claim} $\rho$ is well-defined.\end{claim} 

It suffices to prove that $\rho$ is well-defined under basic moves. Let $\gamma=x_0x_1\cdots x_{n}x_0$. Suppose that by removing some point $x_i$ ($1\leq i\leq n$), we obtain an $\epsilon$-chain $\gamma_i=x_0x_1\cdots x_{i-1}x_{i+1}\cdots x_{n}x_0$. Then, $$\tilde{\rho}(\gamma)-\tilde{\rho}(\gamma_i)=\{x_{i-1},x_i\}+\{x_{i},x_{i+1}\}-\{x_{i-1},x_{i+1}\}=\partial_2(\{x_{i-1},x_i,x_{i+1}\}).$$ 
Because $\gamma_i$ is an $\epsilon$-chain, $d_X(x_{i-1},d_{i+1})\leq \epsilon$, and thus $\{x_{i-1},x_i,x_{i+1}\}\in C_2\left(\VR_{\epsilon}(X);\Z\right)$. It follows that $\tilde{\rho}(\gamma)-\tilde{\rho}(\gamma_i)\in \im (\partial_2)$ and $\rho(\gamma)=\rho(\gamma_i).$
		
If adding $y_i$ results a new $\epsilon$-chain $\gamma_i'=x_0x_1\cdots x_{i-1}y_ix_{i}\cdots x_{n}x_0$, we can apply similar arguments to obtain $\tilde{\rho}(\gamma)-\tilde{\rho}(\gamma_i')=\partial_2(\{ x_{i-1},y_i,x_{i}\})\in \im(\partial_2)$ and thus $\rho(\gamma)=\rho(\gamma_i').$
		
\begin{claim} $\rho$ is a group homomorphism.\end{claim} 

Let $\gamma_1=x_0x_1\cdots x_{n}x_{n+1}$ and $\gamma_2=y_0y_1\cdots y_{m}y_{m+1}$ where $x_0=x_{n+1}=y_0=y_{m+1}$. Then, $$\tilde{\rho}(\gamma_1\gamma_2)=\sum_{i=0}^n\{x_i,x_{i+1}\}+\sum_{j=0}^m\{y_j,y_{j+1}\}=\tilde{\rho}(\gamma_1)+\tilde{\rho}(\gamma_2)\implies \rho(\gamma_1\gamma_2)=\rho(\gamma_1)+\rho(\gamma_2).$$
		
\begin{claim} $\rho$ is surjective.\end{claim} 

Let $\lambda\in \ker(\partial_1)\subset C_1(\VR_{\epsilon }(X);\Z)$. Then, $\lambda$ can be written as $\sum_{i=1}^k n_i\sigma_i$ for some $n_i\in \Z$ and $\sigma_i=\{x_i,y_i\}$ such that $d(x_i,y_i)\leq \epsilon$ for all $i$. Since $\partial_1(\lambda)=0$, we have
		\begin{equation}\label{eq:surjcoef}
		\sum_{i=1}^k n_i(y_i-x_i)=0.
		\end{equation} 
Let $S=\{x_i,y_i:i=1,\cdots,k\}$. Given $q\in S$, 
		\begin{itemize}
			\item let $m_q$ be the sum of coefficients of $q$ in Eq. (\ref{eq:surjcoef}). Note that $m_q=0$.
			\item let $\beta_q$ be an $\epsilon$-chain from $x_0$ to $q$ in $X.$
		\end{itemize}
For each $i$, define $\eta_i:=\beta_{x_i}\sigma_i\beta_{y_i}^{-1}.$ Clearly, every $\eta_i$ is an $\epsilon$-loop. Let $\gamma=\eta_1^{n_1}\eta_2^{n_2}\cdots \eta_k^{n_k}$, which is again an $\epsilon$-loop. In addition, we have $\rho([\gamma])=\lambda$, because 
		\begin{eqnarray*}
			\tilde{\rho}(\gamma)&=&\sum_{i=1}^k n_i \left(\tilde{\rho}(\beta_{x_i})+\sigma_i-\tilde{\rho}(\beta_{y_i})\right)\\
			&=&\sum_{i=1}^k n_i\sigma_i+\sum_{i=1}^k n_i\left(\tilde{\rho}(\beta_{x_i})-\tilde{\rho}(\beta_{y_i})\right)\\
			&=&\sum_{i=1}^k n_i\sigma_i+\sum_{q\in S} m_q\tilde{\rho}(\beta_q)\\
			&=&\sum_{i=1}^k n_i\sigma_i=\lambda.
		\end{eqnarray*}
		
\begin{claim} $\ker(\rho)=[\pi_1^{\epsilon}(X,x_0),\pi_1^{\epsilon}(X,x_0)]$.\end{claim} 

By the previous claims, $\im(\rho)= \pH_1^{\VR,\epsilon}(X)$ is Abelian. Since $\im(\rho)\cong \pi_1^{\epsilon}(X,x_0)/\ker(\rho)$, it must be true that $[\pi_1^{\epsilon}(X,x_0),\pi_1^{\epsilon}(X,x_0)]\subset \ker(\rho)$. It remains to prove the reverse inclusion. Let $\gamma=v_0v_1\cdots v_r v_{r+1}$, with $v_0=v_{r+1}=x_0$, be an $\epsilon$-loop such that $\gamma\in\ker(\rho)$. It follows from $\rho(\gamma)=0$ that there exists $\sigma=\sum_{i=1}^k n_i\sigma_i$ with $\sigma_i=\{x_i,y_i,z_i\}\in C_2\left(\VR_{\epsilon}(X);\Z\right)$ such that 
		\begin{equation}\label{eq:kercoef}
		\gamma=\partial_2(\sigma)=\sum_{i=1}^k n_i \left(\{x_i,y_i\}+\{z_i,x_i\}+\{y_i,z_i\}\right).
		\end{equation}
Denote $\sigma_i^1=\{x_i,y_i\},\sigma_i^2=\{y_i,z_i\}$ and $\sigma_i^3=\{z_i,x_i\}.$ 
		\begin{itemize}
			\item Let $L:=\{\sigma_i^j:i=1,\cdots,k;j=1,2,3\}$. Given $\zeta\in L$, let $m_{\zeta}$ be the sum of coefficients of $\zeta$ in Eq. (\ref{eq:kercoef}).
			\item Let $S$ be the set of endpoints of all $\sigma_i^j$. For each $q\in S$, let $\beta_q$ be an $\epsilon$-chain from $x_0$ to $q$ in $X.$
		\end{itemize}
For each $i$, define $\eta_i:=(\beta_{x_i}\sigma_i^1\beta_{y_i}^{-1})(\beta_{y_i}\sigma_i^2\beta_{z_i}^{-1})(\beta_{z_i}\sigma_i^3\beta_{x_i}^{-1})=\beta_{x_i}\sigma_i^1\sigma_i^2\sigma_i^3\beta_{x_i}^{-1}$, which is an $\epsilon$-chain. Assume $\beta_{x_i}=x_0w_1\cdots w_l x_i$. Since $\{x_i,y_i,z_i\}\in C_2\left(\VR_{\epsilon}(X);\Z\right)$,
		\begin{eqnarray*}
			\eta_i&=&x_0w_1\cdots w_l x_ix_iy_iy_iz_iz_ix_iw_l\cdots w_1x_0\\
			&\sim_1^{\epsilon}&x_0w_1\cdots w_l x_iy_iz_ix_iw_l\cdots w_1x_0\\
			&\sim_1^{\epsilon}&x_0w_1\cdots w_l x_iy_ix_iw_l\cdots w_1x_0\\
			&\sim_1^{\epsilon}&x_0w_1\cdots w_l x_ix_iw_l\cdots w_1x_0\\
			&\sim_1^{\epsilon}& x_0.
		\end{eqnarray*}
Let $\eta=\eta_1^{n_1}\eta_2^{n_2}\cdots \eta_k^{n_k}$, which is $\epsilon$-homotopic to the constant loop $x_0$. 
		
Recall that $\gamma=v_0v_1\cdots v_r v_{r+1}$. Let $\tau_i=\{v_i,v_{i+1}\}$ for $i\in [r]$ and let $\tau =\prod_{i=0}^{r}\beta_{\tau_i(0)}\tau_i\beta_{\tau_i(1)}^{-1}$, which is $\epsilon$-homotopic to $\gamma.$ Therefore, $\gamma \sim_1^{\epsilon} \tau\eta^{-1}.$ For each $\zeta\in L$, the loop $\beta_{\zeta(0)}\zeta\beta_{\zeta(1)}^{-1}$ appears $m_{\zeta}$ times in $\tau$, and $-m_{\zeta}$ times in $\eta^{-1}$. So each $\beta_{\zeta(0)}\zeta\beta_{\zeta(1)}^{-1}$ appears in $\tau\eta^{-1}$ with exponent adding up to zero, so $[\gamma]_\epsilon=[\tau\eta^{-1}]_\epsilon\in [\pi_1^{\epsilon}(X,x_0),\pi_1^{\epsilon}(X,x_0)]$. 
\end{myproof}
 
\begin{remark}By the isomorphism of persistent fundamental groups (cf. Theorem \ref{thm:iso-ppi_1's}), Theorem \ref{thm:hurewicz-discrete} follows from the standard Hurewicz theorem. 
\end{remark}

\section{Dendrograms and a Metric on \texorpdfstring{$\pi_1(X,x_0)$}{pi1}}

In analogy to $0$-th persistent homotopy, which has a natural `tree structure', in this section we construct a dendrogram from the persistent fundamental group. 

\subsection{Discretization theorem} 
\label{subsec:discretization}
	
Let $\gamma:[0,1]\rightarrow X$ be a continuous path and $\epsilon>0$. An \emph{$\epsilon$-chain along $\gamma$} is an $\epsilon$-chain $x_0 x_1\cdots x_n$ where there exists a partition $\{0=t_0,\cdots,t_n=1\}$ of $[0,1]$ such that each $x_i=\gamma(t_i)$. A \emph{strong $\epsilon$-chain along $\gamma$} is an $\epsilon$-chain along $\gamma$ such that $\gamma([t_{i-1},t_i])\subset B(\gamma(t_{i-1}),\epsilon)$ for each $i$. When $\gamma$ is a continuous loop, a (strong) $\epsilon$-chain along $\gamma$ is also called a \emph{(strong) $\epsilon$-loop along $\gamma$}. 
The following lemma permits relating the fundamental group of a space to its discrete fundamental groups.
\begin{lemma}[{Lemma 3.1.7, \cite{wilkins2011discrete}}]\label{lem:epsilon-discret} 
	Let $X$ be a chain-connected metric space, and let $\epsilon>0$ be given. Then, the following map is a well-defined homomorphism (called the $\epsilon$-discretization homomorphism) $$\Phi_{\epsilon}:\pi_1(X)\rightarrow \pi_1^{\epsilon}(X)\text{ with }[\gamma]\mapsto [\alpha]_{\epsilon},$$ where $\alpha$ is a strong $\epsilon$-loop along $\gamma$. If the $\epsilon$-balls of X are path-connected (e.g. $X$ is geodesic or locally path connected), then $\Phi_{\epsilon}$ is surjective.
\end{lemma}
	
\begin{example}\label{ex:s.c.-space} Let $Y$ be any simply connected compact space. By Lemma \ref{lem:epsilon-discret} and the fact $\pi_1(Y)=0$, we have
	\[ \pi_1^{\epsilon}(Y)=0,\forall\epsilon>0.	\] 	
\end{example}

Theorem 3.2 of \cite{vigolo2018fundamental} states that for uniformly locally path connected and uniformly semi-locally simply connected metric spaces, the $\epsilon$-discretization homomorphism is in fact an isomorphism for small enough $\epsilon$, albeit using a slightly different definition for $\epsilon$-homotopy. A similar result and proof apply to our case as well. 

\begin{restatable}[Discretization theorem]{theorem}{discretization} \label{thm:discretization} If a metric space $(X,d_X)$ is u.l.p.c. and u.s.l.s.c., then the $\epsilon$-discretization homomorphism $\Phi_{\epsilon}:\pi_1(X)\xrightarrow{\cong} \pi_1^{\epsilon}(X)$ is an isomorphism for $\epsilon$ small enough.

If $X$ is geodesic and $\operatorname{LGC_1}(\rho,R)$ (see Definition \ref{def:lgc}), 
then $\Phi_{\epsilon}$ is an isomorphism for any $\epsilon\geq 0$ such that 
$\rho(2\epsilon)\leq R.$ 
\end{restatable}
	
\begin{proof} Because $X$ is locally path-connected, it follows from Lemma \ref{lem:epsilon-discret} that $\Phi_{\epsilon}$ is surjective for any $\epsilon$. Since $X$ is u.s.l.s.c., there exists $\delta'>0$ such that a loop in $B(x,\delta')$ is null-homotopic for any $x\in X$. Because $X$ is u.l.p.c., there exists $\delta<\delta'/2$ so that any two points in $B(x,\delta)$ can be connected by a path with image completely contained in $B(x,\delta'/2)$, for any $x\in X$. Fix a parameter value $0<\epsilon<\delta$. 

We now prove that  $\Phi_{\epsilon}$ is injective. Let $\gamma$ be a continuous loop based at $x_0$ whose $\epsilon$-discretization $\alpha=x_0 x_1\cdots x_n$ is trivial. Recall that $\alpha$ is a strong $\epsilon$-path along $\gamma$, i.e., there exists a partition $\{t_0,\cdots,t_n\}$ of $[0,1]$ such that each $x_i=\gamma(t_i)$ and $\gamma([t_{i},t_{i+1}])\subset B(\gamma(t_{i}),\epsilon)$ for each $i$. We want to show $\gamma$ is null-homotopic.
		
Let $H:[n]\times [m]\to X$ be an $\epsilon$-homotopy between $ \alpha$ and $ x_0$. For any $0\leq i\leq n-1$ and $0\leq j\leq m$, the following vertical and horizontal pairs are within distance $\epsilon$: 
		\begin{center}
			\begin{tikzcd} 
				H(i,j) \arrow[dash,"\leq \epsilon"]{r} \arrow[dash,"\alpha_{i,j}",swap]{r}
				\arrow[dash,"\leq \epsilon",swap]{d} \arrow[dash,"\beta_{i,j}"]{d} 
				& 
				H(i,j+1) \arrow[dash,"\leq \epsilon"]{d} 
				\arrow[dash,"\beta_{i,j+1}",swap]{d} 
				\\
				H(i+1,j) \arrow[dash,"\leq \epsilon"]{r}
				\arrow[dash,"\alpha_{i+1,j}",swap]{r}
				& 
				H(i+1,j+1)
			\end{tikzcd}
		\end{center} 
Since $d_X\left( H(i,j),H(i,j+1)\right)\leq \epsilon<\delta,$ there exists a path $\alpha_{i,j}$ joining them with image in $B(H(i,j),\delta'/2)$. Similarly, for $0\leq i\leq n$ and $0\leq j\leq m-1$, there exists a path $\beta_{i,j}$ joining $H(i,j)$ and $H(i+1,j)$ with image in $B(H(i,j),\delta'/2)$. It follows from the triangle inequality that the concatenations $\beta_{i,j}\alpha_{i+1,j}$ and $\alpha_{i,j}\beta_{i,j+1}$ are both contained in $ B(H(i,j),\delta')$. Since $\alpha_{i,j}\beta_{i,j+1}\alpha^{-1}_{i+1,j}\beta^{-1}_{i,j}$ is a loop in $B(H(i,j),\delta')$, it is null-homotopic. For any $0\leq i\leq n-1$ and $0\leq j\leq m-1$, let $\xi_{i,j}:=\alpha_{0,0}\cdots\alpha_{0,j-1}\beta_{0,j-1}\cdots \beta_{i-1,j}$, a path joining $x_0$ to $H(i,j)$, and then define $\eta_{i,j}:=\xi_{i,j}\alpha_{i,j}\beta_{i,j+1}\alpha^{-1}_{i+1,j}\beta^{-1}_{i,j}\xi_{i,j}^{-1}$, which is a null-homotopic loop at $x_0$ based on construction. Concatenating all the $\eta_{i,j}$ in an appropriate order results into a null-homotopic loop which is homotopic to $ \alpha_{0,0}\cdots\alpha_{n-1,0}=\gamma$. Thus, $\Phi_{\epsilon}$ is injective.

If $X$ is geodesic and $\operatorname{LGC_1}(\rho,R)$, we apply a similar argument as above with some technical changes. Choose $\epsilon\geq 0$ such that $\rho(2\epsilon)\leq R.$ 
Again, let $\alpha=x_0 x_1\cdots x_n$ be a trivial $\epsilon$-discretization of $\gamma$, and let $H:[n]\times [m]\to X$ be an $\epsilon$-homotopy between $ \alpha$ and $ x_0$. For any $0\leq i\leq n-1$ and $0\leq j\leq m$, because $X$ is geodesic, there exist a geodesic $\alpha_{i,j}$ joining $H(i,j)$ to $H(i,j+1)$ and a geodesic $\beta_{i,j}$ joining $H(i,j)$ to $H(i+1,j)$, both with image in $B\left(H(i,j),\epsilon\right)$. Thus, the concatenations $\beta_{i,j}\alpha_{i+1,j}$ and $\alpha_{i,j}\beta_{i,j+1}$ are both contained in $ B(H(i,j),2\epsilon)$. 
Because $X$ is $\operatorname{LGC_1}(\rho,R)$, $B\left(H(i,j),2\epsilon\right)$ is $1$-connected in $B\left(H(i,j),\rho(2\epsilon)\right).$ Thus, the loop $\alpha_{i,j}\beta_{i,j+1}\alpha^{-1}_{i+1,j}\beta^{-1}_{i,j}$ is null-homotopic, given that $\rho(2\epsilon)\leq R.$ We construct $\xi_{i,j}$ and $\eta_{i,j}$ as before, and see that concatenating all the $\eta_{i,j}$ gives a null-homotopic loop that is homotopic to $ \alpha_{0,0}\cdots\alpha_{n-1,0}=\gamma$. 
\end{proof}

The well-definedness of $\Phi_{\epsilon}$ indicates that any two $\epsilon$-discretizations of a continuous loop are $\epsilon$-homotopic. It follows that $\Phi_{\epsilon,\epsilon'}\circ \Phi_{\epsilon}=\Phi_{\epsilon'}$ for all $\epsilon\leq\epsilon'$. By the universal property of limit, there is a natural group homomorphism
	$$ \Phi:\pi_1(X)\rightarrow \lim \ppi_1(X) \text{ with }[\gamma]\mapsto \lim[\alpha]_{\epsilon}.$$
It follows from Theorem \ref{thm:discretization} that the following holds:
\begin{corollary} If a metric space $(X,d_X)$ is u.l.p.c. and u.s.l.s.c., then the homomorphism $\Phi:\pi_1(X)\rightarrow \lim \ppi_1(X)$ is an isomorphism. 
\end{corollary}

\subsection{Treegrams and a pseudo-metric on \texorpdfstring{$\caL(X,x_0)$}{loop space}} \label{sec:treegram-pseudo-metric}
By analogy with the ultra-metric $\mu^{(0)}$ discussed in \textsection \ref{sec:p-set}, we introduce a pseudo-ultra-metric on the set of all discrete loops on a given metric space. 
In dimension $0$, one cares about how points cluster to each other along the scale parameters. To apply a similar idea to dimension $1$, the first attempt we take is to study how loops merge with each other, by identifying the first scale parameter $\epsilon$ at which two discrete loops become $\epsilon$-homotopic.
In particular, we view the discrete loop space of a metric space as a metric space in itself:
	
\begin{proposition}\label{prop:metric mu^1} Given a pointed metric space $(X,x_0)$, the following defines a pseudo-ultra-metric on $\caL(X,x_0)$: for any $\gamma,\gamma'\in \caL(X,x_0)$,
		\[\mu_X^{(1)}(\gamma,\gamma'):=\inf\{\epsilon>0:\gamma\sim_1^{\epsilon}\gamma'\}.\]
\end{proposition}
	
\begin{proof}Clearly, $\mu_X^{(1)}(\gamma,\gamma')=\mu_X^{(1)}(\gamma',\gamma)$ for all $\gamma,\gamma'\in \caL(X,x_0)$. It remains to prove the strong triangle inequality. Given arbitrary $\alpha,\beta,\gamma\in \caL(X,x_0)$, let $\epsilon_1=\mu_X^{(1)}(\alpha,\beta)$ and $\epsilon_2=\mu_X^{(1)}(\beta,\gamma)$. For each $\delta_1>\epsilon_1$ and $\delta_2>\epsilon_2$, there exist a sequence of $\delta_1$-paths $H_1=\{\alpha=\gamma_0,\gamma_1,\cdots,\gamma_{k-1},\gamma_k=\beta\}$ and a sequence of $\delta_2$-paths $H_2=\{\beta=\gamma_{k+1},\gamma_{k+2},\cdots,\gamma_{k+l-1},\gamma_{k+l}=\gamma\}$ such that each $\gamma_i$ differs from $\gamma_{i-1}$ by a basic move. Then, $H=\{\alpha=\gamma_0,\gamma_1,\cdots,\gamma_{k+l-1},\gamma_{k+l}=\gamma\}$ is a sequence of $\max\{\delta_1,\delta_2\}$-paths, where each adjacent two paths differ by a basic move. Then, $\alpha$ and $\gamma$ are $\max\{\delta_1,\delta_2\}$-homotopic via $H$. Thus, $\mu_X^{(1)}(\alpha,\gamma)\leq\max\{\delta_1,\delta_2\}$. Letting $\delta_i\searrow \epsilon_i$ ($i=1,2$), we obtain $\mu_X^{(1)}(\alpha,\gamma)\leq\max\{\epsilon_1,\epsilon_2\}.$ 
\end{proof}

\begin{remark}\label{rmk:property-mu_1} For any $\gamma\in \caL(X,x_0),$ note that 
\[\birth(\gamma)=\mu_{X}^{(1)}(\gamma,\gamma)\text{ and }\death(\gamma)=\mu_{X}^{(1)}(\gamma,\{x_0\}).\]
Suppose $\alpha$ and $\beta$ are two discrete loops in $X$. Let $\epsilon=\max\{\birth(\alpha),\birth(\beta)\}$. Then, for any $\gamma\in \caL^{\epsilon}(X,x_0),$ we have 
$$\mu_X^{(1)}(\alpha\ast\gamma,\beta\ast\gamma)=\mu_X^{(1)}(\alpha,\beta).$$ This is because when $\delta\geq \epsilon$, there is a $\delta$-homotopy $H=\{\alpha,\cdots,\beta\}$ iff there is a $\delta$-homotopy $H\ast\gamma:=\{\alpha\ast\gamma,\cdots,\beta\ast\gamma\}$.
\end{remark}

\paragraph*{Stability of $\mu^{(1)}_{\bullet}$.} Let $(X,x_0)$ and $(Y,y_0)$ be pointed compact metric spaces. Let $R:X\xtwoheadleftarrow{\phi_X}Z\xtwoheadrightarrow{\phi_Y}Y$ be a pointed tripod between $X$ and $Y$. Note that $\phi_X$ and $\phi_Y$ induce surjective maps 
\[\caL(Z,z_0)\twoheadrightarrow \caL(X,x_0)\text{  and  } \caL(Z,z_0)\twoheadrightarrow \caL(Y,y_0)\] respectively, still denoted by $\phi_X$ and $\phi_Y$. Thus, we have a tripod between $\left( \caL(X,x_0),\mu_X^{(1)}\right)$ and $\left( \caL(Y,y_0),\mu_Y^{(1)}\right)$, denoted by $R_{\caL}$:
$$\caL(X,x_0)\xtwoheadleftarrow{\phi_X}\caL(Z,z_0)\xtwoheadrightarrow{\phi_Y}\caL(Y,y_0).$$ 
	
\begin{lemma}\label{lem:tripod} Each $R\in\mathfrak{R}^{\pt}((X,x_0),(Y,y_0))$ induces a tripod $R_{\caL}\in \mathfrak{R}(\caL(X,x_0),\caL(Y,y_0))$ with $\dis(R_{\caL})\leq \dis(R)$. In particular, if $\alpha_X$ is an $\epsilon$-chain in $X$ and $(\alpha_X,\alpha_Y)\in R_{\caL}$ (see page \pageref{para:tripod}), then $\alpha_Y$ is an $(\epsilon+\dis(R))$-chain in $Y$.
\end{lemma}

\begin{proof}  Let $R:X\xtwoheadleftarrow{\phi_X}Z\xtwoheadrightarrow{\phi_Y}Y$ be a pointed tripod between $X$ and $Y$. If $\alpha=x_0 x_1\cdots x_n$ is an $\epsilon$-loop in $X$, then there exists a discrete loop $\gamma=z_0\cdots z_n$ in $Z$ such that $\alpha=\phi_X(\gamma)$. Note that $\phi_Y(\gamma)$ is an $(\epsilon+\dis(R))$-loop in $Y$, because for $0\leq i\leq n-1$,
\begin{align*}	
&d_Y(\phi_Y(z_i),\phi_Y(z_{i+1}))\\
        \leq & d_X(\phi_X(z_i),\phi_X(z_{i+1}))+|d_Y(\phi_Y(z_i),\phi_Y(z_{i+1}))- d_X(\phi_X(z_i),\phi_X(z_{i+1}))|\\
			\leq &\epsilon+\dis(R).
\end{align*}
Let $(\alpha,\beta),(\alpha',\beta')\in R_{\caL}$ and $\delta>\mu_X^{(1)}(\alpha,\beta)$. Then, there is a finite sequence of $\delta$-loops in $X$ $$H_X=\{\alpha=\alpha_0,\alpha_1,\cdots,\alpha_{k-1},\alpha_k=\alpha'\}$$ such that each $\alpha_j$ differs from $\alpha_{j-1}$ by a basic move. For each $\alpha_j$, let $\beta_j=\phi_Y(\gamma_j)$ where $\gamma_j$ is such that $\alpha_j=\phi_X(\gamma_j)$. Then, each $\beta_j$ differs from $\beta_{j-1}$ by a basic move. In particular,
	$$H_Y=\{\beta=\beta_0,\beta_1,\cdots,\beta_{k-1},\beta_k=\beta'\}$$ is a $(\delta+\dis(R))$-homotopy between $\beta$ and $\beta'$. Therefore, $\mu_Y^{(1)}(\alpha',\beta')\leq \delta+\dis(R).$ Letting $\delta \searrow \mu_X^{(1)}(\alpha,\beta)$, we obtain $\mu_Y^{(1)}(\alpha',\beta')\leq \mu_X^{(1)}(\alpha,\beta)+\dis(R)$. Similarly, $\mu_X^{(1)}(\alpha,\beta)\leq \mu_Y^{(1)}(\alpha',\beta')+\dis(R)$. This is true for all $(\alpha,\beta),(\alpha',\beta')\in R_{\caL}$, implying that $\dis(R_{\caL})\leq \dis(R)$.
\end{proof}

\begin{theorem}[Stability theorem for $\mu^{(1)}_{\bullet}$]\label{thm:stab-mu1} 
Given pointed compact metric spaces $(X,x_0)$ and $(Y,y_0)$, 
		$$\dgh(\caL(X,x_0),\caL(Y,y_0))\leq \dgh^{\pt}((X,x_0),(Y,y_0)).$$
\end{theorem}
	
\begin{proof} By Lemma \ref{lem:tripod},
		\begin{eqnarray*}
			\dgh(\caL(X,x_0),\caL(Y,y_0))&=&\tfrac{1}{2}\inf\{\dis(R_{\caL}):R_{\caL}\in\mathfrak{R}(\caL(X,x_0),\caL(Y,y_0))\}\\
			&\leq &\tfrac{1}{2}\inf\{\dis(R_{\caL}):R_{\caL}\text{ induced by }R,R\in\mathfrak{R}(X,Y)\} \\
			&\leq &\tfrac{1}{2}\inf\{ \dis(R):R\in\mathfrak{R}(X,Y)\}\\
			&= & \dgh(X,Y).
		\end{eqnarray*}
\end{proof}

\begin{corollary}Given pointed compact metric spaces $(X,x_0)$ and $(Y,y_0)$, 
$$\inf_{(x_0,y_0)\in X\times Y} \dgh(\caL(X,x_0),\caL(Y,y_0))\leq \dgh(X,Y).$$
\end{corollary}

Though the metric $\mu^{(1)}_{\bullet}$ is stable, there is no promise that this stability helps us estimate the Gromov-Hausdorff distance between metric spaces. There is no good reason for one to believe that computing $\dgh(\caL(X,x_0),\caL(Y,y_0))$ is easier than $\dgh^{\pt}((X,x_0),(Y,y_0))$. 
Motivated by this concern, we identify some situations when the pseudo-ultra-metric $\mu^{(1)}_{\bullet}$ induces a treegram structure on the underlying space, in which case we will obtain an interleaving type stability (see Corollary \ref{cor:di-stab-subdendro}).

Given a pointed compact metric space $(X,x_0)$ and $\epsilon>0$, $\pi_1^{\epsilon}(X,x_0)=\caL^{\epsilon}(X,x_0)/\sim_1^{\epsilon}$ forms a partition of $\caL^{\epsilon}(X,x_0)\subset \caL(X,x_0)$, i.e., a subpartition of $\caL(X,x_0)$ (see page \pageref{para:subpartition}). We now show how a treegram (cf. Definition \ref{def:treegram}) can be induced on $\mathcal{L}(X,x_o)$ when $X$ is finite:

\begin{lemma}\label{lem:treegramofL} Let $X$ be a finite metric space. The map $\epsilon\mapsto \pi_1^{\epsilon}(X,x_0)$ induces a treegram over $\caL(X,x_0)$, which we denote by $\theta_{\caL(X,x_0)}^\rms$.
\end{lemma}

For a geodesic space $X$, Lemma \ref{lem:treegramofL} fails because the map $\epsilon\mapsto \pi_1^{\epsilon}(X,x_0)$ no longer satisfies the semi-continuity condition: for any $0<\epsilon<\epsilon'<\diam(X)$, there are always  $\epsilon'$-loops which  are not $\epsilon$-loops, implying that $\theta_{\caL(X,x_0)}^\rms(\epsilon)\neq \theta_{\caL(X,x_0)}^\rms(\epsilon')$.

\begin{myproof}{Lemma \ref{lem:treegramofL}} We want to check that conditions (1'), (2) and (3) in Definition \ref{def:treegram} are satisfied. For $\epsilon\leq \epsilon'$, $\caL^{\epsilon}(X,x_0)\subset\caL^{\epsilon'}(X,x_0)$ and $\theta_{\caL(X,x_0)}^\rms(\epsilon)$ refines $\theta_{\caL(X,x_0)}^\rms(\epsilon')|_{\caL^{\epsilon}(X,x_0)}$, i.e., condition (1') is satisfied. The semi-continuity, i.e., Condition (2) that for all $r$ there exists $\epsilon>0$ s.t. $\theta_{\caL(X,x_0)}^\rms(r)=\theta_{\caL(X,x_0)}^\rms(t)$ for all $t\in [r,r+\epsilon]$, follows from 
the fact $X$ is finite. 
Condition (3) holds because $\theta_{\caL(X,x_0)}^\rms(\epsilon)$ is the single block partition, whenever $\epsilon\geq \diam(X)$. 
\end{myproof}

Examples of treegrams arising from discrete fundamental groups of finite metric spaces are available in the appendix of \cite{memoli2019persistent} (and the extended preprint version of this paper). 

It is not hard to verify that $\mu_X^{(1)}=\mu_{\theta_{\caL(X,x_0)}^\rms}^\rms$, so we can apply Proposition \ref{prop:treegram} to conclude:
	\begin{corollary}
	\label{cor:di-stab-subdendro}
	Given pointed compact metric spaces $(X,x_0)$ and $(Y,y_0)$, $$\inf_{(x_0,y_0)\in X\times Y} \di \left(\caV_{\F}\circ\theta_{\caL(X,x_0)}^\rms, \caV_{\F}\circ\theta_{\caL(Y,y_0)}^\rms\right)\leq \dgh(X,Y).$$
\end{corollary}
In general, 
the computation of the Gromov-Hausdorff distance is NP-hard \cite{Schmiedl17}, while in the left-hand-side the interleaving distance $\di \big(\caV_{\F}\circ\theta_{\caL(X,x_0)}^\rms, \caV_{\F}\circ\theta_{\caL(Y,y_0)}^\rms\big)$ between persistence modules over the real-line is computable in polynomial time, which is a consequence of Theorem \ref{thm:di-db} and the fact that the bottleneck distance is computable in polynomial time \cite{cohen2007stability}. In general, the interleaving distance between persistence modules over $\mathbb{R}^d$ for $d>1$ is NP-hard \cite{bjerkevik2020computing}.

To obtain the semi-continuity condition for the map $\epsilon\mapsto \pi_1^{\epsilon}(X,x_0)$ in the case of geodesic spaces,
we invoke Plaut and Wilkins' results about the so-called set of  \emph{critical values} of discrete fundamental groups, i.e. the parameters at which the isomorphism type of the discrete fundamental group changes.  

\medskip
\paragraph*{Critical values and critical spectrum.} Let $X$ be a chain-connected metric space. A \emph{non-critical interval of $X$} is a non-empty open interval $I\subset\R_+$, such that for $\epsilon,\epsilon'\in I$ with $\epsilon<\epsilon'$, the map $\Phi_{\epsilon,\epsilon'}$ defined as Eq. (\ref{eq:Phieep})  
is bijective. We call a positive number $\epsilon$ a \emph{critical value of $X$} if it is not contained in any non-critical interval. We denote the subset of $\R_+$ consisting of all critical values of $X$ by $\mathrm{Cr}(X)$, and call this the \emph{critical spectrum of $X$} (see \cite{wilkins2011discrete,conant2012discrete}).
\label{para:critical value}

\begin{theorem}[Theorem 3.1.11 of \cite{wilkins2011discrete}]\label{thm:critical-value} Let $X$ be a compact geodesic space. Then, $\mathrm{Cr}(X)$ is discrete and bounded above in $\R_+$. 
\end{theorem}

In the next section, we will consider compact geodesic spaces satisfying some additional fairly mild conditions and apply the above theorem to obtain a dendrogram over the standard fundamental group $\pi_1(X,x_0)$ (instead of the set of all discrete loops).

\subsection{A dendrogram structure and its  ultra-metric on \texorpdfstring{$\pi_1(X,x_0)$}{pi1}}\label{subsec:dendrogram}
	
Let $\bbG=(\{G_{t}\}_{t>0},$ $\{f_{ts}\}_{t\leq s\in \R_+})$ be a persistent group, which is naturally an $(\R_+,\leq)$-shaped diagram in $\grp$. We can consider its limit, i.e., the group 
	$$G_0:= \lim \bbG=\left\{(a_t)_{t>0}\in \prod_{t>0} G_t:a_s=f_{ts}(a_t),\forall\, t\leq s\right\},$$
endowed with natural projections $p_t:G_0\to G_t$ picking out the $t$-th component. Since $p_s=f_{ts}\circ p_t$ for $t\leq s$, we have $\ker(p_t)\subset \ker(p_s)$. If $p_t$ is surjective, then $G_t\cong G_0/\ker(p_t)$ can be regarded as a partition of the limit $G_0$. If $p_t$ and $p_s$ are both surjective for some $t\leq s$, then, as partitions of $G_0$, $G_t$ refines $G_s$. 

By Lemma \ref{lem:epsilon-discret}, if a metric space $(X,d_X)$ is locally path-connected, then each $\Phi_{\epsilon}^X$ is surjective. In this case, every $\pi_1^{\epsilon}(X)$ can be regarded as a partition of $\pi_1(X)$. 

\dendrogram*
\begin{proof} 
We first verify that $\theta_{\pi_1(X)}$ is a generalized dendrogram by checking Condition (1), (2) and (3) of Definition \ref{def:dendrogram}. 
Condition (1) is immediate, because $\pi_1^\epsilon(X)$ refines $\pi_1^{\epsilon'}(X)$ as partitions of $\pi_1(X)$, for any $\epsilon\leq \epsilon'$. Condition (3) is also trivial, since $\theta_{\pi_1(X)}(\epsilon)$ is the single block partition, whenever $\epsilon\geq \diam(X)$. The semi-continuity, i.e., Condition (2) that for all $r$ there exists $\epsilon>0$ s.t. $\theta_{\pi_1(X)}(r)=\theta_{\pi_1(X)}(t)$ for all $t\in [r,r+\epsilon]$, follows from Theorem \ref{thm:critical-value} because $X$ is a compact geodesic space. 

By Theorem \ref{thm:discretization}, if $X$ is both u.l.p.c. and u.s.l.s.c., then $\pi_1(X)\cong  \lim \pi_1^{\epsilon}(X)$. Thus, the generalized dendrogram $\theta_{\pi_1(X)}$ becomes a dendrogram. 
\end{proof}

As an example, let us look at the dendrogram associated with $\bbS^1$ and determine the induced metric $\mu_{\theta_{\pi_1(\bbS^1)}}$.
\begin{example} [The ultra-metric $\mu_{\theta_{\pi_1(\bbS^1)}}$]
When $r=1$ in Example \ref{ex:circle-discrete}, associated to $\ppi_1(\bbS^1)$ we have a dendrogram over $\pi_1(\bbS^1)\cong \Z$ shown in Figure \ref{fig:dendrogram-pi1} on page \pageref{fig:dendrogram-pi1}. The metric $\mu_{\theta_{\pi_1(\bbS^1)}}$ induces an ultra-metric $d_X$ on $\Z$ given by
\[d_X(n,m)=\begin{cases}\tfrac{2\pi}{3},&\mbox{if $n\neq m$,}\\ 
		0,&\mbox{if $n=m$.}
\end{cases}\] 
\end{example}

For a more complicated example of the dendrogram induced by the persistent fundamental groups, let us consider the product of spaces. If $X$ and $Y$ are both geodesic, u.l.p.c. and u.s.l.s.c., so is $X\times Y$. 
Meanwhile, the product of dendrograms $\theta_{\pi_1(X)}\times \theta_{\pi_1(Y)}$ gives a dendrogram over $\pi_1(X)\times \pi_1(Y)\cong \pi_1(X\times Y)$, cf. page \pageref{para:prod of dendrogram}. In particular, the isomorphism given in Eq. (\ref{eq:prod of ppi_1}) implies that for each $\epsilon\geq 0$, the following forms a partition of $\pi_1(X\times Y):$
\begin{equation*}
\theta_{\pi_1(X\times Y)}(\epsilon)=\left( \theta_{\pi_1(X)}\times \theta_{\pi_1(Y)} \right)(\epsilon):=\begin{cases} \pi_1^{\epsilon}(X)\times \pi_1^{\epsilon}(Y), & \mbox{if $\epsilon>0$}\\
\pi_1(X)\times\pi_1(Y), & \mbox{if $\epsilon=0$,}
\end{cases}
\end{equation*}

\begin{example}[Dendrogram over $\pi_1(\mathbb{S}^1(r_1)\times \mathbb{S}^1(r_2))\cong \Z\times\Z$]
\label{ex:dendrogram: product of S1} 
Recall from Example \ref{ex:product of S1} the persistent fundamental group of the torus $\mathbb{S}^1(r_1)\times \mathbb{S}^1(r_2)$, where $0<r_1\leq r_2$. The dendrogram associated to $\ppi_1(\mathbb{S}^1(r_1)\times \mathbb{S}^1(r_2))$ over $\pi_1(\mathbb{S}^1(r_1)\times \mathbb{S}^1(r_2))\cong \Z\times\Z$ is described in Figure \ref{fig:dendrogram-T2}.
\begin{figure}[ht!]
\centering 
	\begin{tikzpicture}
    \begin{axis} [ 
    axis y line=left, 
    axis x line=none,
    xlabel=$\epsilon$,
    ylabel=$\pi_1(\bbT^2(r_1\text{,}r_2))$,
    ytick={-1,0,1},
    yticklabels={-1,0,1},
    xtick={0,1,1.5},
    xticklabels={0,$\tfrac{2\pi}{3}r_1$,$\tfrac{2\pi}{3}r_2$},
    xmin=0, xmax=2,
    ymin=-1.8, ymax=1.8,]
    \addplot +[mark=none] coordinates {(1,0) (1,0.07)};
    \addplot +[mark=none,color=blue, ultra thick, opacity=0.4] coordinates {(1.5,1.7) (1.5,-1.7)};
    \addplot +[mark=none,color=blue,ultra thick, opacity=0.4] coordinates {(1,-0.37) (1,0.37)};
    \addplot +[mark=none,color=blue,ultra thick, opacity=0.4] coordinates {(1,0.63) (1,1.37)};
    \addplot +[mark=none,color=blue,ultra thick, opacity=0.4] coordinates {(1,-0.63) (1,-1.37)};
    \addplot[domain=0:1.5,color=blue,ultra thick,opacity=0.4]{1};
    \addplot[domain=0:1.5,color=blue,ultra thick,opacity=0.4]{-1};
    \addplot[domain=0:1,color=blue,thick,opacity=0.4]{0.37};
    \addplot[domain=0:1,color=blue,thick,opacity=0.4]{0.3};
    \addplot[domain=0:1,color=blue,thick,opacity=0.4]{0.25};
    \addplot[domain=0:1,color=blue,thick,opacity=0.4]{0.2};    \addplot[domain=0:1,color=blue,thick,opacity=0.4]{-0.37};
    \addplot[domain=0:1,color=blue,thick,opacity=0.4]{-0.3};
    \addplot[domain=0:1,color=blue,thick,opacity=0.4]{-0.25};
    \addplot[domain=0:1,color=blue,thick,opacity=0.4]{-0.2};
    \addplot[domain=0:1,color=blue,thick,opacity=0.4]{1.37};
    \addplot[domain=0:1,color=blue,thick,opacity=0.4]{1.3};
    \addplot[domain=0:1,color=blue,thick,opacity=0.4]{1.25};
    \addplot[domain=0:1,color=blue,thick,opacity=0.4]{1.2};    \addplot[domain=0:1,color=blue,thick,opacity=0.4]{0.63};
    \addplot[domain=0:1,color=blue,thick,opacity=0.4]{0.7};
    \addplot[domain=0:1,color=blue,thick,opacity=0.4]{0.75};
    \addplot[domain=0:1,color=blue,thick,opacity=0.4]{0.8};
    \addplot[domain=0:1,color=blue,thick,opacity=0.4]{-1.37};
    \addplot[domain=0:1,color=blue,thick,opacity=0.4]{-1.3};
    \addplot[domain=0:1,color=blue,thick,opacity=0.4]{-1.25};
    \addplot[domain=0:1,color=blue,thick,opacity=0.4]{-1.2};    \addplot[domain=0:1,color=blue,thick,opacity=0.4]{-0.63};
    \addplot[domain=0:1,color=blue,thick,opacity=0.4]{-0.7};
    \addplot[domain=0:1,color=blue,thick,opacity=0.4]{-0.75};
    \addplot[domain=0:1,color=blue,thick,opacity=0.4]{-0.8};
    \addplot[domain=0.45:0.55,color=blue,dotted,thick,opacity=0.4]{0.1};
    \addplot[domain=0.45:0.55,color=blue,dotted,thick,opacity=0.4]{-0.1};
    \addplot[domain=0.45:0.55,color=blue,dotted,thick,opacity=0.4]{0.9};
    \addplot[domain=0.45:0.55,color=blue,dotted,thick,opacity=0.4]{-0.9};
    \addplot[domain=0.45:0.55,color=blue,dotted,thick,opacity=0.4]{-1.1};
    \addplot[domain=0.45:0.55,color=blue,dotted,thick,opacity=0.4]{1.1};
    \addplot[domain=0.7:0.79,color=blue,dotted, ultra thick,opacity=0.4]{1.6};
    \addplot[domain=0.7:0.78,color=blue,dotted,ultra thick,opacity=0.4]{-1.6};
    \addplot[color=blue,ultra thick,opacity=0.4]{0};
   \node[below right=0.1pt of {(1,0)}, outer sep=1pt,fill=white] {$\tfrac{2\pi}{3}r_1$};
     \node[below right=0.1pt of {(1.5,0)}, outer sep=0.9pt,fill=white] {$\tfrac{2\pi}{3}r_2$};
    \end{axis} 
    \end{tikzpicture}
    \caption{The dendrogram associated to $\ppi_1(\mathbb{S}^1(r_1)\times \mathbb{S}^1(r_2))$. The $y$-axis represents elements of $\pi_1(\mathbb{S}^1(r_1)\times \mathbb{S}^1(r_2))\cong \Z\gamma_1\times\Z\gamma_2$, where $\gamma_1$ and $\gamma_2$ are generators of $\pi_1(\bbS^1(r_1))$ and $\pi_1(\bbS^1(r_2))$, respectively. More precisely, for $k\geq 1$ and $l\in \Z$, the element $(\gamma_1^{\pm k},\gamma_2^{l})$ is represented by the $y$-value $l\pm \tfrac{1}{k+2}$ and the element $(\gamma_1^{0},\gamma_2^{l})$ is represented by the $y$-value $l$. 
    } \label{fig:dendrogram-T2}
\end{figure}
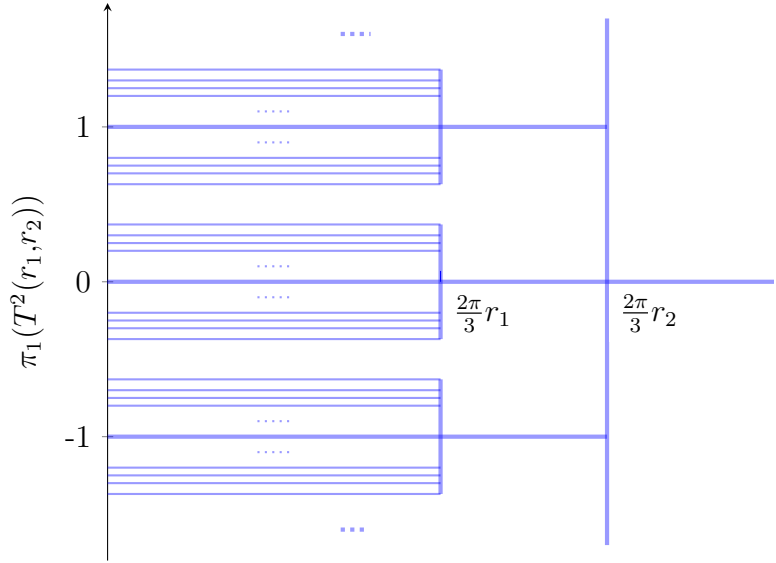  
\end{example}

\begin{proposition} \label{prop:cr=im}
For a compact geodesic s.l.s.c. space $X$, we have $$\mathrm{Cr}(X) = \mathrm{Im}\big(\mu_{\theta_{\pi_1(X)}}\big).$$ Thus, $\sup \mathrm{Cr}(X)=\max \mathrm{Cr}(X) = \mathrm{diam}\big(\pi_1(X)\big)$ and $\inf \mathrm{Cr}(X) = \mathrm{sep}\big(\pi_1(X)\big)$.
\end{proposition}

\begin{proof}
Indeed, $a\in \mathrm{Cr}(X)$ iff there is some $\epsilon>0$ such that $\pi_1^{t}(X)\cong \pi_1^a(X)$ for all $t\in (a-\epsilon,a]$ but $\pi_1^a(X)\not\cong \pi_1^{a+\delta}(X)$ for any $\delta>0$, iff $a$ is a `merging' point for branches of the dendrogram, iff $a\in \mathrm{Im}\big(\mu_{\theta_{\pi_1(X)}}\big)$. 
\end{proof}

We will see via Proposition \ref{prop:stab-cr} that the critical spectrum is stable due to the stability of persistent fundamental groups. As one would expect, the persistent fundamental group is a stronger invariant than the critical spectrum. For example, the critical spectra of $\bbS^1\vee\bbS^1$ and $\bbS^1\times\bbS^1$ coincide and are equal to $\{\frac{2\pi}{3}\}$ but the persistent fundamental groups of these two spaces are not only isomorphic but are at a positive interleaving distance, cf. Example \ref{ex:prod-wedge-S1-S1}.

\subsubsection{The diameter and separation of \texorpdfstring{$\pi_1(X)$}{pi1}} 
\label{sec:diam and sep}

In this section, we see some applications of the diameter and separation of the ultra-metric space $\left(\pi_1(X),\mu_{\theta_{\pi_1(\bbS^1)}}\right)$ (or simply $\pi_1(X)$).

First, it is clear that a simply-connected compact space $Y$ has a trivial associated dendrogram. Thus, the metric space $\big( \pi_1(Y),\mu_{\theta_{\pi_1(Y)}}\big)$ is isometric to the one-point metric space. By Proposition \ref{prop:property-GH},
the following corollary holds:
\begin{corollary} Let a compact geodesic space $X$ be semi-locally simply connected (s.l.s.c.) and let $Y$ be a simply connected compact space, then
$$\dgh\left( \left( \pi_1(X),\mu_{\theta_{\pi_1(X)}}\right),\left( \pi_1(Y),\mu_{\theta_{\pi_1(Y)}}\right)\right)=\tfrac{1}{2}\diam(\pi_1(X)).$$
\end{corollary}
For example, taking $X=\bbS^1$ and $Y=\mathbb{S}^2$, we obtain that 
$$\dgh\left( \left( \pi_1(\bbS^1),\mu_{\theta_{\pi_1(\bbS^1)}}\right),\left( \pi_1(\bbS^2),\mu_{\theta_{\pi_1(\bbS^2)}}\right)\right)=\tfrac{1}{2}\diam(\pi_1(\bbS^1))=\tfrac{\pi}{3}.$$

Another application of the diameter of the induced ultra-metric on $\pi_1(X)$ is that it provides a lower bound for the hyperbolicity of $X$ (Definition \ref{def:hyp}):
\begin{proposition}
For a space $X$ satisfying the conditions in Theorem \ref{thm:dendrogram}, we have
\[\diam(\pi_1(X))\leq \frac{1}{2} \hyp(X).\]
\end{proposition}
\begin{proof} If $\hyp(X)=\infty$, this is trivial. If $\hyp(X)<\infty$, the proposition follows directly from Corollary 8.1 of \cite{lim2020vietorisrips} which shows that for any $t>\hyp(X)$, the Vietoris-Rips complex $\VR_{2t}(X)$ is contractible.
\end{proof}

The next useful geometric quantity of the induced ultra-metric on $\pi_1(\cdot)$ is the separation (see page \pageref{para:sep}), which coincides
with the first critical value of the discrete fundamental groups. Using $\operatorname{sep}(\pi_1(\cdot))$, we obtain a lower bound for the interleaving distance between persistent fundamental groups, which will be applied in \textsection \ref{sec:finiteness} to prove finiteness theorems for the fundamental groups and persistent fundamental groups.
\begin{proposition}\label{prop:sep_of_dendrogram}
For two spaces $X$ and $Y$ satisfying the conditions in Theorem \ref{thm:dendrogram}, we have
\[\min\{\operatorname{sep}(\pi_1(X)),\operatorname{sep}(\pi_1(Y))\}\leq 2\cdot\di(\ppi_1(X),\ppi_1(Y)).\]
\end{proposition}

\begin{proof}
Suppose not. Then, there is a $\delta$-interleaving between $\ppi_1(X)$ and $\ppi_1(Y)$ for some $\delta<\frac{1}{2}\cdot \min\{\operatorname{sep}(\pi_1(X)),\operatorname{sep}(\pi_1(Y))\}.$ Then, for some $\epsilon>0$ small enough, we have the following commutative diagram:
\[
\begin{tikzcd}
    &
    \pi_1(X)\ar[dl] \ar[d] \ar[dr]
    &
    \\
    \pi_1^\epsilon(X)\ar[r]
    \ar[dr,color=red]
    &
    \pi_1^{\epsilon+\delta}(X)\ar[r]
    \ar[dr,color=red]
    &
    \pi_1^{\epsilon+2\delta}(X)
    \\
    \pi_1^\epsilon(Y)\ar[r]
    \ar[ur,color=blue]
    &
    \pi_1^{\epsilon+\delta}(Y)\ar[r]
    \ar[ur,color=blue]
    &
    \pi_1^{\epsilon+2\delta}(Y)
    \\
    &
    \pi_1(Y)\ar[ul] \ar[u] \ar[ur]
    &
\end{tikzcd}
\]
such that all black arrows are group isomorphism. Let $f$ be the composition of $\pi_1(X)\to  \pi_1^\epsilon(X)\to \pi_1^{\epsilon+\delta}(Y)\to \pi_1(Y)$, which is equal to the composition of $\pi_1(X)\to  \pi_1^{\epsilon+\delta}(X)\to \pi_1^{\epsilon+2\delta}(Y)\to \pi_1(Y)$ (because the diagram commutes). Similarly, let $g$ be the composition of $\pi_1(Y)\to  \pi_1^\epsilon(Y)\to \pi_1^{\epsilon+\delta}(X)\to \pi_1(X)$. 
It follows by diagram chasing that $g\circ f=\Id_{\pi_1(X)}$ and $f\circ g=\Id_{\pi_1(Y)}$. Thus, $\pi_1(X)\cong \pi_1(Y)$, which gives a contradiction.
\end{proof}

\subsubsection{Stability of the ultrametric \texorpdfstring{$\mu_{\theta_{\pi_1(\bullet)}}$}{pi1}}
\label{sec:stab of ultra metric}

It follows immediately from Theorem \ref{thm:dendrogram}, Proposition \ref{prop:stab-dendrogram} and Theorem \ref{thm:discretization} that the Gromov-Hausdorff distance between fundamental groups can be estimated by the interleaving distance between dendrograms.
\begin{theorem} \label{thm:pi_1-interleaving} If compact geodesic spaces $X$ and $Y$ are l.p.c. and u.s.l.s.c., then 
	$$\tfrac{1}{2}\cdot\di (\caV_{\F}\circ\theta_{\pi_1(X)},\caV_{\F}\circ\theta_{\pi_1(Y)})\leq \dgh\left(  \pi_1(X), \pi_1(Y)
	\right)\leq\di(\ppi_1(X),\ppi_1(Y)).$$
\end{theorem}

Next, we prove the following $\ell^{\infty}$-stability result for $\mu_{\theta_{\pi_1(\bullet)}}$ and discuss an application of it to Riemannian manifolds.

 \begin{restatable}[$\ell^{\infty}$-stability of $\mu_{\theta_{\pi_1(\bullet)}}$]{theorem}{infystabofpi}
 \label{thm:stab-l^infty-mu} Let compact geodesic spaces $(X,d_X)$ and $(Y,d_Y)$ be u.l.p.c. and u.s.l.s.c. Suppose $X$ and $Y$ are homotopy equivalent, i.e., there exists continuous maps $f:X\to Y$ and $g:Y\to X$ such that $f\circ g$ is homotopic to $ \Id_{Y}$ and $g\circ f$ is homotopic to $ \Id_{X}$. Then, 
 $$\left\|\mu_{\theta_{\pi_1(X)}}-\mu_{\theta_{\pi_1(Y)}}\circ (\pi_1f,\pi_1f)\right\|_{\infty}\leq \max\{\dis(f),\dis(g)\},$$
 where $\pi_1f:\pi_1(X)\to \pi_1(Y)$ is the induced homomorphism. In particular, when $X=Y$ and $d_X$ and $d_Y$ induce the same topology on $X$, we can take $f$ and $g$ to be the identity map. In this case, 
 \[\left\|\mu_{\theta_{\pi_1(X)}}-\mu_{\theta_{\pi_1(Y)}}\right\|_{\infty}\leq \|d_X-d_Y\|_{\infty}.\] 
 \end{restatable}

\begin{proof} Note that the induced homomorphisms $\pi_1f:\pi_1(X)\to \pi_1(Y)$ and $\pi_1g:\pi_1(Y)\to \pi_1(X)$ are isomorphisms, and these isomorphisms are inverses of each other. Since $\dis(\pi_1f)=\codis(\pi_1f,\pi_1g)=\dis(\pi_1g)$, we have $$\max\{\dis(\pi_1f),\codis(\pi_1f,\pi_1g),\dis(\pi_1g)\}=\dis(\pi_1f).$$ 
Let $\gamma$ and $\gamma'$ be two continuous loops in $X$ based at $x_0$. Let $\delta>\mu_{\theta_{\pi_1(X)}}(\gamma,\gamma')$. Then, there exists $\epsilon\geq \delta$ and $\epsilon$-loops $\alpha$ and $\alpha'$, which are strong $\epsilon$-loops along $\gamma$ and $\gamma'$ respectively, such that $\alpha\sim^{\epsilon}_1 \alpha'$ via an $\epsilon$-homotopy $H$. By an argument similar to that of Proposition \ref{prop:induced}, $f(\alpha)$ and $f(\alpha')$ are $(\epsilon+\dis(f))$-loops, and $f(H)$ is an $(\epsilon+\dis(f))$-homotopy. Directly checking from the definition, we see that $f(\alpha)$ and $f(\alpha')$ are strong $(\epsilon+\dis(f))$-loops along $\pi_1f(\gamma)$ and $\pi_1f(\gamma')$, respectively. Thus, $\mu_{\theta_{\pi_1(Y)}}(\pi_1f(\gamma),\pi_1f(\gamma'))\leq \epsilon+\dis(f)$. Letting $\delta \searrow \mu_{\theta_{\pi_1(X)}}(\gamma,\gamma')$, we obtain $\mu_{\theta_{\pi_1(Y)}}(\pi_1f(\gamma),\pi_1f(\gamma'))\leq \mu_{\theta_{\pi_1(X)}}(\gamma,\gamma')+\dis(f).$ Using the fact that $\pi_1f$ and $\pi_1g$ are inverses of each other, we can prove similarly that $\mu_{\theta_{\pi_1(X)}}(\gamma,\gamma')\leq \mu_{\theta_{\pi_1(Y)}}(\pi_1f(\gamma),\pi_1f(\gamma'))+\dis(g).$ It follows that 
\[\left\|\mu_{\theta_{\pi_1(X)}}-\mu_{\theta_{\pi_1(Y)}}\circ (\pi_1f,\pi_1f)\right\|_{\infty}=\dis(\pi_1f)\leq \max\{\dis(f),\dis(g)\}.\]
\end{proof} 

\begin{remark} Despite Theorem \ref{thm:stab-l^infty-mu}, it is not clear whether
\begin{equation}\label{eq:stab-pi_1}
    \dgh\left( \left( \pi_1(X),\mu_{\theta_{\pi_1(X)}}\right),\left( \pi_1(Y),\mu_{\theta_{\pi_1(Y)}}\right)\right)\leq \dgh(X,Y).
\end{equation} The difficulty in proving such a claim is that arbitrary set maps do not necessarily induce homomorphisms between fundamental groups. In \textsection \ref{sec:stability-ppi_n}, we will prove a weaker version of Eq. (\ref{eq:stab-pi_1}): with a factor $2$ on the right-hand side (cf. Theorem \ref{thm:stab-pi_1}).
\end{remark}

\paragraph*{An application to Riemannian manifolds.} 
As an application of Theorem \ref{thm:stab-l^infty-mu}, let us consider a connected Riemannian manifold $(M,g)$, where $M$ is a smooth manifold and $g$ smoothly assigns an inner product $g_p$ to each tangent space $T_p M$ for each $p\in M$. Let $\lambda:M\to \R_+$ be a smooth positive function. Then, $(M,\lambda\cdot g)$ is also a Riemannian manifold. A change of Riemannian metric $g\to \tilde{g}$ is called a \emph{conformal change}, if angles between two vectors with respect to $g$ and $\tilde{g}$ are the same at every point of the manifold. It is clear that $g\to \lambda\cdot g$ is a conformal change. Let $\lambda $ and $\tilde{\lambda}$ be two smooth positive functions from $M$ to $\R_+$. Given any two distinct points $x$ and $x'$ in $M$, if $\gamma$ is a smooth curve on $M$ from $x$ to $x'$, then 
\[\left|  L_{\lambda g}(\gamma)-L_{\tilde{\lambda}g}(\gamma)\right| =\left|  \int_{\gamma} \lambda g-\tilde{\lambda}g\right| \leq \|\lambda -\tilde{\lambda}\|_{\infty}\cdot \int_{\gamma}g,\]
where $L_{g}(\gamma)$ represents the length of $\gamma$ in $(M,g)$. It follows that
\begin{align*}
  d_{\lambda g}(x,x')&\leq \inf_{\gamma} L_{\lambda g}(\gamma)+\inf_{\gamma}\left|  L_{\lambda g}(\gamma)-L_{\tilde{\lambda}g}(\gamma)\right| \\
  &\leq \inf_{\gamma} L_{\tilde{\lambda}g}(\gamma)+ \|\lambda -\tilde{\lambda}\|_{\infty}\cdot \inf_{\gamma}\int_{\gamma}g\\ 
    &\leq d_{\tilde{\lambda}g}(x,x')+ \|\lambda -\tilde{\lambda}\|_{\infty}\cdot \diam(M,g).
\end{align*}
Thus, we have 
\begin{equation}\label{eq:conformal change}
  \| d_{\lambda g}-d_{\tilde{\lambda}g}\|_{\infty}\leq  \|\lambda -\tilde{\lambda}\|_{\infty}\cdot \diam(M,g).  
\end{equation}

Since $(M,\lambda\cdot g )$ and $(M,\tilde{\lambda}\cdot g)$ have the same homotopy type, we can apply Theorem \ref{thm:stab-l^infty-mu} and Eq. (\ref{eq:conformal change})
to obtain:
\begin{corollary}
$\left\|\mu_{\theta_{\pi_1(M,\lambda g)}}-\mu_{\theta_{\pi_1(M,\tilde{\lambda} g)}}\right\|_{\infty}\leq  \|\lambda -\tilde{\lambda}\|_{\infty}\cdot \diam(M,g).$
\end{corollary}
For example, consider the $2$-dimensional flat torus $(\bbT^2,g)$ defined as the quotient space of the Euclidean square $[0,2\pi]\times[0,2\pi]$ obtained by identifying the opposite sides (see Figure \ref{fig:torus}). Then, we have
\[\left\|\mu_{\theta_{\pi_1(\bbT^2,\lambda  g)}}-\mu_{\theta_{\pi_1(\bbT^2,\tilde{\lambda} g)}}\right\|_{\infty}\leq \|\lambda -\tilde{\lambda}\|_{\infty}\cdot \sqrt{5}\pi.\]

\begin{figure}[ht!]	  
\centering
\begin{tikzpicture}[scale=2]
  \fill[blue!10!white] (0,0) --(0,1)-- (1,1) --(1,0);
  \draw[
        decoration={markings, mark=at position 0.5 with {\arrow{>}}},
        postaction={decorate}
        ](0,0) --(1,0); 
  \draw[
        decoration={markings, mark=at position 0.5 with {\arrow{>}}},
        postaction={decorate}
        ] (0,1) -- (1,1);
  \draw[
        decoration={markings, mark=at position 0.55 with {\arrow{>}}},
        postaction={decorate}
        ] (0,0) -- (0,1);
  \draw[
        decoration={markings, mark=at position 0.55 with {\arrow{>}}},
        postaction={decorate}
        ] (1,0) -- (1,1);
  \draw[
        decoration={markings, mark=at position 0.45 with {\arrow{>}}},
        postaction={decorate}
        ] (0,0) -- (0,1);
  \draw[
        decoration={markings, mark=at position 0.45 with {\arrow{>}}},
        postaction={decorate}
        ] (1,0) -- (1,1);
\draw[dashed, red](0,0)--(0.5,1);
    \end{tikzpicture} 
\caption{The manifold $(\bbT^2,g)$, as the quotient space of the Euclidean square $[0,2\pi]\times[0,2\pi]$ by identifying the opposite sides. The length of the red line realizes the diameter of $\bbT^2$.}\label{fig:torus}
\end{figure}

\section{Stability of Persistent Homotopy Groups} \label{sec:stability-ppi_n} 

Recall that the interleaving distance between persistent groups $\bbG$ and $\bbH$ is given in Definition \ref{def:interleaving} by letting $\caC=\grp$:		
$$\di(\bbG,\bbH):=\inf\{\delta\geq0: \bbG\text{ and }\bbH\text{ are }\delta\text{-interleaved}\}.$$
In this section, we prove the following stability theorem for the interleaving distance of persistent $\rmK$-homotopy groups (equivalently, for persistent $\VR$-homotopy groups):
\stabofppiKn*

As a consequence of the above theorem, we can derive a stability theorem for fundamental groups:
\GHstabofpi*

Given the isomorphism between the persistent fundamental group and the persistent $\rmK$-fundamental group (cf. Theorem \ref{thm:iso-ppi_1's}), another immediate application of Theorem \ref{thm:stab-ppi_n} is the stability of persistent fundamental groups; see the rightmost inequality of the following theorem. 
The leftmost inequality below follows from the persistent Hurewicz theorem.

\begin{restatable} [$\di$-stability of $\ppi_1(\bullet)$]{theorem}{stabofppi} \label{thm:stab-di-ppi_1} 
Let $(X,x_0)$ and $(Y,y_0)$ be pointed compact metric spaces. Then
		$$\di(\pH_1(X),\pH_1(Y))\leq\di(\ppi_1(X,x_0),\ppi_1(Y,y_0))\leq 2\cdot \dgh^{\pt}((X,x_0),(Y,y_0)).$$
\end{restatable}

In the previous sections, we have studied the notion of persistent homotopy groups, its properties and some applications. In the next section, we show that persistent homotopy groups are stable under the Gromov-Hausdorff distance of metric spaces, and see that there are cases when persistent homotopy has stronger distinguishing power than persistent homology.


\subsection{Stability \& discriminating power of \texorpdfstring{$\ppi_n^{\rmK}(\bullet)$, $\ppi_n^{\VR}(\bullet)$}{K and VR}, \texorpdfstring{$\mu_{\theta_{\pi_1(\bullet)}}$}{pi1} \& \texorpdfstring{$\mathrm{Cr}(\bullet)$}{Cr}} \label{sec:stability-proof} 
To prove Theorem \ref{thm:stab-ppi_n}, we use the stability of homotopy interleaving distance under the Gromov-Hausdorff distance (see Theorem \ref{thm:stability-HI}). As homotopy groups depend on basepoints when spaces are not connected, in this section all maps between topological spaces are required to be basepoint preserving. 

\paragraph*{$\R$-spaces and homotopy interleavings.} Recall that $\topo$ is the category of compactly generated weakly Hausdorff topological spaces, and an object in $\Ptop^{(\R,\leq)}$ is also called an \emph{$\R$-space} (see \cite{blumberg2017universality}). Let $\bbX$ be an $\R$-space. If for every $t\leq s\in \R$, the map $X_t\to X_s$ is an inclusion, then we call $\bbX$ a \emph{filtration}. Blumberg and Lesnick showed that the interleaving distance between $\R$-spaces is not homotopy invariant (see Remark 3.3 of \cite{blumberg2017universality}), and then defined homotopy interleavings between $\R$-spaces as certain homotopy invariant analogues of interleavings. 

Given two $\R$-spaces $\bbX$ and $\bbY$, a natural transformation $f:\bbX\Rightarrow \bbY$ is an \emph{(objectwise) weak equivalence} if for each $t\in \R$, $f_t:X_t\to Y_t$ is a weak homotopy equivalence, i.e., it induces isomorphisms on all homotopy groups. \label{para:weak h.e.} The $\R$-spaces $\bbX$ and $\bbY$ are \emph{weakly equivalent}, denoted by $\bbX\simeq \bbY$, if there exists an $\R$-space $\bbW$ and natural transformations $f:\bbW\Rightarrow \bbX$ and $g:\bbW\Rightarrow \bbY$ that are (objectwise) weak equivalences:
\[\bbX\xLeftarrow{f} \bbW\xRightarrow{g} \bbY.\]
The relation $\bbX\simeq \bbY$ is an equivalence relation (see \cite{blumberg2017universality} for details). Given $T\subset \R$, we can consider the restriction of any given $\R$-space to $T$ and study $T$-spaces similarly. For $\delta\geq 0$, two $\R$-spaces $\bbX$ and $\bbY$ are \emph{$\delta$-homotopy-interleaved} if there exist $\R$-spaces $\bbX'\simeq \bbX$ and $\bbY'\simeq \bbY$ such that $\bbX'$ and $\bbY'$ are $\delta$-interleaved, as in Definition \ref{def:delta-interleaving} with $\caC=\topo$.
 
\begin{definition}[Homotopy interleaving distance] \label{def:HI} The \emph{homotopy interleaving distance} between two $\R$-spaces $\bbX$ and $\bbY$ is given by
	$$\dhi(\bbX,\bbY):=\inf \left\{\delta\geq 0:\bbX,\bbY \text{ are }\delta\text{-homotopy-interleaved}\right\}.$$ 
\end{definition}

\begin{theorem} [{Theorem 1.6, \cite{blumberg2017universality}}] \label{thm:stability-HI} Given two compact metric spaces $X$ and $Y$, we have
	$$\db(\dgm_k(X),\dgm_k(Y))\leq \dhi(\left| \VR_{\bullet}(X)\right| ,\left| \VR_{\bullet}(Y)\right| )\leq  2\cdot \dgh(X,Y).$$
Here $\left| \VR_{\bullet}(X)\right| $ is the $\R_{\geq 0}$-space given by the geometric realizations of Vietoris-Rips complexes of $X$.
\end{theorem}

Given an $\R_+$-space $\bbV:(\R_+,\leq)\to\topo$, we denote by $\lim \bbV$ the limit of $\bbV$ together with morphisms $v_t:\lim \bbV\to \bbV(t)$ for each $t\in \R_+$. For any $x\in \lim \bbV$, let $(\bbV,x)$ be the functor from $(\R_+,\leq)$ to $\topo^*$ such that $(\bbV,x)(t)=(V_t,v_t(x))$ and $(\bbV,x)(t\leq s)=v_t^s.$ Given two $\R_+$-spaces $\bbV$ and $\bbW$, as well as $x\in \lim \bbV$ and $y\in \lim \bbW$, we say that $(\bbV,x)$ and $(\bbW,y)$ are weakly equivalent if $\bbV\simeq \bbW$ via basepoint preserving maps. And $(\bbV,x)$ and $(\bbW,y)$ are said to be $\delta$-homotopy-interleaved if there exist $\bbV',\bbW':(\R_+,\leq)\to\topo^*$ such that $(\bbV',x')\simeq (\bbV,x)$, $(\bbW',y')\simeq (\bbW,y)$ for some $x'\in \lim \bbV'$ and $y'\in \lim \bbW'$, and $(\bbV',x')$ and $(\bbW',y')$ are $\delta$-interleaved via basepoint preserving maps. We define the homotopy interleaving distance $\dhi((\bbV,x),(\bbW,y))$ between $(\bbV,x)$ and $(\bbW,y)$ to be the infimum of $\delta$ such that $(\bbV,x)$ and $(\bbW,y)$ are $\delta$-homotopy-interleaved.

\begin{lemma}\label{lem:iso of ppi} Let $\bbV$ and $\bbW$ be two $\R_+$-spaces. If $\bbV\simeq \bbW$, then there exists $x\in \lim\bbV$ and $y\in \lim \bbW$ such that $(\bbV,x)\simeq (\bbW,y)$. Moreover, if $(\bbV,x)\simeq (\bbW,y)$, then for all $n\geq1,$
\[\pi_n\circ(\bbV,x)\cong \pi_n\circ(\bbW,y).\]
\end{lemma}
\begin{proof}Since $\bbV\simeq \bbW$, there exist an $\R_+$-space $\bbU$ and natural transformations $f:\bbU\Rightarrow \bbV$ and $g:\bbU\Rightarrow \bbW$ that are (objectwise) weak equivalences. Because of the universal property of $\lim \bbV$, there exists a morphism $\lim f$ such that the following diagram commutes for all $0<t\leq s$:
	\begin{center}
		\begin{tikzcd}
		V_t \ar[dd] &&& 
		U_t \ar[lll,"\simeq" below, "f_t" above]
		\ar[dd]
			\\
		& \lim \bbV \ar[ul]
		    \ar[dl]
		 &\lim \bbU \ar[l, "\lim f" above]
		 \ar[ur]
		 \ar[dr]\\
		 V_s
		&&	& U_s. \ar[lll,"\simeq" below, "f_s" above]
		\end{tikzcd}
	\end{center} 
Similarly, there exists a morphism $\lim g:\lim \bbU\to \lim \bbW$ satisfying a similar commutative diagram. Take any $z\in \lim \bbU$ and let $x:=(\lim f)(z)$ and $y:=(\lim g)(z)$. It is clear that $(\bbV,x)\simeq (\bbW,y)$.

In addition, we have the following commuting diagrams (for each $0<t\leq s$):
\begin{center}
	\begin{tikzcd}[column sep=2.5cm ]
		\pi_n(V_t,v_t(x)) \ar[d] & 
	    \pi_n(W_t,w_t(y)) \ar[l,"\cong" below, "\pi_n(f_t)\circ (\pi_n(g_t))^{-1}" above] \ar[r,"\cong" below, "\pi_n(g_t)\circ (\pi_n(f_t))^{-1}" above]
		\ar[d] & \pi_n(V_t,v_t(x)) \ar[d]
			\\
		\pi_n(V_s,v_s(x))
		& \pi_n(W_s,w_s(y)) \ar[l,"\cong" above, "\pi_n(f_s)\circ (\pi_n(g_s))^{-1}" below] \ar[r,"\cong" above, "\pi_n(g_s)\circ (\pi_n(f_s))^{-1}" below] &	\pi_n(V_s,v_s(x)).
	\end{tikzcd}
\end{center} 
Thus, $\pi_n\circ(\bbV,x)$ and $\pi_n\circ(\bbW,y)$ are naturally isomorphic.
\end{proof}	

\begin{lemma}\label{lem:di-dhi} Let $\bbV$ and $\bbW$ be two $\R_+$-spaces. Then, for any $x\in \lim\bbV$ and $y\in \lim \bbW$, 
\[\di\left(\pi_n\circ(\bbV,x), \pi_n\circ(\bbW,y)\right)\leq \dhi\left((\bbV,x),(\bbW,y)\right).\]
\end{lemma}
\begin{proof}Suppose $(\bbV,x)$ and $(\bbW,y)$ are $\delta$-homotopy-interleaved for some $\delta\geq0$. Then, there exist $\bbV',\bbW':(\R_+,\leq)\to\topo^*$ such that $(\bbV',x')\simeq (\bbV,x)$, $(\bbW',y')\simeq (\bbW,y)$ for some $x'\in \lim \bbV'$ and $y'\in \lim \bbW'$, and $(\bbV',x')$ and $(\bbW',y')$ are $\delta$-interleaved. By Lemma \ref{lem:iso of ppi}, we know that $\pi_n\circ(\bbV,x)$ and $\pi_n\circ(\bbV',x')$ are $0$-interleaved, and $\pi_n\circ(\bbW,y)$ and $\pi_n\circ(\bbW',y')$ are $0$-interleaved. Thus,
\begin{align*}
    &\di\left(\pi_n\circ(\bbV,x),\, \pi_n\circ(\bbW,y)\right)\\
    \leq& \di\left(\pi_n\circ(\bbV,x),\, \pi_n\circ(\bbV',x')\right)+
    \di\left(\pi_n\circ(\bbV',x'),\, \pi_n\circ(\bbW',y')\right)\\
    &+
    \di\left(\pi_n\circ(\bbW',y'),\, \pi_n\circ(\bbW,y)\right)\\
    \leq& 0+\di\left((\bbV',x'),\, (\bbW',y')\right)+0\\
    \leq& \delta.
\end{align*}
\end{proof}
	
\begin{myproof}{Theorem \ref{thm:stab-ppi_n}} We first notice that the proof of Theorem \ref{thm:stability-HI} (see page 20 of \cite{blumberg2017universality}) can be modified to show that 
$$\dhi\left(\left(\left| \VR(X)\right| ,x_0\right),\,\left(\left| \VR(Y)\right| ,y_0\right)\right)\leq 2\cdot \dgh^{\pt}((X,x_0),(Y,y_0)).$$
Applying Lemma \ref{lem:di-dhi}, for each $n\in\Z_{\geq 1}$ we obtain
\begin{align*}
    \di\left(\ppi_n^{\VR}(X,x_0),\ppi_n^{\VR}(Y,y_0)\right)&= \di\left(\pi_n\circ\left(\left| \VR(X)\right| ,x_0\right),\,\pi_n\circ\left(\left| \VR(Y)\right| ,y_0\right)\right)\\
    &\leq \dhi\left(\left(\left| \VR(X)\right| ,x_0\right),\,\left(\left| \VR(Y)\right| ,y_0\right)\right)
    \\
    &\leq 2\cdot \dgh^{\pt}((X,x_0),(Y,y_0)).
\end{align*}
It follows from Corollary \ref{cor:iso-ppi_n} that
\begin{align*}
    &\di\left(\ppi_n^{\rmK}(X,x_0),\ppi_n^{\rmK}(Y,y_0)\right)\\
    =&\tfrac{1}{2}\cdot\di\left(\ppi_n^{\VR}(X,x_0),\ppi_n^{\VR}(Y,y_0)\right)\leq  \dgh^{\pt}((X,x_0),(Y,y_0)).
\end{align*}
When $X$ and $Y$ are chain-connected, we can take the infimum over all pairs $(x_0,y_0)\in X\times Y$ of both sides of the above equation to get
$$\di\left(\ppi_n^{\rmK}(X),\ppi_n^{\rmK}(Y)\right)=\tfrac{1}{2}\cdot\di\left(\ppi_n^{\VR}(X),\ppi_n^{\VR}(Y)\right)\leq  \dgh(X,Y).$$
\end{myproof}

\begin{myproof}{Theorem \ref{thm:stab-pi_1}} By Theorem \ref{thm:pi_1-interleaving}, we have the first inequality below:
\begin{align}\label{eq:dgh of pi-di of ppi_1}
    &\dgh\left( \left( \pi_1(X),\mu_{\theta_{\pi_1(X)}}\right),\left( \pi_1(Y),\mu_{\theta_{\pi_1(Y)}}\right)\right)\\
    \leq& \di \left(\theta_{\pi_1(X)},\theta_{\pi_1(Y)}\right)\leq \di(\ppi_1(X),\ppi_1(Y)).
\end{align}
The second inequality is true because any $\delta$-interleaving between $\ppi_1(X)$ and $\ppi_1(Y)$ clearly induces a $\delta$-interleaving between $\theta_{\pi_1(X)}$ and $\theta_{\pi_1(Y)}$. It follows from Theorem \ref{thm:iso-ppi_1's} that
\begin{equation}\label{eq:cor-ppi-ppi*}
\di(\ppi_1(X),\ppi_1(Y))=2\cdot\di\left(\ppi_1^{\rmK}(X),\ppi_1^{\rmK}(Y)\right).
\end{equation}
Finally we combine the above equations and 
apply Lemma \ref{lem:di-dhi} for the case $n=1$, to obtain
$$\dgh\left( \left( \pi_1(X),\mu_{\theta_{\pi_1(X)}}\right),\left( \pi_1(Y),\mu_{\theta_{\pi_1(Y)}}\right)\right)\leq 2\cdot \dgh(X,Y).$$
\end{myproof}

\paragraph*{Stability of critical spectrum.} Recall from page \pageref{para:critical value} the definition of critical values and critical spectra of metric spaces. As an application of Theorem \ref{thm:stab-pi_1}, we prove that the Hausdorff distance between critical spectra of metric spaces is stable under the Gromov-Hausdorff distance between the spaces:
\begin{proposition}
\label{prop:stab-cr}
Given compact geodesic s.l.s.c. spaces $X$ and $Y$, we have
\[\dhaus^{\R}(\mathrm{Cr}(X),\mathrm{Cr}(Y))\leq  2\cdot \dgh\left( \left( \pi_1(X),\mu_{\theta_{\pi_1(X)}}\right),\left( \pi_1(Y),\mu_{\theta_{\pi_1(Y)}}\right)\right) \leq 4\cdot\dgh(X,Y).\]
\end{proposition}

It seems interesting to try to define the critical spectrum as a \emph{multiset} by considering a suitable notion of multiplicity for critical points. By doing so, one could then study a notion of stability result of such critical spectra under a stronger distance: the bottleneck distance, in the same spirit as how \cite{cohen2007stability} studied the stability of persistence diagrams.
We leave this as future work.

To prove the Proposition \ref{prop:stab-cr}, first recall the following lemma from Theorem 3.4 of \cite{memoli2012some}:
\begin{lemma} \label{lem:stab-dh-cr}
For any two metric spaces $(Z_1,d_1)$ and $(Z_2,d_2)$, \[\dhaus^{\R}(\mathrm{Im}(d_1),\mathrm{Im}(d_2))\leq 2\cdot \dgh(Z_1,Z_2).\] 
\end{lemma}

\begin{myproof}{Proposition \ref{prop:stab-cr}} 
Recall from Proposition \ref{prop:cr=im} that the critical spectrum $\mathrm{Cr}(X)$ agrees with the image of the ultra-metric $\mu_{\theta_{\pi_1(X)}}$ on $\pi_1(X)$. Combining with Lemma \ref{lem:stab-dh-cr} and Theorem \ref{thm:stab-pi_1}, we obtain
\begin{align*}
\dhaus^{\R}(\mathrm{Cr}(X),\mathrm{Cr}(Y)) =& \dhaus^{\R}\left(\mathrm{Im}\big(\mu_{\theta_{\pi_1(X)}}\big),\mathrm{Im}\big(\mu_{\theta_{\pi_1(Y)}}\big)\right)\\
\leq & 2\cdot \dgh\left( \left( \pi_1(X),\mu_{\theta_{\pi_1(X)}}\right),\left( \pi_1(Y),\mu_{\theta_{\pi_1(Y)}}\right)\right) \\
\leq & 4\cdot \dgh(X,Y).    
\end{align*} 
\end{myproof}

\subsection{Examples of stability of persistent homotopy groups}

Stability results such as Theorem \ref{thm:stab-di-ppi_1} can be used to estimate the Gromov-Hausdorff distance between compact metric spaces. See \cite{lim2021gromov} for the precise calculation of the Gromov-Hausdorff distance between spheres.

To see that persistent homotopy sometimes provides a better approximation to the Gromov-Hausdorff distance than persistent homology, we examine several examples in this section.

\begin{example}[$\bbS^1\times \bbS^1$ vs. $\bbS^1\vee \bbS^1$]\label{ex:prod-wedge-S1-S1} 
Let $\bbT^2=\bbS^1\times \bbS^1$ denote the torus constructed as the $\ell^{\infty}$ product of two unit circles, and let $C=\bbS^1\vee \bbS^1$ be the wedge sum of two unit circles. By Example \ref{ex:circle-VR-ppi_n} and Example \ref{ex:bouquet of circles}, the persistent homology groups of $C$ are: for all $k\geq 1$,
\[\pH_{n}^{\VR}(C)\cong\begin{cases} 
 (\Z\times\Z)\left(\tfrac{k-1}{2k-1}\, 2\pi,\,\tfrac{k}{2k+1}\, 2\pi\right),&\mbox{if $n=2k-1,$}\\[6pt] 
 (\Z^{\times \mathfrak{c}}\times\Z^{\times \mathfrak{c}})\left[\tfrac{k}{2k+1}\, 2\pi,\,\tfrac{k}{2k+1}\, 2\pi\right],&\mbox{if $n=2k$.}
 \end{cases}\]
Following from the above, we can compute the persistent homology with coefficients  in $\R$ and obtain the undecorated persistence diagrams of $C$ as follows: for all $k\geq 0$,
\begin{align*}
    \dgm_{2k+1}(C)&=\left\{\left(\tfrac{k}{2k+1}\,2\pi,\,\tfrac{k+1}{2k+3}\,2\pi\right),\left(\tfrac{k}{2k+1}\,2\pi,\,\tfrac{k+1}{2k+3}\,2\pi\right)\right\},\\   
    \dgm_{2k+2}(C)&=\emptyset.
\end{align*}
It is clear that $\dgm_0(C)=\dgm_0(\bbT^2)=\{(0,\infty)\}$. Together with Example \ref{ex:product of S1}, for all $k\geq 0$, we have $\dgm_{2k+1}(\bbT^2)=\dgm_{2k+1}(C)$, $\dgm_{4k+4}(\bbT^2)=\dgm_{4k+4}(C)$,
\[ \dgm_{4k+2}(\bbT^2)=\left\{\left(\tfrac{k}{2k+1}\,2\pi,\,\tfrac{k+1}{2k+3}\,2\pi\right)\right\} \text{ and }\dgm_{4k+2}(C)=\emptyset.\] 
Thus,
\begin{align*}
    \sup_{n\geq 0}\db\left(\dgm_{n}(\bbT^2),\dgm_{n}(C)\right)
    &=\sup_{n=4k+2,k\geq 0}\db\left(\dgm_{n}(\bbT^2),\dgm_{n}(C)\right)\\
    &=\sup_{k\geq 0}\db\left(\left\{\left(\tfrac{k}{2k+1}\,2\pi,\,\tfrac{k+1}{2k+3}\,2\pi\right)\right\},\emptyset\right)\\
    &=\sup_{k\geq 0}\tfrac{\pi}{(2k+3)(2k+1)}=\tfrac{\pi}{3}.
\end{align*}
Thus, in this case the best approximation to $\dgh(\bbT^2,C)$ given by persistence diagrams is obtained at homological dimension $2$. Now let us see that the same lower bound for $\dgh(\bbT^2,C)$ can already be realized by the interleaving distance between persistent fundamental groups. Recall that
	\[\ppi_1(\bbT^2)=(\Z\times \Z)\left(0,\tfrac{2\pi}{3}\right)\text{ and }
	\ppi_1(C)=(\Z* \Z)\left(0,\tfrac{2\pi}{3}\right).\] 
We claim that $$\di(\ppi_1(\bbT^2),\ppi_1(C))= \tfrac{\pi}{3}.$$
Indeed, because the composition of group homomorphisms $\Z\ast \Z\to \Z\times\Z\to \Z\ast \Z$ can never be $\Id_{\Z\ast \Z}$, it must be true that $\di(\ppi_1(\bbT^2),\ppi_1(C))\geq \tfrac{\pi}{3}.$ By directly checking that $\ppi_1(\bbT^2)$ and $\ppi_1(C)$ are $\tfrac{\pi}{3}$-interleaved, we can conclude that $\di(\ppi_1(\bbT^2),\ppi_1(C))= \tfrac{\pi}{3}.$
\end{example}

To compute further examples, we recall from \cite{Hatcher01(AT)} and \cite{whitehead1950generalization} the following facts about homotopy groups under wedge sum. 
\begin{proposition}
 \label{prop:homotopy_wedge}
For $n\geq 2$, $\pi_n(X\vee Y)\cong \pi_n(X)\oplus \pi_n(Y)\oplus \pi_n(X\times Y,X\vee Y).$
    \begin{romanlist}
        \item $\pi_n(\bbS^1\vee \bbS^n)\cong\Z[t,t^{-1}]$ the Laurent polynomials in $t$ and $t^{-1}$. 
        \item 
        $\pi_n\left(\bbS^p\vee \bbS^q\right)\cong\pi_n(\bbS^p)\oplus\pi_n(\bbS^q)\oplus\pi_n(\bbS^{p+q-1})$ for $n<p+q+\min\{p,q\}-3$.
    \end{romanlist}
\end{proposition}

\begin{example}[$\bbS^1\times \bbS^m$ vs. $\bbS^1\vee \bbS^m$] \label{ex:prod-wedge-S1-Sn}
For $m\geq 1$, let $\bbS^m$ denote the $m$-sphere of radius $1$, equipped with the geodesic distance. We let $a_m:=\tfrac{1}{2}\arccos(-\tfrac{1}{m+1})> \tfrac{\pi}{4}$. It follows from Theorem 10 of \cite{lim2020vietorisrips} that $\left(\bbS^m\right)^{\epsilon}$ is homotopy equivalent to $\bbS^m$ for any $\epsilon\in \left(0,a_m\right)$ and the inclusion $\left(\bbS^m\right)^{\epsilon}\hookrightarrow \left(\bbS^m\right)^{\epsilon'}$ is a homotopy equivalence for $0<\epsilon\leq\epsilon' \leq a_m$. Thus, 
$$\left(\left|  \VR_{2\bullet}(\bbS^m)\right| \right)|_{(0,a_m)} \cong (\bbS^m)^{\bullet}|_{(0,a_m)}\cong \bbS^m(0,a_m),$$
where $|_{(0,a_m)}$ means restriction of functors to the interval $(0,a_m)$ and $\bbS^m(0,a_m)$ is an interval persistence module (see page \pageref{para:interval g.p.m.}). Since $a_m\leq a_1$, we have $\left(\left|  \VR_{2\bullet}(\bbS^1)\right| \right)|_{(0,a_m)} \cong \bbS^1(0,a_m)$, and thus, for each $n\geq 1$,
\begin{center}
\begin{tabular}{ c|c|c } 
 & $\ppi_n^{\VR}|_{(0,2a_m)}$ & $\pH_n^{\VR}|_{(0,2a_m)}$ \\[1ex]
\hline
\rule{0pt}{4ex} 
$\bbS^1\times \bbS^m$ & $\left(\pi_n(\bbS^1\times \bbS^m)\right) (0,2a_m)$ & $\left( \rmH_n(\bbS^1\times \bbS^m)\right) (0,2a_m)$\\ [1ex]
$\bbS^1\vee \bbS^m$ & $\left(\pi_n(\bbS^1\vee \bbS^m)\right) (0,2a_m)$ & $\left( \rmH_n(\bbS^1\vee \bbS^m)\right) (0,2a_m)$ \\ [1ex]
\end{tabular}.
\end{center}

When $n=m\geq 2$, we apply Proposition \ref{prop:homotopy_wedge} (i) to the above table to obtain
\begin{center}
\begin{tabular}{ c|c|c } 
 & $\ppi_m^{\VR}|_{(0,2a_m)}$ & $\pH_m^{\VR}|_{(0,2a_m)}$ \\[1ex]
\hline
\rule{0pt}{4ex} 
$\bbS^1\times \bbS^m$ & $\Z(0,2a_m)$ & $\Z(0,2a_m)$\\ [1ex]
$\bbS^1\vee \bbS^m$ & $(\Z[t,t^{-1}] )(0,2a_m)$ & $\Z(0,2a_m)$ \\ [1ex]
\end{tabular}.
\end{center}
Because the composition $\Z[t,t^{-1}]\to \Z\to \Z[t,t^{-1}]$ can never be $\Id_{\Z[t,t^{-1}]}$, the leftmost inequality below is true:  
$$\tfrac{1}{2}\cdot a_m\leq
\tfrac{1}{2}\cdot \di\left(\ppi_m^{\VR}\left(\bbS^1\vee\bbS^m\right),\ppi_m^{\VR}\left(\bbS^1\times\bbS^m\right)\right)
\leq \dgh\left(\bbS^1\vee\bbS^m,\bbS^1\times\bbS^m\right).$$
The rightmost inequality follows from Theorem \ref{thm:stab-ppi_n}. On the contrary, 
$$\tfrac{1}{2}\cdot\di\left(\pH_m^{\VR}\left(\bbS^1\vee\bbS^m\right),\pH_m^{\VR}\left(\bbS^1\times\bbS^m\right)\right)\leq 
\tfrac{1}{2}\cdot \tfrac{\pi-2a_m}{2}
<
\tfrac{1}{2}\cdot a_m.$$
Thus, in this example, the persistent homotopy contains richer information of spaces, compared to persistent homology.

The case of $n=1$ was covered in Example \ref{ex:prod-wedge-S1-S1}, where we saw
\begin{center}
\begin{tabular}{ c|c|c } 
 & $\ppi_1^{\VR}$ & $\pH_1^{\VR}$ \\[1ex]
\hline
\rule{0pt}{4ex} 
$\bbS^1\times \bbS^1$ & $(\Z\times \Z )(0,\tfrac{2\pi}{3})$ & $(\Z\times \Z )(0,\tfrac{2\pi}{3})$\\ [1ex]
$\bbS^1\vee \bbS^1$ & $(\Z\ast \Z )(0,\tfrac{2\pi}{3})$ & $(\Z\times \Z )(0,\tfrac{2\pi}{3})$ \\ [1ex]
\end{tabular}.
\end{center}
\end{example}

\begin{example}[$\bbS^1\times \bbS^1$ vs. $\bbS^1\vee \bbS^1\vee \bbS^2$] \label{ex:torus-S1S1S2} By an argument similar to that of Example \ref{ex:prod-wedge-S1-Sn}, we obtain the following table where $a_2=\tfrac{1}{2}\arccos(-\tfrac{1}{3})$ and each $n\geq 1$:
\begin{center}
\begin{tabular}{ c|c|c } 
 & $\ppi_n^{\VR}|_{(0,2a_2)}$ & $\pH_n^{\VR}|_{(0,2a_2)}$ \\[1ex]
\hline
\rule{0pt}{4ex} 
$\bbS^1\times \bbS^1$ & $\left(\pi_n(\bbS^1\times \bbS^1)\right) (0,2a_2)$ & $\left( \rmH_n(\bbS^1\times \bbS^1)\right) (0,2a_2)$\\ [1ex]
$\bbS^1\vee \bbS^1\vee\bbS^2$ & $\left(\pi_n(\bbS^1\vee \bbS^1\vee\bbS^2)\right) (0,2a_2)$ & $\left( \rmH_n(\bbS^1\vee \bbS^1\vee\bbS^2)\right) (0,2a_2)$ \\ [1ex]
\end{tabular}.
\end{center}
Because $\bbS^1\times \bbS^1$ and $\bbS^1\vee \bbS^1\vee \bbS^2$ have isomorphic homology groups in all dimensions, their persistent homology restricted on the interval $(0,2a_2)$ are also isomorphic. However,
\[\ppi_1^{\VR}|_{(0,2a_2)}(\bbS^1\times \bbS^1)=(\Z\times\Z)(0,2a_2)  \text{ and } \ppi_1^{\VR}|_{(0,2a_2)}(\bbS^1\vee \bbS^1\vee \bbS^2)=(\Z\ast\Z)(0,2a_2) \]
are not isomorphic. Moreover, because the composition of group homomorphisms $\Z\ast \Z\to \Z\times\Z\to \Z\ast \Z$ can never be $\Id_{\Z\ast \Z}$, the leftmost inequality below is true:
$$\tfrac{1}{2}\cdot a_2\leq 
\tfrac{1}{2}\cdot \di\left(\ppi_1^{\VR}\left(\bbS^1\vee\bbS^1\vee\bbS^2 \right),\ppi_1^{\VR}\left(\bbS^1\times\bbS^1\right)\right)
\leq \dgh\left(\bbS^1\vee\bbS^1\vee\bbS^2,\bbS^1\times\bbS^1\right).$$
The rightmost inequality follows from Theorem \ref{thm:stab-ppi_n} with $n=1$.
\end{example}

In Example \ref{ex:prod-wedge-S1-S1}, Example \ref{ex:prod-wedge-S1-Sn} and Example \ref{ex:torus-S1S1S2}, we have seen how persistent fundamental groups help in distinguishing spaces. In the next example, we consider a pair of simply connected spaces in which case the persistent rational homotopy groups, a weaker invariant than persistent homotopy groups, still capture more information than the persistent homology groups with rational coefficients.

\begin{example}[$\bbS^m\times\bbS^m$ vs. $\bbS^m\vee\bbS^m\vee\bbS^{2m}$] 
\label{ex:rational_homotopy} Let $m\geq 2$.
Recall from Example \ref{ex:prod-wedge-S1-Sn} that for any $m$, $a_m=\tfrac{1}{2}\arccos(-\tfrac{1}{m+1})$ and $\left(\left|  \VR_{2\bullet}(\bbS^m)\right| \right)|_{(0,a_m)} \cong \bbS^m(0,a_m)$.
Since $a_{2m}\leq a_m$, we have $\left(\left|  \VR_{2\bullet}(\bbS^m)\right| \right)|_{(0,a_{2m})} \cong \bbS^m(0,a_{2m})$. Thus, for each $n\geq 1$, the persistent homology of the two spaces restricted on the interval $(0,2a_{2m})$ are isomorphic:
\[\pH_n^{\VR}(\bbS^m\times \bbS^m;\Q)|_{(0,2a_{2m})}\cong\pH_n^{\VR}(\bbS^m\vee\bbS^m\vee \bbS^{2m};\Q)|_{(0,2a_{2m})}.\]
But the persistent rational homotopy groups can tell the two spaces apart. 

Let $m\geq 3$ and $n=3m-1$. Then, $3m-1<m+2m+\min\{m,2m\}-3$, so it follows from Proposition \ref{prop:homotopy_wedge} (ii) that 
\[\pi_{3m-1}(\bbS^m\vee \bbS^{2m})=\pi_{3m-1}(\bbS^{m})\oplus\pi_{3m-1}(\bbS^{2m})\oplus\pi_{3m-1}(\bbS^{3m-1})=\Z.\]
Applying Proposition \ref{prop:homotopy_wedge} again, we see that $$\pi_{3m-1}(\bbS^m\vee\bbS^m\vee \bbS^{2m})\otimes \Q=\Q\oplus \pi_{3m-1}\left(\bbS^m\times(\bbS^m\vee \bbS^{2m}),\bbS^m\vee(\bbS^m\vee \bbS^{2m})\right)\otimes\Q$$ has rank at least $1$. Therefore, the restriction of persistent rational homotopy $\ppi_{3m-1}^{\VR}(\bbS^m\vee\bbS^m\vee \bbS^{2m})\otimes\Q|_{(0,2a_{2m})}$ is non-zero everywhere, whereas $\ppi_{3m-1}^{\VR}(\bbS^m\times \bbS^m)\otimes\Q|_{(0,2a_{2m})}=\bbmzero$. Combined with the stability of persistent rational homotopy groups, we obtain
$$\tfrac{1}{2}\cdot a_{2m}
\leq \dgh\left(\bbS^m\vee\bbS^m\vee \bbS^{2m},\bbS^m\times \bbS^m\right).$$

In the case of $m=2$, take $n=3$. It follows from $\pi_{3}(\bbS^2)=\Z$ and Proposition \ref{prop:homotopy_wedge} that $\pi_{3}(\bbS^2\vee\bbS^2\vee \bbS^{4})\otimes \Q$ has rank at least $1$, while $\pi_3(\bbS^2\times\bbS^2)=0$. Applying a similar argument as before, we conclude that $\tfrac{1}{2}\cdot a_{4}
\leq \dgh\left(\bbS^2\vee\bbS^2\vee \bbS^{4},\bbS^2\times \bbS^2\right).$
\end{example}

\subsection{A finiteness theorem for \texorpdfstring{$\pi_1(\bullet)$}{pi1} 
} 
\label{sec:finiteness}

In this section, we establish a finiteness theorem for $\pi_1(\bullet)$ as an application of our stability result for $\ppi_1(\bullet)$ (cf. Theorem \ref{thm:stab-di-ppi_1}).

Finiteness theorems are rooted in the classification problem for manifolds. By considering various bounds on Riemannian invariants, such as curvature, diameter, volume, etc., one tries to prove that there are only finitely many diffeomorphism/homeomorphism/homotopy types of Riemannian manifolds satisfying such restrictions. We refer the readers to Berger's exposition in \textsection 12.4 of \cite{berger2003panoramic} for a detailed overview.

For example, in 1967, Weinstein showed that the class of even-dimensional Riemannian manifolds with both positive lower and upper bounds on their sectional curvature has finitely many homotopy types \cite{weinstein1967homotopy}. Around the same time, Cheeger established a finiteness theorem for the diffeomorphism types of manifolds satisfying uniform lower and upper bounds on sectional curvature, a lower bound on volume and an upper bound on their diameter \cite{cheeger1970finiteness}. In 1984, Peters reproved Cheeger's finiteness theorem via Gromov's compactness theorem, see \cite{Peters+1984+77+82}. In \cite{grove1988bounding} Grove and Petersen removed the upper bound for sectional curvature in Cheeger's theorem for the case of homotopy types, and in \cite{grove1990geometric} together with Wu improved their results to homomorphism types and diffeomorphism types with minor restrictions on the dimension of the manifolds.

In \cite{yamaguchi1988homotopy}, Yamaguchi proved that any precompact family of Riemannian manifolds with a uniform lower bound on their contractibility radii has finitely many homotopy types, where the contractibility radius of a manifold is the supremum of $r$ such that every metric ball of radius $r$ contains no critical points of the distance function from the center of the ball. Petersen generalized this result in \cite{peter1990finiteness} to the setting of compact metric spaces using a finer notion of local contractibility, called local geometric $n$-connectedness (see Definition \ref{def:lgc} for the case $n=1$).

In \cite{gromov1999metric}, Gromov proved that the fundamental group of a compact Riemannian manifold $M$ has a set $\{g_i\}_i$ of generators each with length at most twice of $\diam(M)$ and the relations are of the form $g_ig_j=g_k$. 
A fundamental group finiteness theorem follows immediately when considering any family of spaces with a global bound $N$ such that for any space $M$ in the family $\pi_1(M)$ can be generated by at most $N$ many loops each with length at most twice of $\diam(M).$ In this case, $\pi_1(M)$ has at most $N$ generators and $N^3$ relations, so there are no more than $2^{N^3}$ many choices for $\pi_1(M)$.

In \cite{plaut2013discrete}, via discrete homotopy methods, Plaut and Wilkins generalized Gromov's theorem by (1) proving a result which is applicable to semi-locally simply connected (s.l.s.c.) geodesic spaces (a far more general setting than Riemannian manifolds); (2) providing an explicit upper bound on the number of generators for the fundamental groups (see Eq. (\ref{eq:num of g_i})). 

In this section, we continue this line of work and establish (see Proposition \ref{prop:finiteness_pi1})  a finiteness result for the fundamental groups of spaces satisfying an $\operatorname{LGC_1}(\rho,R)$ condition for some geometric contractibility function $\rho$. This condition is slightly stronger than s.l.s.c. and permits proving finer estimates on the number of isomorphism classes of fundamental groups, cf. Example \ref{ex:finiteness}.
This finiteness result follows  by combining the stability property of $\ppi_1(\bullet)$ (see Theorem \ref{thm:stab-di-ppi_1} and Lemma \ref{lem:same_pi_1}) with a  result on the covering number of Gromov-Hausdorff precompact classes (cf. Lemma \ref{lem:finiteness}).

\paragraph*{Finiteness result for $\pi_1(\bullet)$
.} 
Recall from Definition \ref{def:lgc} that a metric space $X$ is said to be $\operatorname{LGC_1}(\rho,R)$ for some $R\geq 0$ and some non-decreasing function $\rho:[0,R]\to [0,\infty)$, if for all $x\in X$ and $r\in [0,R]$ the ball $B(x,r)$ is $0$-connected and $1$-connected inside $B(x,\rho(r))$. 
Here $R\geq 0$ and $\rho:[0,R]\to [0,\infty)$ is a function (not necessarily continuous) such that $\rho(\epsilon)\geq \epsilon$ for all $\epsilon$ and $\rho(\epsilon)\to 0$ as $\epsilon\to 0$.

Recall that a family $\mathcal{F}$ of compact metric spaces is said to be \emph{precompact} (with respect to the Gromov-Hausdorff distance) if, for any $\epsilon>0$, $\mathcal{F}$ can be covered by finitely many open $\epsilon$-balls.
Given $\epsilon>0$ and a compact metric space $X$, let $\mathrm{C}(X,\epsilon)$ be the minimal number of $\epsilon$-balls that cover $X$. 
It follows from Proposition 5.2 of \cite{gromov1999metric} that for a precompact family $\mathcal{F}$ of compact metric spaces and any $\epsilon>0$, $$\mathrm{C}_\mathcal{F}(\epsilon):=\sup_{X\in \mathcal{F}}\mathrm{C}(X,\epsilon)$$ is finite, i.e. every space in the family is covered by at most $\mathrm{C}_\mathcal{F}(\epsilon)$ $\epsilon$-balls. 
Let 
\begin{itemize}
\item $\epsilon_\rho \left(\leq \tfrac{R}{2}\right)$ be the largest $\epsilon>0$ such that $\rho(2\epsilon)\leq R$, i.e.
$$\epsilon_\rho:=\sup\{\epsilon \in[0,R]|\,\rho(2\epsilon)\leq R\}.
$$ 
\item $G:\N\rightarrow \N$ be the function given by $$G(n):=n^{\frac{n(n-1)}{2}+1}.$$ 
\end{itemize}

We establish the following finiteness result for $\pi_1(\bullet)$.

\begin{proposition}[Finiteness result for $\pi_1(\bullet)$]\label{prop:finiteness_pi1}
Let $\mathcal{F}$ be a precompact family of compact geodesic $\operatorname{LGC_1}(\rho,R)$ metric spaces. Define an equivalence relation on $\mathcal{F}$ via $X\approx Y$ iff  $\pi_1(X)\cong \pi_1(Y)$. 
Then, 
\begin{equation}\label{eq:our bound}
\card\left(\mathcal{F}/\approx\right)\leq G\left(\mathrm{C}_\mathcal{F}\left(\tfrac{\epsilon_\rho}{12}\right)\right).
\end{equation}
\end{proposition}

Proposition \ref{prop:finiteness_pi1} provides an explicit bound on the number of different isomorphism classes of fundamental groups. Let us compare this result with what can be derived from Theorem 1 of \cite{plaut2013discrete}, i.e. Eq. (\ref{eq:plaut bound}) in the remark below. 

\begin{remark}[Plaut and Wilkins' finiteness result for $\pi_1(\bullet)$]\label{rmk:finiteness_pi1}
Let $\mathcal{F}$ be as in Proposition \ref{prop:finiteness_pi1}. By Theorem 1 of \cite{plaut2013discrete},
\begin{equation}\label{eq:plaut bound}
\card\left(\mathcal{F}/\approx\right)\leq 2^{\left(\left(\frac{2D}{\epsilon_\rho}+8\right)\mathrm{C}_\mathcal{F}\left(\epsilon_\rho\right)^{\frac{2D}{\epsilon_\rho}+8}\right)^3}. 
\end{equation}

Eq. (\ref{eq:plaut bound}) holds because of the following. First, there is clearly a uniform upper bound on the diameters of spaces in $\mathcal{F}$, which we denote by $D$. For any $X\in \mathcal{F}$, let $\sigma(X)$ be the infimum of lengths of non null-homotopic closed geodesics in $X$, and recall from Theorem 1 of \cite{plaut2013discrete} that $\pi_1(X)$ has a set $\{g_i\}_i$ of generators with at most 
\begin{equation}\label{eq:num of g_i}
    \left(\frac{8D}{\sigma(X)}+1\right)\mathrm{C}\left(X,\frac{\sigma(X)}{4}\right)^{\frac{8D}{\sigma(X)}+1}
\end{equation}
elements with relations of the form $g_ig_j=g_k$. 

We claim that for any non-simply-connected $X\in \mathcal{F}$, 
\[\sigma(X)\geq 4\epsilon_\rho.\]
Indeed, 
because there is a non null-homotopic loop $\gamma$ in $X$ with length $\sigma(X)>0$, for arbitrary $t>0$, we can find open balls of radius $\frac{\sigma(X)}{2}+t$ (e.g. take such balls to be centered at points in $\gamma$) which are not $1$-connected in $X$ and thus are not $1$-connected in any $\rho\big(\frac{\sigma(X)}{2}+t\big)$-ball. 
So we have
$\frac{\sigma(X)}{2}+t>R\geq 2\epsilon_\rho$ for any $t>0$. Thus, $\sigma(X)\geq 4\epsilon_\rho.$ 

Since Eq. (\ref{eq:num of g_i}) is non-increasing with respect to $\sigma(X)$ and $\sigma(X)\geq 4\epsilon_\rho$, Eq. (\ref{eq:num of g_i}) is upper bounded by
\[N_{\mathcal{F}}:=\left(\frac{2D}{\epsilon_\rho}+8\right)\cdot \mathrm{C}_\mathcal{F}\,\left(\epsilon_\rho\right)^{\frac{2D}{\epsilon_\rho}+8},\]
Since $\card(\mathcal{F}/\approx)\leq 2^{(N_{\mathcal{F}})^3}$, this implies Eq. (\ref{eq:plaut bound}).
\end{remark}

\medskip

We now consider a case when our bound (i.e. Eq. (\ref{eq:our bound})) is tighter than Eq. (\ref{eq:plaut bound}).

\begin{example} \label{ex:finiteness} 
Consider the family $\mathcal{F}$ in Proposition \ref{prop:finiteness_pi1} to be a precompact family of compact $\operatorname{LGC_1}(\rho,R)$ $d$-dimensional Riemannian manifolds with non-negative Ricci curvature.  
We now see that this family $\mathcal{F}$ is one for which our upper bound is (asymptotically) smaller than the upper bound obtained via Theorem 1 from \cite{plaut2013discrete}. 

\medskip
Let $0<\epsilon\ll \min\{\epsilon_\rho,1\}$. Because the right-hand side of Eq. (\ref{eq:our bound}) and is non-increasing with respect to $\epsilon_\rho$, replacing $\epsilon_\rho$ with $\epsilon$ gives us the following:
\[\card\left(\mathcal{F}/\approx\right)\leq G\left(\mathrm{C}_\mathcal{F}\left(\tfrac{\epsilon}{12}\right)\right)= 2^{\frac{n^2+n}{2}\log_2 n}=:A, \]
where $n:=\mathrm{C}_\mathcal{F}\left(\frac{\epsilon}{12}\right)$. By a similar reason, Eq. (\ref{eq:plaut bound}) implies that 
\[
\card\left(\mathcal{F}/\approx\right)\leq 2^{m^{3\delta}\delta^3}=:B,
\]
where $ m:=\mathrm{C}_\mathcal{F}\left(\epsilon\right)$ and $\delta:=\frac{2D}{\epsilon}+8$. Here $D<\infty$ is a uniform upper bound on the diameter of spaces in $\mathcal{F}$.

For the quantities $n,m$ and $\delta$, we have the following asymptotic estimates since $0<\epsilon\ll 1$:
\[n=\mathrm{C}_\mathcal{F}\left(\frac{\epsilon}{12}\right)\sim \epsilon^{-d}\sim \mathrm{C}_\mathcal{F}\left(\epsilon\right)=m\,\,\text{  and  }\,\,\delta=\frac{2D}{\epsilon}+8 \sim \epsilon^{-1}.\]
Applying the above to $A$ and $B$ respectively, we obtain
\[A= 2^{\frac{n^2+n}{2}\log_2 n} < 2^{n^{3}}
\,\ll\,   
2^{n^{3\delta}\delta^3}\sim 2^{m^{3\delta}\delta^3} = B.\]
Therefore, our upper bound, $A$, for $\card(\mathcal{F}/\approx)$ is (asymptotically) smaller than $B$, the one derived from Theorem 1 from \cite{plaut2013discrete}.
\end{example}

\paragraph*{The proof of Proposition \ref{prop:finiteness_pi1}.}
Below, we state a lemma (Lemma \ref{lem:finiteness}) that has been implicitly used in several articles, such as \cite{yamaguchi1988homotopy,peter1990finiteness}. Since we were not able to find a complete proof of this lemma in the literature, we provide one in \textsection \ref{app}.

\begin{lemma}\label{lem:finiteness}
 Let $\mathcal{F}$ be a precompact family of connected compact metric spaces. Then, for any $\epsilon>0$, $(\mathcal{F},\dgh)$ can be covered by $G\left(\mathrm{C}_\mathcal{F}\left(\tfrac{\epsilon}{3 }\right)\right)$ many closed $\epsilon$-balls. 
\end{lemma}
 
Proposition \ref{prop:finiteness_pi1} follows from Lemma \ref{lem:finiteness} and Lemma \ref{lem:same_pi_1} below.

\begin{lemma}\label{lem:same_pi_1}
 Let the metric spaces $X$ and $Y$ be geodesic and $\operatorname{LGC_1}(\rho,R)$. 
 If $\dgh(X,Y)<\tfrac{1}{4}\epsilon_\rho$, then $X$ and $Y$ have isomorphic fundamental groups.
\end{lemma}

\begin{remark}
It follows from the theorem on page 392 of \cite{peter1990finiteness} that if $\dgh(X,Y)<\epsilon$ for $\epsilon>0$ s.t. $\rho_1(18\epsilon+8\rho_1(4\epsilon))<R$, then $X$ and $Y$ are homotopy equivalent. This does not imply Lemma \ref{lem:same_pi_1}, where $\epsilon$ is only required to satisfy the weaker condition that $\rho_1(8\epsilon)<R$. 

This condition is indeed less stringent than $\rho_1(18\epsilon + 8 \rho_1(\epsilon))<R$ since $\rho_1$ is non-decreasing: for any $\epsilon$ such that $\rho_1(18\epsilon + 8 \rho_1(\epsilon))<R$ we also have $\rho_1(8\epsilon)\leq \rho_1(18\epsilon )\leq \rho_1(18\epsilon + 8 \rho_1(\epsilon))<R$.
For example, if $\rho:[0,R]\to [0,\infty)$ is given by $\rho(\epsilon)=C\cdot\epsilon$ for some constant $C>0$, then $\rho_1(\epsilon)=(C+1)\epsilon$, cf. page 389 of \cite{peter1990finiteness}. In this case, $\rho_1(8\epsilon)<R$ iff $\epsilon<\frac{R}{8(C+1)}$, while $\rho_1(18\epsilon + 8 \rho_1(\epsilon))<R$ iff $\epsilon<\frac{R}{(8C+26)(C+1)}\left(<\frac{1}{3}\cdot\frac{R}{8(C+1)}\right)$.
\end{remark}

\begin{myproof}{Lemma \ref{lem:same_pi_1}}
Assume $\pi_1(X)$ is not isomorphic to $\pi_1(Y)$. We first point out that
\[\epsilon_\rho\leq  \min\{\operatorname{sep}(\ppi_1(X)),\operatorname{sep}(\ppi_1(Y))\}\leq 2\cdot \di (\ppi_1(X),\ppi_1(Y)),\]
where the leftmost inequality follows from Theorem \ref{thm:discretization} and the rightmost inequality follows from Proposition \ref{prop:sep_of_dendrogram}. Combined with Theorem \ref{thm:stab-di-ppi_1} the stability of persistent fundamental groups, which will be proved in \S \ref{sec:stability-ppi_n}, we have
\[\tfrac{1}{4} \epsilon_\rho\leq  \tfrac{1}{2}\cdot\di (\ppi_1(X),\ppi_1(Y))\leq \dgh(X,Y)< \tfrac{1}{4}\epsilon_\rho,\]
which gives a contradiction.
\end{myproof}

\subsection{Second proof of stability of \texorpdfstring{$\ppi_1(\bullet)$}{persistent fundamental groups}}\label{sec:stability-second-proof}	
In this section we provide an alternative proof of Theorem \ref{thm:stab-di-ppi_1} inspired by Wilkins' work \cite{wilkins2011discrete}. This alternative proof is more constructive and independent of the notion of persistent $\rmK$-fundamental groups. Let us recall the stability theorem here:
\stabofppi*

Based on Lemma 6.3.3 of \cite{wilkins2011discrete}, we first introduce a way to construct maps between discrete fundamental groups and prove Lemma \ref{lem:epsi-del-homo}. 

\begin{definition}[$(\epsilon,\delta)$-homomorphisms] \label{def:epsilon homomorphism by a tripod} 
Let $X$ and $Y$ be compact metric spaces with basepoints $x_0$ and $y_0$, respectively. Let $$R:X\xtwoheadleftarrow{\phi_X}Z\xtwoheadrightarrow{\phi_Y}Y$$ be a pointed tripod (i.e., there exists $z_0\in Z$ such that $\phi_X(z_0)=x_0$ and $\phi_Y(z_0)=y_0$) with $\dis(R)\leq \epsilon$. Since $\phi_Y$ is surjective, for each $\delta>0$ and each $\delta$-loop $\beta:[n]\to Y$, there is a discrete loop $\gamma_\beta$ in $Z$ such that $\phi_Y\circ\gamma_\beta=\beta$ and $\gamma_\beta(0)=\gamma_\beta(n)=z_0$, as expressed by the following diagram:
	\begin{center}
		\begin{tikzcd} 
			&Z
			\ar[dl,"\phi_X" above left,twoheadrightarrow ]
			\ar[dr,"\phi_Y",twoheadrightarrow]
			&
			\\
			X&
		    {[n]}
			\ar[u,"\exists \gamma_\beta" right]
			\ar[l, dashrightarrow ]
			\ar[r,"\beta" below] 
			& 
			Y
		\end{tikzcd}
	\end{center}
Using the above, we define a map $$\psi_{\delta}^{\delta+\epsilon}:\pi_1^{\delta}(Y,y_0)\rightarrow\pi_1^{\epsilon+\delta}(X,x_0)\text{ such that } [\beta]_{\delta}\mapsto [\phi_Y\circ\gamma_\beta]_{\epsilon+\delta},$$
and similarly a map
$$\phi_{\delta}^{\delta+\epsilon}:\pi_1^{\delta}(Y,y_0)\rightarrow\pi_1^{\epsilon+\delta}(X,x_0)\text{ such that } [\alpha]_{\delta}\mapsto [\phi_X\circ\gamma_\alpha]_{\epsilon+\delta}.$$ 
We call the maps $\psi_{\delta}^{\delta+\epsilon}$ and $\phi_{\delta}^{\delta+\epsilon}$ the \emph{$(\epsilon,\delta)$-homomorphisms induced by the tripod $R$}.
\end{definition}

In \cite{wilkins2011discrete}, Wilkins' construction of group homomorphisms between discrete fundamental groups is induced by some isometric embeddings of $X$ and $Y$ into a common metric space $Z$, and is applicable to any $\epsilon>2 \dhaus^Z(X,Y)$. Notice that given a tripod $R$ between metric spaces $X$ and $Y$, $R$ induces a metric space $Z_R$ which $X$ and $Y$ are isometrically embedded into, with $\dhaus^{Z_R}(X,Y)=2\dis(R)$. We will prove in the next lemma that Definition \ref{def:epsilon homomorphism by a tripod} defines group homomorphisms for all $\epsilon> 2\dis(R)=\dhaus^{Z_R}(X,Y)$, which is an improvement over Wilkins' construction.

\begin{lemma}\label{lem:epsi-del-homo} 
With the notation from Definition \ref{def:epsilon homomorphism by a tripod}, the $(\epsilon,\delta)$-homomorphisms induced by a tripod $R$: \[\psi_{\delta}^{\delta+\epsilon}:\pi_1^{\delta}(Y,y_0)\rightarrow \pi_1^{\delta+\epsilon}(X,x_0)\text{ and }\phi_{\delta}^{\delta+\epsilon}:\pi_1^{\delta}(X,x_0)\rightarrow \pi_1^{\delta+\epsilon}(Y,y_0)\] are well-defined group homomorphisms. Furthermore, $\ppi_1(X)$ and $\ppi_1(Y)$ are $\epsilon$-interleaved via $\psi_{\epsilon}:=\{\psi_{\delta}^{\delta+\epsilon}\}_{\delta>0}$ and $\phi_{\epsilon}:=\{\psi_{\delta}^{\delta+\epsilon}\}_{\delta>0}$.
\end{lemma}
	
\begin{proof} Fix $\delta>0$. For a $\delta$-loop $\beta$ in $Y$, since $\phi_Y$ is surjective, let $\gamma_\beta$ be a discrete loop in $Z$ such that $\phi_Y(\gamma_\beta)=\beta$ and $\gamma_\beta(0)=\gamma_\beta(n)=z_0$. Because $(\phi_X(\gamma_\beta),\beta)\in R_{\caL}$ and by Lemma \ref{lem:tripod}, $\phi_X(\gamma_\beta)$ is an $(\epsilon+\delta)$-loop in $X$. So $\psi_{\delta}^{\delta+\epsilon}$ maps to $\pi_1^{\epsilon+\delta}(X,x_0).$

\begin{claim} $\psi_{\delta}^{\delta+\epsilon}$ is independent of the choice of $\gamma_\beta$.\end{claim} 

Suppose $\beta=y_0y_1\cdots y_n$ with $y_n=y_0$ is a $\delta$-loop in $Y$. Suppose that $\gamma_\beta=z_0\cdots z_n$ and $\gamma_\beta'=z'_0\cdots z'_n $, where $z_n=z'_{n}=z'_0=z_0$ the basepoint of $Z$. For each $i\in [n]$, as $y_i=\phi_Y(z_i)=\phi_Y(z_i'),$ we have
	$$d_X(\phi_X(z_i),\phi_X(z_i'))\leq d_{Y}(y_i,y_i)+\dis(R)\leq \epsilon,$$
and for each $i=0,\cdots,n-1$, we have
	$$d_X(\phi_X(z_i),\phi_X(z_{i+1}'))\leq d_{Y}(y_i,y_{i+1})+\dis(R)\leq\delta+\epsilon.$$
Thus, inserting $\phi_X(z_i')$ between $\phi_X(z_{i-1})$ and $\phi_X(z_i)$ is a basic move up to $(\delta+\epsilon)$-homotopy, for all $i=1,\cdots,n$. This results into the following $(\delta+\epsilon)$-chain in $X$:
$$\phi_X(z_0)\phi_X(z_1')\phi_X(z_1)\cdots \phi_X(z_n')\phi_X(z_n).$$
By sequentially removing $\phi_X(z_i)$ for each $i=1,\cdots,n$, we obtain a $(\delta+\epsilon)$-homotopy from $\phi_X(\gamma_{\beta})$ to $\phi_X(\gamma_{\beta}')$. Hence, $[\phi_X(\gamma_{\beta})]_{\delta+\epsilon}=[\phi_X(\gamma_{\beta}')]_{\delta+\epsilon}$. 
		
\begin{claim} $\psi_{\delta}^{\delta+\epsilon}$ is independent of the choice of $\beta$.\end{claim} 

Suppose $\beta\sim_1^{\delta}\beta'$ in $Y$, where $\beta=y_0y_1\cdots y_n$ and $\beta'=y_0'y_1'\cdots y_m'$ with $y_n=y_m'=y_0'=y_0$. Then, there is a $\delta$-homotopy in $Y$:     $$H_Y=\left\{\beta=\beta_0,\beta_1,\cdots,\beta_{k-1},\beta_k=\beta'\right\}$$ 
such that each $\beta_j$ differs from $\beta_{j-1}$ by a basic move. For each $\beta_j$, let $\gamma_{\beta_j}$ be a discrete loop in $Z$ such that $\phi_Y(\gamma_{\beta_j})=\beta_j$ and $\gamma_{\beta_j}(0)=\gamma_{\beta_j}(n)=z_0$. Then
	$$H_X=\left\{\phi_X(\gamma_{\beta_0}),\phi_X(\gamma_{\beta_1}),\cdots,\phi_X(\gamma_{\beta_{k-1}}),\phi_X(\gamma_{\beta_k})\right\}$$ 
is an $(\epsilon+\delta)$-homotopy between $\phi_X(\gamma_{\beta})$ and $\phi_X(\gamma_{\beta'})$. Therefore, \[\psi_{\delta}^{\delta+\epsilon}([\beta]_{\delta})= [\phi_X(\gamma_{\beta})]_{\epsilon+\delta}=[\phi_X(\gamma_{\beta'})]_{\epsilon+\delta}=\psi_{\delta}^{\delta+\epsilon}([\beta']_{\delta}).\]
		
\begin{claim} $\psi_{\delta}^{\delta+\epsilon}$ is a group homomorphism.\end{claim} 

Take two $\delta$-loops $\beta$ and $\beta'$ in $Y$, and choose $\gamma_\beta$ and $\gamma_{\beta'}$ so that $\psi_{\delta}^{\delta+\epsilon}([\beta]_{\delta})=[\phi_X(\gamma_{\beta})]_{\epsilon+\delta}$ and $\psi_{\delta}^{\delta+\epsilon}([\beta']_{\delta})=[\phi_X(\gamma_{\beta'})]_{\epsilon+\delta}$ as before. It follows from $\phi_X(\gamma_{\beta}\ast \gamma_{\beta'})=\phi_X(\gamma_{\beta})\ast \phi_X(\gamma_{\beta'})$ that $\left(\phi_X(\gamma_{\beta})\ast \phi_X(\gamma_{\beta'}),\beta\ast \beta'\right)\in R_{\caL}$. Thus,
	$$\psi_{\delta}^{\delta+\epsilon}([\beta\ast\beta']_{\delta})=[\phi_X(\gamma_{\beta})\ast \phi_X(\gamma_{\beta'})]_{\epsilon+\delta}=\psi_{\delta}^{\delta+\epsilon}([\beta]_{\delta})\psi_{\delta}^{\delta+\epsilon}([\beta']_{\delta}).$$
		
\begin{claim} $\ppi_1(X)$ and $\ppi_1(Y)$ are $\epsilon$-interleaved via $\psi_{\epsilon}$ and $\phi_{\epsilon}$.\end{claim} 

Note that the following diagram commutes for all $\delta\leq \delta'$:
		\begin{center}
			\begin{tikzcd} 
				\ppi_1^{\delta}(Y) \ar[r, "\psi_{\epsilon}"] 
				\ar[d, "\Phi^Y_{\delta,\delta'}" left]
				& 
				\ppi_1^{\delta+\epsilon}(X) \ar[d, "\Phi^X_{\delta+\epsilon,\delta'+\epsilon}"]
				\\
				\ppi_1^{\delta'}(Y)\ar[r ,"\psi_{\epsilon}" below]
				& 
				\ppi_1^{\delta'+\epsilon}(X).
			\end{tikzcd}
\end{center} Indeed, given a $\delta$-loop $\beta$ in $Y$, let $\gamma_\beta$ be a discrete loop in $Z$ chosen as before. Then,
\begin{align*}\psi_{\epsilon}\circ\Phi^Y_{\delta,\delta'}([\beta]_{\delta})&=\psi_{\epsilon}([\beta]_{\delta'})=[\phi_X(\gamma_\beta)]_{\delta'+\epsilon},\\
\Phi^X_{\delta+\epsilon,\delta'+\epsilon}\circ \psi_\epsilon([\beta]_{\delta})&=\Phi^X_{\delta+\epsilon,\delta'+\epsilon}([\phi_X(\gamma_\beta)]_{\delta+\epsilon})=[\phi_X(\gamma_\beta)]_{\delta'+\epsilon}.
\end{align*}
Therefore, $\psi_{\epsilon}$ defines a homomorphism of degree $\epsilon$ from $\ppi_1(Y)$ to $\ppi_1(X))$, and similarly $\phi_{\epsilon}$ defines a homomorphism of degree $\epsilon$ from $\ppi_1(X)$ to $\ppi_1(Y)$. Also, by directly checking the relevant definitions, it is not hard to see that $\psi_{\epsilon}\circ\phi_{\epsilon}=1_{\ppi_1(X)}^{2\epsilon}$ and $\phi_{\epsilon}\circ\psi_{\epsilon}=1_{\ppi_1(Y)}^{2\epsilon}.$ 
\end{proof}	

It is clear that Lemma \ref{lem:epsi-del-homo} immediately implies Theorem \ref{thm:stab-di-ppi_1}.

\section{Discussion and Open Problems}

\paragraph{Computability questions.} Due to Theorem \ref{thm:iso of d.f.g}, computing discrete fundamental groups is equivalent to computing edge-path groups on clique complexes. According to Theorem 4.6 and Theorem 5.16 of \cite{lupton2019digital}  all finitely presented groups   
can arise as edge-path groups of (finite) clique complexes. Since the \emph{word problem} is undecidable even for finitely presented groups \cite{novikov1955algorithmic,boone1959word} (the problem of deciding whether an element is trivial in a group), computing persistent/standard fundamental groups is also expected to be difficult. 

In terms of existing work related to the computation of fundamental groups, effort has been put into computing presentations \cite{brendel2015computing} (in the standard case) or invariants \cite{Letscher:2012:PHK:2090236.2090270} (in the persistent setting) for fundamental groups. Another thread is to restrict to certain families of spaces with fundamental groups admitting decidable word problems, such as compact 3-manifolds whose fundamental groups are finitely presented residually finite groups which have a decidable word problem (see page 31 of \cite{miller1992decision}). Computing persistent fundamental groups of such spaces is then an interesting and possibly manageable question.

\paragraph{Discrete versions of higher persistent homotopy groups.}
Because the persistent fundamental group $\ppi_1(\bullet)$ agrees with the persistent $\rmK$-homotopy groups $\ppi_n^{\rmK}(\bullet)$ (thus, also with persistent $\VR$-homotopy groups) for $n=1$, we can regard $\ppi_n^{\rmK}(\bullet)$ as generalizations of $\ppi_1(\bullet)$ to higher dimensions, cf. Theorem \ref{thm:iso-ppi_1's}.  
Since our notion of $\ppi_1(\bullet)$ arose from discrete homotopy theory ideas, it is natural to try to lift higher-dimensional the discrete homotopy groups constructed by Barcelo et al. \cite{barcelo2005perspectives} to a persistent version and see how they relate to $\ppi_n^{\rmK}(\bullet)$ or other choices of persistent homotopy groups arising from topological filtrations. 

\paragraph{Vietoris-Rips filtration v.s. Kuratowski filtration under closed convention.} In Corollary \ref{cor:iso-ppi_n}, we saw that $\ppi_n^{\rmK}(\bullet)$ and $\ppi_n^{\VR}(\bullet)$ are isomorphic when considering the open convention. Notice that the closed convention can be more informative than the open one, such as in the case of the unit geodesic circle \cite{AdamA15}.  
It is then natural to ask whether the isomorphism holds for the closed convention in general, and a further question is whether for non-finite metric spaces the closed Vietoris-Rips filtration and the closed Kuratowski filtration are homotopy equivalent.

\paragraph{Rationalization of persistent homotopy groups.} We considered the rationalization of persistent homotopy groups  as a weaker but 
more manageable alternative to persistent homotopy groups. We proved the stability for all persistent (rational) homotopy groups, and established some geometric applications of persistent fundamental groups, e.g. the finiteness result in Proposition \ref{prop:finiteness_pi1}.

\paragraph{Tree structure on higher persistent homotopy groups.}
We have generalized the tree-structure in the $0$-th persistent homotopy to dimension $1$. How about higher dimensional persistent homotopy groups?
Via a similar strategy, the first step would be to identify conditions on the metric space $X$ such that the transition maps in its persistent homotopy groups are surjective, from which we will obtain a treegram structure. To reach a dendrogram structure, in particular asking for a dendrogram over $\pi_n(X)$, further constraints on $X$ should be identified so that all homotopy classes in $\VR_\epsilon(X)$ (or some other alternative geometric filtration) are induced by some homotopy classes in $X$.

\appendix


\section*{Acknowledgements}
We thank Prof. C. Plaut for bringing to our attention his results \cite{plaut2009equivalent} on connecting the discrete fundamental group construction in \cite{wilkins2011discrete} to the Vietoris-Rips complex (cf. Theorem \ref{thm:iso-ppi_1's}). We thank Prof. R. Jardine for sharing his work on persistent fundamental groupoids and Gunnar Carlsson for his suggestion that we consider (persistent) rational homotopy groups. We also thank Dr. S. Lim for helping us with the proof of Lemma \ref{lem:finiteness}.

This research was supported by NSF under grants DMS-1723003, CCF-1740761, and CCF-1526513.

\section{Proof of Lemma \ref{lem:finiteness}}
\label{app}

\begin{myproof}{Lemma \ref{lem:finiteness}} The proof is modeled from Proposition 5.2 of \cite{gromov1999metric}.\footnote{This particular proof strategy was suggested by Dr. Sunhyuk Lim.}  

Let $\epsilon':=\tfrac{\epsilon}{3 }$, $C:=\mathrm{C}_\mathcal{F}\left(\epsilon'\right)$ and $m\in\{1,\dots,C\}$. Let $\mathcal{F}_m$ be the subfamily of $\mathcal{F}$ consisting of spaces $X\in \mathcal{F}$ such that $\mathrm{C}(X,\epsilon')=m.$
For any $X\in \mathcal{F}_m$, let $P_X\subset X$ be an $\epsilon'$-net of $X$ with $\card(P_X)=m$. \\

\noindent \textit{Claim 1:} for any $x,x'\in P_X$, $d_X(x,x')\in \left[0,2\epsilon'm\right]$. Thus, $P_X$ induces a vector $v_X\in \left[0,2\epsilon'm\right]^{\frac{m^2-m}{2}}$.  

Because $X$ is connected, for any $x\in P_X$, the ball $B(x,\epsilon)$ intersects with some other ball $B(x',\epsilon)$ for $x'\in P_X$. In this case, $d_X(x,x')<2\epsilon$. Thus, for arbitrary $x,x'\in P_X$, there is a sequence of distinct $\epsilon$-balls centered at points in $P_X$ that connects $B(x,\epsilon)$ and $B(x',\epsilon)$. It follows that $d_X(x,x')<2\epsilon m.$
Fix an arbitrary ordering on elements in $P_X$, and define a vector $v_X\in \left[0,2\epsilon'm\right]^{\frac{m^2-m}{2}}$ by ordering $d_X(x_i,x_j)$ lexicographically by $(i,j)$, for all $1\leq i<j\leq m$. \\

\noindent \textit{Claim 2:} if there exist $v_X$ for $X$ and $v_Y$ for $Y$ such that $\|v_X-v_Y\|_{\infty}\leq \epsilon'$, then $ \dgh(X,Y)\leq \epsilon$. 

If $\|v_X-v_Y\|_{\infty}< \epsilon'$, then \[ \dgh(P_X,P_Y)\leq  \min_{\mathrm{bijection}\, f:X\to Y}\dis(f)< \epsilon' .\] 
Thus, $ \dgh(X,Y)\leq  \dgh(X,P_X)+ \dgh(P_X,P_Y)+ \dgh(P_Y,Y)<3 \cdot\epsilon'=\epsilon.$\\

\noindent \textit{Claim 1:} $(\mathcal{F}_m,d)$ can be covered by $ m^{\frac{m^2-m}{2}}$ many $\epsilon$-balls, for each $m=1,\dots,C$. 

It is clear that the closed $\epsilon'$-ball covering number of $\left[0,2\epsilon'm\right]^{\frac{m^2-m}{2}}$ (equipped with the $\ell^\infty$-metric) is $m^{\frac{m^2-m}{2}}$. Thus,
it follows from Claim 2 that $\mathcal{F}_m$ can be covered by no more than $m^{\frac{m^2-m}{2}}$ many $\epsilon$-balls. 

\medskip
Recall that $C=\mathrm{C}_\mathcal{F}\left(\epsilon'\right)$ and $\epsilon':=\tfrac{\epsilon}{3 }$. By Claim 3, we see that $(\mathcal{F},d)$ can be covered by $G\left(\mathrm{C}_\mathcal{F}\left(\tfrac{\epsilon}{3 }\right)\right)$ many closed $\epsilon$-balls, since
\[
\sum_{m=1}^C m^{\frac{m^2-m}{2}}\leq C^{\frac{C^2-C}{2}+1}=G\left(\mathrm{C}_\mathcal{F}\left(\epsilon'\right)\right)=G\left(\mathrm{C}_\mathcal{F}\left(\tfrac{\epsilon}{3 }\right)\right).\]
\end{myproof}

\section{Application to finite metric spaces
}\label{sec:app to finite}

In this section, we analyze several examples of finite metric spaces to further understand the treegrams obtained in Lemma \ref{lem:treegramofL}. In general, given a pointed metric space $(X,x_0)$, the set of discrete loops $\caL(X,x_0)$ is in bijection with the set $X^{\N}=\{\{x_i\}_{i\in \N}:x_i\in X\}$. Even when $X$ is a finite metric space, $\caL(X,x_0)$ can have large cardinality, making the illustration of its treegram rather difficult. To simplify the graphical representation of the resulting treegrams, we introduce the equivalence relation $\sim$ on $\caL(X,x_0)$ given by: \label{para:relation sim}
\begin{center}
    $\gamma\sim\gamma'$ iff $\gamma$ and $\gamma'$ have the \emph{same} birth time $\delta\geq0$ and $\gamma\sim_1^{\delta}\gamma'$,
\end{center} which implies that $\gamma\sim_1^{\epsilon}\gamma'$ for all $\epsilon\geq \delta$. Clearly, $\sim$ is an equivalence relation and we denote  \label{para:L space}
$$L(X,x_0):=\caL(X,x_0)\,/\sim,$$
and write $[\gamma]$ for the equivalence class of a discrete loop $\gamma$ under the relation $\sim$. Let $p_X:\caL(X,x_0)\twoheadrightarrow L(X,x_0)$ be the resulting quotient map, i.e., $p_X(\gamma)=[\gamma]$. It is not hard to see that $L(X,x_0)$ is a monoid under the operation $p_X(\gamma)\bullet p_X(\gamma'):=p_X(\gamma\ast\gamma').$ The existence of inverse elements is not guaranteed because $\gamma\ast\gamma^{-1}$, where $\birth(\gamma)=\delta>0$, is a $\delta$-loop and cannot be equivalent to the $0$-loop $x_0$. Let $L^{\epsilon}(X,x_0)$ be the set of equivalence classes of discrete loops in $X$ with birth time $\epsilon$, i.e., \label{para:L epsilon space} $$L^{\epsilon}(X,x_0):=\left(\caL^{\epsilon}(X,x_0)-\bigcup_{\delta<\epsilon}\caL^{\delta}(X,x_0)\right)\Big/\sim.$$ 
It is clear that $L^{\epsilon}(X,x_0)$ is a semigroup, i.e., it is closed under the operation $\bullet$ which satisfies the associative property. Furthermore, $L^{\epsilon}(X,x_0)$ is a sub-semigroup of $L(X,x_0)$.

It can be directly checked that the constructions and results in \textsection \ref{sec:treegram-pseudo-metric} all apply here:
\begin{itemize}
    \item The pseudo-ultra-metric $\mu_X^{(1)}$ on $\caL(X,x_0)$ induces a pseudo-ultra-metric on $L(X,x_0)$, still denoted by $\mu_X^{(1)}$. In particular, for any two discrete loops $\gamma$ and $\gamma'$, $$\mu_X^{(1)}\left([\gamma],[\gamma']\right):=\mu_X^{(1)}(\gamma,\gamma').$$
    \item  Let $(X,x_0)$ and $(Y,y_0)$ be pointed compact metric spaces. Then, each pointed tripod $$R:X\xtwoheadleftarrow{\phi_X}Z\xtwoheadrightarrow{\phi_Y}Y$$ induces the following tripod between $L(X,x_0)$ and $L(Y,y_0)$:
    $$R_{L}:=L(X,x_0)\xtwoheadleftarrow{p_X\circ\phi_X}\caL(Z,z_0)\xtwoheadrightarrow{p_Y\circ\phi_Y}L(Y,y_0) ,$$  
    with $\dis(R)\leq \dis(R_{L})$.
    \item  Given $(X,x_0)$ and $(Y,y_0)$ be pointed compact metric spaces, 
		$$\dgh(L(X,x_0) ,L(Y,y_0))\leq \dgh^{\pt}((X,x_0),(Y,y_0)).$$
    \item Let $X$ be a compact geodesic space or a finite metric space. The map $\epsilon\mapsto \pi_1^{\epsilon}(X,x_0)$ induces a treegram over $L(X,x_0)$, denoted by $\theta_{L(X,x_0)}^\rms$.
\end{itemize}

\begin{proposition}\label{prop:L-epsion-subsemigroup} Given a pointed compact metric space $(X,x_0)$ and $\epsilon>0$, if $L^{\epsilon}(X,x_0)$ is non-empty, then $L^{\epsilon}(X,x_0)$ is isomorphic to a sub-semigroup of $\pi_1^{\epsilon}(X,x_0)$.
\end{proposition}

\begin{proof} It suffices to show the map $f:L^{\epsilon}(X,x_0)\to\pi_1^{\epsilon}(X,x_0)$, with $[\gamma]\to [\gamma]_{\epsilon}$ is an injective semigroup homomorphism. Given discrete loops $\gamma$ and $\gamma'$ with birth time $\epsilon$ such that $\gamma\sim \gamma'$, we have $\gamma\sim^{\epsilon}_1\gamma'$. Thus, $f$ is well-defined. Clearly, $f$ preserve the semigroup operation. It remains to check that $f$ is injective. Indeed, if $f([\gamma])=f([\gamma'])$ for some $\gamma$ and $\gamma'$ with birth time $\epsilon$, then $\gamma\sim^{\epsilon}_1\gamma'$. Thus, we have $\gamma\sim \gamma'$. 
\end{proof}
The following corollary follows immediately:
\begin{corollary}\label{cor:L-epsilon} Let $(X,x_0)$ be a compact metric space and $\epsilon>0$. If $\pi_1^{\epsilon}(X,x_0)=0$ and $L^{\epsilon}(X,x_0)\neq\emptyset$, i.e., there exists a discrete loop $\gamma_\epsilon$ in $X$ with birth time $\epsilon$, then 
$$L^{\epsilon}(X,x_0)=\left\{[\gamma_{\epsilon}]\right\}.$$
\end{corollary}

\subsection{Finite metric spaces arising from graphs} \label{sec:finite-graph}
Next we study the discrete loop set and the persistent fundamental group of graphs. All our graphs are finite, undirected and simple (i.e., with no self-loops or multiple edges) graphs, with weight $1$ on each edge. Given a graph $G$, we can obtain a metric space by equipping the vertex set with the shortest path distance $d_G$. For simplicity, the resulting metric space is written as $(G,d_G)$, where the base space shall be understood as the vertex set of $G$. The \emph{girth} of a graph is the length of its shortest cycle or $\infty$ for a forest (see \cite{Adam13}).

\begin{proposition}[Fact 2.1 \& Proposition 2.2, \cite{Adam13}] \label{prop:ppi_1 of graph} For any connected graph $G$ and $r\geq 1$, the map of fundamental groups
$$\pi_1\left( \left|  \VR_{ 1}(G)\right| \right)\twoheadrightarrow \pi_1\left( \left|  \VR_{ r}(G)\right| \right)$$
induced by the inclusion $\VR_{ 1}(G)\hookrightarrow \VR_{ r}(G)$ is surjective. Furthermore, if $r\geq 1$ is such that $G$ is a graph of girth at least $3r+1$, then 
$$\pi_1\left( \left|  \VR_{ k-1}(G)\right| \right)\xrightarrow{\cong} \pi_1\left( \left|  \VR_{ k}(G)\right| \right)$$
for each $2\leq k\leq r$.
\end{proposition}

Let $n\in \Z_{\geq 3}$. A \emph{cycle graph} $C_n$ is a graph on $n$ vertices containing a single cycle through all its vertices. A \emph{star graph} $S_n$ is a tree on $n$ vertices where one vertex (called the \emph{center}) has degree $n-1$ and the others have degree $1$. Because of the symmetry, the set of discrete loops $\caL(C_n,v_0)$ and the persistent fundamental group $\ppi_1(C_n,v_0)$ do not depend on choices of the basepoint $v_0$. As for $S_n$, we always choose its center to be the basepoint. For convenience, we denote the vertex set $V(C_n)=\{0,1,\cdots,n-1\}$ such that $d_{C_n}(i,i+1)=1$, and denote the vertex set $V(S_n)=\{0,1,\cdots,n-1\}$, with $0$ the basepoint in both cases.

\begin{figure}[ht!]	  
\centering
\begin{tikzpicture}[scale=0.8]
\filldraw (-2,1) [color=blue] circle[radius=1.5pt];
\filldraw (-1.135,-0.5)[color=blue] circle[radius=1.5pt];
\filldraw (-2.866,-0.5) [color=blue] circle[radius=1.5pt];
\node[below right=0.4pt of {(-1.135,-0.5)}, outer sep=1.5pt,fill=white] {1};
\node[below left=0.4pt of {(-2.866,-0.5)}, outer sep=1.5pt,fill=white] {2};
\node[above=0.4pt of {(-2,1)}, outer sep=1.5pt,fill=white] {0};
\draw [dashed]  (-2,1)--(-1.135,-0.5);
\draw [dashed] (-2,1) --(-2.866,-0.5);
\draw [dashed] (-2.866,-0.5)--(-1.135,-0.5);

\filldraw (2,1) [color=blue] circle[radius=1.5pt];
\filldraw (1.135,-0.5)[color=blue] circle[radius=1.5pt];
\filldraw (2.866,-0.5) [color=blue] circle[radius=1.5pt];
\filldraw (2,0) [color=blue] circle[radius=1.5pt];
\node[below left=0.4pt of {(1.135,-0.5)}, outer sep=1.5pt,fill=white] {3};
\node[below right=0.4pt of {(2.866,-0.5)}, outer sep=1.5pt,fill=white] {2};
\node[above=0.4pt of {(2,1)}, outer sep=1.5pt,fill=white] {1};
\node[right=0.4pt of {(2,0.2)}, outer sep=1.5pt,fill=white] {0};
\draw [dashed]  (2,1)--(2,0);
\draw [dashed] (2,0) --(2.866,-0.5);
\draw [dashed] (2,0)--(1.135,-0.5);
    \end{tikzpicture} 
\caption{Graphs $C_3$ (left) and $S_4$ (right).} \label{fig:C3-S4}
\end{figure}
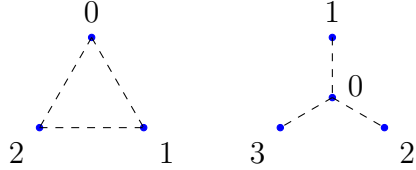
For later purpose, we introduce another metric space on the set of $n$ points, $E_n:=\left(\{0,1,\cdots,n-1\},d_{E_n}\right)$, where \label{para:E_n}
\[d_{E_n}(i,j)=\begin{cases}1,&\mbox{if $i\neq j$,}\\
0,&\mbox{if $i=j$}.
\end{cases}\]

\begin{theorem}[Corollary 6.6, \cite{Adam13}]\label{thm:VR-cycle-graph} For any $n\geq 3$ and $0\leq r<\tfrac{n}{2}$, there are homotopy equivalences
 \begin{equation*}
 \left| \VR_{ r}(C_n)\right| \cong \begin{cases}\bbS^{2l+1},&\mbox{if $\tfrac{l}{2l+1}<\tfrac{r}{n}<\tfrac{l+1}{2l+3}$ for some $l=0,1,\cdots$,}\\
 \bigvee^{n-2r-1}\bbS^{2l},&\mbox{if $\tfrac{r}{n}=\tfrac{l}{2l+1}$ for some $l=0,1,\cdots.$}
 \end{cases}
 \end{equation*}
For any integer $k\geq \tfrac{1}{2}n$, $\left| \VR_{ k}(C_n)\right| $ is contractible.
\end{theorem}

For $n\in \Z_{\geq 3}$ and $r=0,\cdots, \floor*{\tfrac{n}{2}}$, we denote the following $r$-loop in $C_n$ as \label{para:gamma_r}
$$\gamma_r:=0r(r+1)\cdots (n-1)0.$$
Although $\gamma_r$ depends on $n$, for simplicity of notation we do not specify $n$ unless necessary.

\begin{proposition}\label{prop:C_n} Fix $n\in \Z_{\geq 3}$. 
\begin{enumerate} 
\item  Let $ \floor*{x}$ denote the largest integer less than or equal to $x$. Then
$$\ppi_1(C_n)=\pH_1^{\VR}(C_n)=\Z\left[1,\floor*{\tfrac{n+2}{3}}\right).$$ 
Furthermore, for $r=1,\cdots, \floor*{\tfrac{n-1}{3}}$, a generator of $\ppi_1^{r}(C_n)$ is $[\gamma_r]_r$.
\item  We have pseudo-ultra-metric space
$$L(C_n)=\left\{[\gamma_0],[\gamma_r^{\ast k}],[\gamma_s]:r=1,\cdots, \floor*{\tfrac{n-1}{3}};s= \floor*{\tfrac{n+2}{3}},\cdots, \floor*{\tfrac{n}{2}}; k\in \Z-\{0\}\right\},$$ 
with the pseudo-metric $\mu^{(1)}_{C_n}$ given by: for $k,k'\in \Z-\{0\}$, 
\[\mu^{(1)}_{C_n}([\gamma_r^{\ast k}],[\gamma_{r'}^{\ast k'}])=\begin{cases}0,&\mbox{if $r=r'=0$}\\ 
\max\left\{\floor*{\tfrac{n-1}{3}},r\right\},&\mbox{if $r'=0$, $r\geq 1$}\\ 
\max\{r,r'\},&\mbox{otherwise.}
\end{cases}\]  
\end{enumerate}
\end{proposition}

\begin{proof} For Part (1), we first apply Theorem \ref{thm:VR-cycle-graph} to obtain that $\left| \VR_{ r}(C_n)\right| \cong \bbS^1$ iff $1\leq r< \floor*{\tfrac{n+2}{3}}$ and $\ppi_1^{r}(C_n)=\pH_1^{\VR,r}(C_n)=0$ otherwise. Clearly, a generator of $\ppi_1^1(C_n)$ is $[\gamma_1]_1$. Since the girth of $C_n$ is $n$, Proposition \ref{prop:ppi_1 of graph} implies that the inclusion $\VR_{ 1}(C_n)\hookrightarrow \VR_{ r}(C_n)$ induces an isomorphism $$\ppi_1^1(C_n)\to \ppi_1^r(C_n)\text{ with }[\gamma_1]_1\mapsto [\gamma_1]_r=[\gamma_r]_r,$$ for $r=1,\cdots, \floor*{\tfrac{n-1}{3}}$. Thus, $[\gamma_r]_r$ is a generator of $\ppi_1^{r}(C_n)$ for $r=1,\cdots,\floor*{\tfrac{n-1}{3}}$.

For (2), we first notice that the possible birth times of discrete loops in $C_n$ are $0,\cdots,\floor*{\tfrac{n}{2}}$ and each $\gamma_r$ has birth time $r$. For $r=0$ or $ \floor*{\tfrac{n+2}{3}},\cdots,\floor*{\tfrac{n}{2}}$, because Part (1) implies that $\ppi_1^{r}(C_n)=0$, by Proposition \ref{cor:L-epsilon} we then have
$$L^r(C_n)=\{[\gamma_r]\}.$$
When $r=1,\cdots, \floor*{\tfrac{n+2}{3}}$, we claim $$L^r(C_n)=\left\{[\gamma_i^{\ast k}]:k\in \Z-\{0\}\right\}.$$ It is clear that $\{[\gamma_r^{\ast k}]:k\in \Z-\{0\}\}\subset L^r(C_n)$, so it remains to show that any discrete loop with birth time $r$ is $r$-homotopic to $\gamma_r^k$ for some $k\in \Z-\{0\}$, which is true because $\ppi_1^{r}(C_n)$ is generated by $ [\gamma_r]_r$. The calculation of $\mu_{C_n}^{(1)}$ is straightforward, given Proposition \ref{prop:ppi_1 of graph}.
\end{proof}

\begin{remark}Note that $\left( L(C_n)\,,\mu^{(1)}_{C_n}\right)$ can be represented by the treegram depicted in Figure \ref{fig:treegram of Cn}, where the pseudo-ultra-metric induced by the treegram (see page \pageref{para:treegram}) agrees with $\mu^{(1)}_{C_n}$. 
\end{remark}

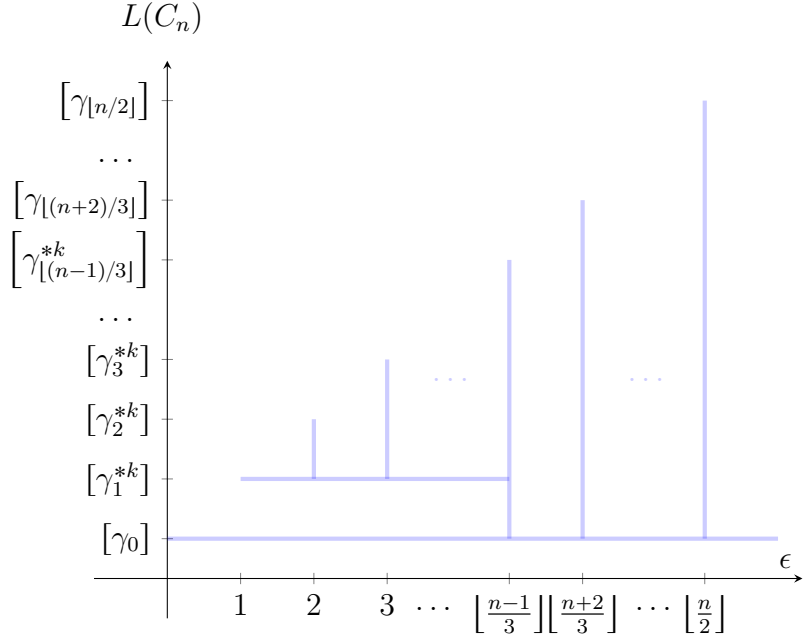
\begin{figure}[ht!]
\centering
    \begin{tikzpicture}
    \begin{axis} [ 
    height=9cm,
    width= 11cm,
    axis y line=middle, 
    axis x line=middle,
    xlabel=$\epsilon$,
    ylabel=$L(C_n)$,
    ytick={0.2,0.5,0.8,1.1,1.6,1.9,2.4},
    every axis y label/.style={at={(current axis.north west)},above right=3mm}, 
   yticklabels={$[\gamma_0]$, $[\gamma_1^{\ast k}]$,$[\gamma_2^{\ast k}]$,$[\gamma_3^{\ast k}]$,$\left[\gamma_{\floor*{(n-1)/3}}^{\ast k}\right]$,$\left[\gamma_{\floor*{(n+2)/3}}\right]$,$\left[\gamma_{\floor*{n/2}}\right]$ },
    xtick={0,0.3,0.6,0.9,1.4,1.7,2.2},
    xticklabels={0,1,2,3,$\floor*{\tfrac{n-1}{3}}$,$\floor*{\tfrac{n+2}{3}}$,$\floor*{\tfrac{n}{2}}$},
    xmin=-.3, xmax=2.6,
    ymin=-.2, ymax=2.6,]
    \addplot[domain=0:2.5,color=blue,opacity=0.2, ultra thick]{0.2};    
    \addplot[domain=0.3:1.4,color=blue, ultra thick,opacity=0.2]{0.5}; 
    \addplot [mark=none,color=blue, ultra thick,opacity=0.2] coordinates {(0.6,0.5) (0.6,0.8)};
    \addplot [mark=none,color=blue,  ultra thick,opacity=0.2] coordinates {(0.9,0.5) (0.9,1.1)};
    \addplot [mark=none,color=blue,  ultra thick, opacity=0.2] coordinates {(1.4,0.2) (1.4,1.6)};
    \addplot [mark=none,color=blue, ultra thick, opacity=0.2] coordinates {(1.7,0.2) (1.7,1.9)};
    \addplot [mark=none,color=blue, ultra thick, opacity=0.2] coordinates {(2.2,0.2) (2.2,2.4)};   
    \node[mark=none,color=blue, opacity=0.2] at (axis cs:1.17,1){$\dots$};    
    \node[mark=none,color=blue, opacity=0.2] at (axis cs:1.97,1){$\dots$};    
    \node[mark=none] at (axis cs:1.1,-0.15){$\dots$};
    \node[mark=none] at (axis cs:2,-0.15){$\dots$};
    \node[mark=none] at (axis cs:-0.2,2.1){$\dots$};
    \node[mark=none] at (axis cs:-0.2,1.3){$\dots$};
    \end{axis}
    \end{tikzpicture}
    \caption{Treegram over $L(C_n)$, where $k$ runs over non-zero integers for each $[\gamma_r^{\ast k}]$ when $r=1,\cdots, \floor*{\tfrac{n-1}{3}}$.} \label{fig:treegram of Cn}
\end{figure}

Now let us apply Proposition \ref{prop:C_n} to the cases $n=3$ and $4$. It follows that $\ppi_1(C_3)=\bbmzero,$ and $\left( L(C_3)\, ,\mu_{C_3}^{(1)}\right)$ is a two-point pseudo-metric space given by the distance matrix
\begin{equation*}
\mu_{C_3}^{(1)}=
  \begin{blockarray}{*{2}{c} l}
    \begin{block}{*{2}{>{$\footnotesize}c<{$}} l}
      $[\gamma_0]$ & $[\gamma_1]$ \\
    \end{block}
    \begin{block}{[*{2}{c}]>{$\footnotesize}l<{$}}
      0 & 1 & $[\gamma_0]$ \\
       & 1 & $[\gamma_1]$ \\
    \end{block}
  \end{blockarray}.
\end{equation*}
Similarly, we have $\ppi_1(C_4)=\Z[1,2).$ The corresponding treegrams over $L(C_3)$ and $L(C_4)$ are depicted in Figure \ref{fig:treegram of C3}.
\begin{figure}[ht!]
\centering
 \begin{tikzpicture}
    \begin{axis} [ 
    height=4cm,
    axis y line=left, 
    axis x line=middle,
    xlabel=$\epsilon$,
    ylabel=$L(C_3)$,
    ytick={0.5,1},
    every axis y label/.style={at={(current axis.north west)},above=2mm},   
    yticklabels={$[\gamma_0]$, $[\gamma_1]$},
    xtick={0,1.33},
    xticklabels={0,1},
    xmin=0, xmax=2.5,
    ymin=0, ymax=1.5,]
    \addplot +[mark=none,color=blue, ultra thick,opacity=0.2] coordinates {(1.33,0.5)(1.33,1)};
    \addplot[domain=0:2.5,color=blue,ultra thick,opacity=0.2]{0.5};
    \end{axis}
    \end{tikzpicture}\hspace{1cm}
    \begin{tikzpicture}
    \begin{axis} [ 
    height=4.5cm,
    axis y line=left, 
    axis x line=middle,
    every axis y label/.style={at={(current axis.north west)},above=2mm},
    xlabel=$\epsilon$,
    ylabel=$L(C_4)$,
    ytick={0.5,1,1.5},
    yticklabels={$[\gamma_0]$, $[\gamma_1^{\ast k}]$, $[\gamma_2]$},
    xtick={0,1,2},
    xticklabels={0,1,2},
    xmin=0, xmax=2.5,
    ymin=0, ymax=2,]
    \addplot +[mark=none,color=blue, ultra thick,opacity=0.2] coordinates {(2,1.5) (2,0.5)};
    \addplot[domain=0:2.5,color=blue,ultra thick,opacity=0.2]{0.5};
    \addplot[domain=1:2,color=blue,  ultra thick,opacity=0.2]{1};   
    \end{axis}
    \end{tikzpicture}
    \caption{Treegram over $L(C_3)$ (left) and $L(C_4)$ (right).}  \label{fig:treegram of C3} 
\end{figure}
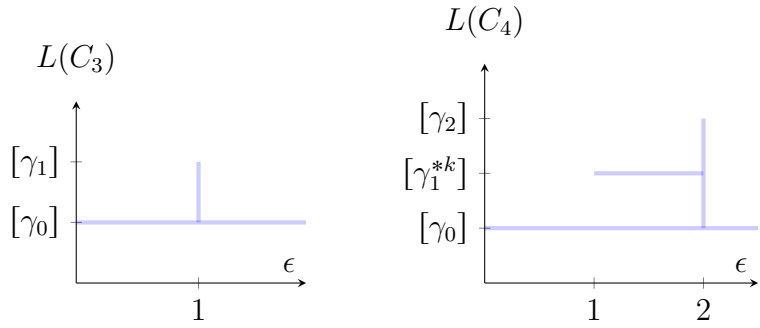

\begin{proposition}\label{prop:S_n} Fix $n\in \Z_{\geq 3}$. Then,
\begin{enumerate}
\item $\ppi_1(S_n)=0;$
\item $L(S_n)=\{[\lambda_0],[\lambda_1],[\lambda_2]\}$, where $\lambda_0=0$, $\lambda_1=010$ and $\lambda_2=0120$. And the pseudo-metric $\mu^{(1)}_{S_n}$ is given by the distance matrix
\begin{equation*}
\mu^{(1)}_{S_n}=
  \begin{blockarray}{*{3}{c} l}
    \begin{block}{*{3}{>{$\footnotesize}c<{$}} l}
      $[\lambda_0]$ & $[\lambda_1]$ & $[\lambda_2]$ \\
    \end{block}
    \begin{block}{[*{3}{c}]>{$\footnotesize}l<{$}}
      0 & 1 & 2 & $[\lambda_0]$ \\
       & 1 & 2 & $[\lambda_1]$ \\
       & & 2 & $[\lambda_2]$ \\
    \end{block}
  \end{blockarray}
\end{equation*}
\end{enumerate}
\end{proposition}

\begin{proof} The possible birth times of discrete loops in $S_n$ are $0,1$ and $2$. Since $S_n$ is a finite tree, by Proposition \ref{ex:finite-tree} we have $\ppi_1^r(S_n)=0$ for all $r$. It then follows from Proposition \ref{cor:L-epsilon} that there is only one equivalence class for each possible birth time, which are $\lambda_0$, $\lambda_1$ and $\lambda_2$.
\end{proof}

The corresponding generalized subdengrogram over $L(S_n)$ is depicted in Figure \ref{fig:treegram of Sn}.
\begin{figure}[ht!]
\centering
    \begin{tikzpicture}
    \begin{axis} [ 
    height=4.5cm,
    axis y line=left, 
    axis x line=middle,
    every axis y label/.style={at={(current axis.north west)},above=2mm},
    xlabel=$\epsilon$,
    ylabel=$L(S_n)$,
    ytick={0.5,1,1.5},
    yticklabels={$[\lambda_0]$, $[\lambda_1]$, $[\lambda_2]$},
    xtick={0,1,2},
    xticklabels={0,1,2},
    xmin=0, xmax=2.5,
    ymin=0, ymax=2,]
    \addplot +[mark=none,color=blue, ultra thick,opacity=0.2] coordinates {(2,1.5) (2,0.5)};
    \addplot[domain=0:2.5,color=blue,ultra thick,opacity=0.2]{0.5};
    \addplot[domain=1:2,color=blue, ultra thick,opacity=0.2]coordinates {(1,1) (1,0.5)};   
    \end{axis}
    \end{tikzpicture}
    \caption{Treegram over $L(S_n)$.}  \label{fig:treegram of Sn} 
\end{figure}

Next we compute the Gromov-Hausdorff distance between cycle graphs and star graphs, and then compare it with several distances given in Proposition \ref{prop:bounds for dgh of Cm and Sn}.
\begin{proposition}\label{prop:dgh of Cm and Sn} For $m,n\in \Z_{\geq 3}$, we have 
\[\dgh(C_{m},S_{n})=\begin{cases}\tfrac{1}{2}\left(\floor*{\tfrac{m}{2}}-1\right),&\mbox{if $m\geq 6$,}\\ 
1,&\mbox{if $3\leq m\leq 5$ and $m< n-1$,}\\
\tfrac{1}{2} &\mbox{if $3\leq m\leq 5$ and $m\geq n-1$.}
	\end{cases}
\]
\end{proposition}

\begin{remark}\label{rmk:dgh pt for Cm and Sn} Exactly same statement in Proposition \ref{prop:dgh of Cm and Sn} is true for $\dgh^{\pt}((C_m,0),(S_n,\bfzero))$ as well, because the proof still holds if we require all tripods to be pointed, i.e., containing $(0,\bfzero)$.
\end{remark}

\begin{proposition}\label{prop:bounds for dgh of Cm and Sn} For $m,n\in \Z_{\geq 3}$, we have the following:
\begin{enumerate}
    \item $\dgh\left(\left( C_{m},\mu^{(0)}_{C_{m}}\right),\left( S_{n},\mu^{(0)}_{S_{n}}\right)\right)=\tfrac{1}{2}$ when $m\neq n$, and $0$ otherwise.
    \item $\tfrac{1}{2}\cdot\di(\ppi_1(C_{m}),\ppi_1(S_n))= \tfrac{1}{2}\cdot\db(\dgm_1(C_{m}),\dgm_1(S_n))=\tfrac{1}{4}\cdot\left(\floor*{\tfrac{m+2}{3}}-1\right).$
\end{enumerate}
\end{proposition}

\begin{remark} Since Proposition \ref{prop:dgh of Cm and Sn} implies that $\dgh\left( C_m,S_n\right)\to\infty$ as $m\to \infty$, we can see that the value given by Proposition \ref{prop:bounds for dgh of Cm and Sn} (1) is not ideal as a lower bound for $\dgh(C_m,S_n)$. Although strictly less than $\dgh(C_m,S_n)$, Proposition \ref{prop:bounds for dgh of Cm and Sn} (2) increases to infinity as $m\to \infty$. 
\end{remark}

The lower bound in Proposition \ref{prop:property-GH} can be improved in the following case:
\begin{proposition}\label{prop:property-GH-rad} Let $(X,d_X)$ and $(Y,d_Y)$ be compact metric spaces. Suppose $\rad(Y)\leq \diam(X)$ and every $x\in X$ has an antipode, i.e. a point $\tilde{x}\in X$ such that $d_X(x,\tilde{x})=\diam(X)$. Then
\begin{equation}
\label{eq:property-GH-rad}
    \tfrac{1}{2}\left(\diam(X)-\rad(Y)\right)\leq \dgh(X,Y).
\end{equation}
\end{proposition}

\begin{proof} The eccentricity of a point $y\in Y$ is $\ecc(y):=\sup\{d_X(y,y'):y'\in Y\}.$ Note that $\rad(Y):=\inf\{\ecc(y):y\in Y\}.$ Let $y_0$ be a point in $Y$ such that $\ecc(y_0)=\rad(Y)$.
Suppose $R$ is an arbitrary tripod between $X$ and $Y$. Clearly, there exists some $x\in X$ such that $(x,y_0)\in R$. Let $\tilde{x}$ be an antipode of $x$. There exists some $y\in Y$ such that $(\tilde{x},y)\in R$. Therefore,
$$\diam(X)-\rad(Y)\leq d_{X}(x,\tilde{x})-d_Y(y_0,y)\leq \dis(R).$$
Since $R$ is arbitrary, we can conclude that Eq. (\ref{eq:property-GH-rad}) is true.
\end{proof}

\begin{myproof}{Proposition \ref{prop:dgh of Cm and Sn}}
The proof is divided into four cases:
\begin{enumerate}[label=(\alph*)]
\item \label{m=6} $m\geq 6$.
\item \label{m>n-1} $m\geq 4$ and $m\geq n-1$.
\item \label{m=4,5} $m=4$ or $5$, and $m<n-1$.
\item \label{m=3} $m=3$.
\end{enumerate}
To distinguish points in $C_m$ and $S_n$, we will write the vertex set of $C_m$ as $\{0,1,\cdots,m-1\}$ and the vertex set of $S_n$ as $\{\bfzero,\bfone,\cdots,\bfn-\bfone\}$. For any two positive integers $l$ and $k$, by $l \operatorname{Mod} k$ we will mean the remainder of the Euclidean division of $l$ divided by $k$.

When $m\geq 4$, $\diam(C_m)=\floor*{\tfrac{m}{2}}\geq 1=\rad(S_n)$ and each point in $C_m$ has an antipode. Thus, we can apply Proposition \ref{prop:property-GH-rad} to obtain
\begin{equation}\label{eq:lower bound}
\tfrac{1}{2}\cdot\left(\floor*{\tfrac{m}{2}}-1\right)\leq \dgh(C_{m},S_{n})\leq \tfrac{1}{2}\cdot \floor*{\tfrac{m}{2}}.
\end{equation}

For case \ref{m=6}: $m=6$, we claim that $\dgh(C_{m},S_{n})=\tfrac{1}{2}\cdot\left(\floor*{\tfrac{m}{2}}-1\right)$. Because of Equation (\ref{eq:lower bound}), it remains to construct a tripod $R$ such that $\dis(R)\leq\floor*{\tfrac{m}{2}}-1.$ In other words, we want the tripod $R$ to satisfy the following condition: $\forall (i_1,\bfj_{\bfone}),(i_2,\bfj_{\bftwo})\in R$,
\[|d_{C_m}(i_1,i_2)-d_{S_n}(\bfj_{\bfone},\bfj_{\bftwo})|< \floor*{\tfrac{m}{2}}.\]
Given $m\geq 6$, it is always true that $\forall (i_1,\bfj_{\bfone}),(i_2,\bfj_{\bftwo})\in R$,
\begin{equation}\label{eq:d_Cm-d_Sn}
    \floor*{\tfrac{m}{2}}<-2\leq d_{C_m}(i_1,i_2)-d_{S_n}(\bfj_{\bfone},\bfj_{\bftwo})\leq \floor*{\tfrac{m}{2}}.
\end{equation}
Since the leftmost inequality of Equation (\ref{eq:d_Cm-d_Sn}) is strict, we have that
\begin{align*}
&|d_{C_m}(i_1,i_2)-d_{S_n}(\bfj_{\bfone},\bfj_{\bftwo})|= \floor*{\tfrac{m}{2}},\\
\text{iff } & d_{C_m}(i_1,i_2)-d_{S_n}(\bfj_{\bfone},\bfj_{\bftwo})= \floor*{\tfrac{m}{2}},\\
\text{iff }  & d_{C_m}(i_1,i_2)=\floor*{\tfrac{m}{2}}\text{ and }d_{S_n}(\bfj_{\bfone},\bfj_{\bftwo})=0.
\end{align*}
Therefore, to construct a tripod $R$ with $\dis(R)\leq\floor*{\tfrac{m}{2}}-1$, it suffices to construct a tripod where any pair of antipodes in $C_m$ do not correspond to the same point in $S_n$. More precisely, $R$ shall satisfy the condition:
\begin{align*}\label{eq:condition*}\tag{$\ast$}
\forall (i_1,\bfj_{\bfone}),(i_2,\bfj_{\bftwo})\in R,\text{ if } d_{C_m}(i_1,i_2)=\floor*{\tfrac{m}{2}},\text{ then } \bfj_{\bfone}\neq\bfj_{\bftwo}.
\end{align*}

Suppose $m=2k$ for some $k\geq 3$. When $m\geq n$, we define a tripod $R_1$ between $C_m$ and $S_n$ by
$$R_1:=\left\{(i,\,\bftwo\bfi\operatorname{Mod} n),\left( k+i,\,(\bftwo\bfi\bfplus\bfone)\operatorname{Mod} n\right): i=0,\cdots,k-1\right\}.$$
When $m< n$, we define a tripod $R_2$ between $C_m$ and $S_n$ by
$$R_2:=\left\{(i\operatorname{Mod} m,\,\bftwo\bfi),\left( (k+i)\operatorname{Mod} m,\,\bftwo\bfi\bfplus\bfone\right): i=0,\cdots,\floor*{\tfrac{n-1}{2}}\right\}.$$

Next we assume $m=2k+1$ for some integer $k\geq 3$. When $m\geq n$, we define a tripod $R_3$ between $C_m$ and $S_n$ by
\[R_3:=\left\{(0,\bfzero),(k-i,(\bftwo\bfi\bfplus\bfone)\operatorname{Mod} n),\left( m-i,\bftwo\bfi \operatorname{Mod} n\right): i=0,\cdots,k-1\right\}.\]
When $m< n$, we define a tripod $R_4$ between $C_m$ and $S_n$ by
\[R_4:=\left\{((k-i)\operatorname{Mod} m,\,\bftwo\bfi\bfplus\bfone),\left( (m-i)\operatorname{Mod} m,\,\bftwo\bfi\right): i=0,\cdots,\floor*{\tfrac{n-1}{2}}\right\}.\]

It can be directly checked that $R_1$, $R_2$, $R_3$ and $R_4$ are tripods satisfying Condition (\ref{eq:condition*}).

For case \ref{m>n-1}: $m\geq 4$ and $m\geq n-1$, we claim that $\dgh(C_{m},S_{n})=\tfrac{1}{2}\cdot\left(\floor*{\tfrac{m}{2}}-1\right)$. Let $k:=\floor*{\tfrac{m}{2}}.$ Because of Equation (\ref{eq:lower bound}), it remains to construct a tripod $R$ such that $\dis(R)\leq\floor*{\tfrac{m}{2}}-1.$ We will use the same constructions $R_1$, $R_2$, $R_3$ and $R_4$ as case \ref{m=6} in corresponding cases. Indeed, when $m$ is even, we use $R_1$ if $m\geq n$ and $R_2$ if $m=n-1$; when $m$ is odd, we use $R_3$ if $m\geq n$ and $R_4$ if $m=n-1$. Assume $l=1,2,3,4$. Let $i_1$ and $i_2$ be any two points from $C_m$, and suppose $(i_1,\bfj_{\bfone})$ are $(i_2,\bfj_{\bftwo})$ in $R_l$. If $i_2=i_1+k$, then it can be checked that
\begin{equation}\label{eq:dis bound}
    0\leq k-2\leq d_{C_m}(i_1,i_2)-d_{S_n}(\bfj_{\bfone},\bfj_{\bftwo})\leq k-1.
\end{equation}
If $d_{C_m}(i_1,i_2)\leq k-1$, then
\begin{equation}\label{eq:dis bound 1}
-(k-1)\leq-1\leq d_{C_m}(i_1,i_2)-d_{S_n}(\bfj_{\bfone},\bfj_{\bftwo})\leq k-1.
\end{equation}
The second inequality of Equation (\ref{eq:dis bound 1}) is true because when $i_1=i_2\neq 0$, we always have $\bfj_{\bfone}=\bfj_{\bftwo}$ due to the construction of $R_l$, and when $i_1=i_2=0$, we have $d_{S_n}(\bfj_{\bfone},\bfj_{\bftwo})\leq 1$ (for $R_2$ and $R_4$, this argument relies on the condition $m=n-1$). By Equation (\ref{eq:dis bound}) and Equation (\ref{eq:dis bound 1}), we have $\dis(R_l)\leq k-1$. 

For case \ref{m=4,5}: $m=4$ or $5$, and $m<n-1$, we claim that $\dgh(C_{m},S_{n})=\tfrac{1}{2}\cdot\floor*{\tfrac{m}{2}}=1.$ By the upper bound from Equation (\ref{eq:lower bound}), it suffices to show that for each tripod $R$ between $C_m$ and $S_n$, $\dis(R)\geq \floor*{\tfrac{m}{2}}=2.$ By the pigeonhole principle, since $m<n-1$, there must be two elements $\bfi,\bfj\in \{\bfone,\cdots,\bfn-\bfone\}$ such that $(l,\bfi),(l,\bfj)\in R$ for some $l\in C_m$. Therefore, $\dis(R)\geq d_{S_n}(\bfi,\bfj)=2.$

For case \ref{m=3}: $m=3$, since $\diam(C_3)=1$ and $\diam(S_n)=2$, Proposition \ref{prop:property-GH} implies that 
\begin{equation*}\label{eq:diam bound}
    \tfrac{1}{2}=\tfrac{1}{2}\cdot\left|  1-2\right| \leq \dgh(C_{m},S_{n})\leq \tfrac{1}{2}\cdot \max\left\{1,2\right\}=1.
\end{equation*} 
If $m<n-1$, we can apply the pigeonhole principle as in case \ref{m=4,5}, to show each tripod between $C_m$ and $S_n$ has distortion $2$. Thus, $\dgh(C_{m},S_{n})$ reaches the upper bound value $1$. If $m\geq n-1$, i.e., $n=3$ or $4$, we claim that $\dgh(C_{m},S_{n})=\tfrac{1}{2}$ by constructing tripod with distortion $1$. Indeed, we can again utilize the construction $R_3$ for $n=3$ and $R_4$ for $n=4$. It is not hard to verify that their distortions are both $1$.
\end{myproof}

\begin{myproof}{Proposition \ref{prop:bounds for dgh of Cm and Sn}} 
For (1), we first notice that $\left( C_{m},\mu^{(0)}_{C_{m}}\right)\cong(E_{m},d_{E_{m}})$ and $\left( S_n,\mu^{(0)}_{S_n})\right)\cong(E_n,d_{E_n})$, where the metric space $(E_n,d_{E_n})$ is given in page \pageref{para:E_n}. When $m=n$, it is clear that $\dgh\left((E_n,d_{E_n}),(E_n,d_{E_n})\right)=0$. When $m\neq n$, suppose without loss of generality that $m>n$. We claim that $$\dgh\left((E_m,d_{E_m}),(E_n,d_{E_n})\right)=\tfrac{1}{2}.$$ By Proposition \ref{prop:property-GH}, $\dgh\left((E_m,d_{E_m}),(E_n,d_{E_n})\right)\leq\tfrac{1}{2}$. It remains to show that $\tfrac{1}{2}$ is also a lower bound for $\dgh\left((E_m,d_{E_m}),(E_n,d_{E_n})\right)$. Given any pointed tripod $R,$ as $E_m$ has more points than $E_n$, some distinct points in $E_m$ must correspond to the same point in $E_n$. Thus, $\dis(R)\geq 1-0=1.$ 

For Part (2), note that $$\di(\ppi_1(C_{m}),\ppi_1(S_n))=\di(\pH_1^{\VR}(C_{m}),\pH_1^{\VR}(S_n))= \db(\dgm_1(C_{m}),\dgm_1(S_n)).$$ It follows from Proposition \ref{prop:C_n} (1) that $\dgm_1(C_{m})=\left\{ \left( 1,\floor*{\tfrac{m+2}{3}}\right) \right\}$ for $m>3$ and $\emptyset$ for $m=3$, and from Proposition \ref{prop:S_n} (1) it follows that $\dgm_1(S_n)=\emptyset.$ Thus, $$\db(\dgm_1(C_{m}),\dgm_1(S_n))=\tfrac{1}{2}\cdot\left(\floor*{\tfrac{m+2}{3}}-1\right).$$
\end{myproof}

\subsection{Discretization of the geodesic circle}\label{sec:discretization-circle}
Let the unit circle $\mathbb{S}^1$ be embedded into the complex plane with the center at 0, i.e., $\mathbb{S}^1=\{z\in \C:|z|=1\}$. 
Again, we equip $\bbS^1$ with its geodesic metric $d$. For $n\in \Z_{\geq 0}$, we consider the metric subspace of $(\bbS^1,d)$: $$\Delta_n:=\left\{e^{\tfrac{2\pi i k}{n}}: k\in [n-1] \right\},$$ where $i=\sqrt{-1}$. 
For simplicity, we write $\Delta_n=\{0,1,\cdots, n-1\}$ and assume 0 is the basepoint. Also, we denote $d|_{\Delta_n\times\Delta_n}$ by $d_n$ and note that $\left( \Delta_n,d_n\right)\cong\left( C_n,\tfrac{2\pi}{n}d_{C_n}\right)$. For example, Figure \ref{fig:delta3-4} shows the graphs corresponding to $\Delta_3$ and $\Delta_4$.
\begin{figure}[ht!]	  
\centering
\begin{tikzpicture}[scale=0.8]
\draw [dashed, radius=1] (-2,0) circle ;
\filldraw (-2,1) [color=blue] circle[radius=1.5pt];
\filldraw (-1.135,-0.5)[color=blue] circle[radius=1.5pt];
\filldraw (-2.866,-0.5) [color=blue] circle[radius=1.5pt];
\node[below right=0.4pt of {(-1.135,-0.5)}, outer sep=1.5pt,fill=white] {1};
\node[below left=0.4pt of {(-2.866,-0.5)}, outer sep=1.5pt,fill=white] {2};
\node[above=0.4pt of {(-2,1)}, outer sep=1.5pt,fill=white] {0};
\draw [dashed, radius=1] (2,0) circle ;
\filldraw (2,1) [color=blue] circle[radius=1.5pt];
\filldraw (2,-1)[color=blue] circle[radius=1.5pt];
\filldraw (3,0) [color=blue] circle[radius=1.5pt];
\filldraw (1,0) [color=blue] circle[radius=1.5pt];
\node[right=0.4pt of {(3,0)}, outer sep=1.5pt,fill=white] {1};
\node[below=0.4pt of {(2,-1)}, outer sep=1.5pt,fill=white] {2};
\node[above=0.4pt of {(2,1)}, outer sep=1.5pt,fill=white] {0};
\node[left=0.4pt of {(1,0)}, outer sep=1.5pt,fill=white] {3};
    \end{tikzpicture} 
\caption{Graphs of $\Delta_3$ (left) and $\Delta_4$ (right).}\label{fig:delta3-4}
\end{figure}
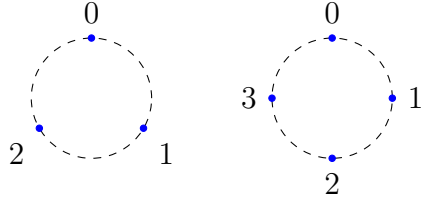

\begin{proposition}\label{prop:distance for Delta_n} We have the following:
\begin{enumerate}
    \item $\dgh((\Delta_3,d_3),(\Delta_4,d_4))=\tfrac{\pi}{4}$.
    \item $\dgh\left(\left(\Delta_3,\mu^{(0)}_{\Delta_3}\right),\left(\Delta_4,\mu^{(0)}_{\Delta_4}\right)\right)=\tfrac{\pi}{4}$.
   \item $\dgh\left(\left( L(\Delta_3)\, ,\mu^{(1)}_{\Delta_3}\right),\left( L(\Delta_4)\, ,\mu^{(1)}_{\Delta_4}\right)\right)=\tfrac{\pi}{6}$.
    \item $\tfrac{1}{2}\cdot \di(\ppi_1(\Delta_3),\ppi_1(\Delta_4))= \tfrac{1}{2}\cdot\db(\dgm_1(\Delta_3),\dgm_1(\Delta_4))=\tfrac{\pi}{8}.$
\end{enumerate}
\end{proposition}

\begin{proof} For (1), we first note that $\im(d_3)=\left\{0,\tfrac{2\pi}{3}\right\}$ and $\im(d_4)=\left\{0,\tfrac{\pi}{2},\pi\right\}$. For any pointed tripod $R:(\Delta_3,0)\xtwoheadleftarrow{\phi}Z\xtwoheadrightarrow{\psi}(\Delta_4,0)$, we have 
$$\tfrac{\pi}{2}=\min\left\{\pi-0,\tfrac{\pi}{2}-0\right\}\leq \max_{z,z'\in Z}\left|  d_3(\phi(z),\phi(z'))-d_4(\psi(z),\psi(z'))\right| \leq
\dis(R),$$ 
since $\Delta_4$ has more elements than $\Delta_3$. It follows that $\dgh(\Delta_3,\Delta_4)\geq\tfrac{\pi}{4}.$ Next we construct a tripod $R$ whose distortion is $\tfrac{\pi}{2}$, by first defining two set maps $f$ and $g$ as in Figure \ref{fig:set maps}. We set $Z=\Delta_3\sqcup\Delta_4$, and define set maps $\phi=(\Id_{\Delta_3},f):\Delta_3\sqcup\Delta_4\to\Delta_3$ and $\psi=(g,\Id_{\Delta_4}):\Delta_3\sqcup\Delta_4\to\Delta_4$. This forms a pointed tripod $R$. By direct calculation, we obtain $\dis(R)=\tfrac{\pi}{2}.$
 
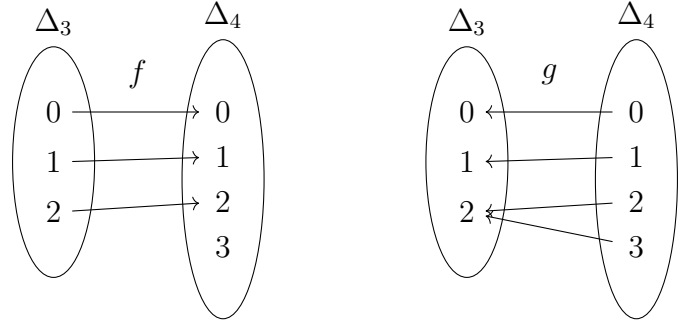
\begin{figure}[ht!]
\centering
  \begin{tikzpicture}
    [group/.style={ellipse, draw, minimum height=30pt, minimum width=20pt, label=above:#1},
      my dot/.style={circle, fill, minimum width=2.5pt, label=above:#1, inner sep=0pt}]
    \node (a)  {0};
    \node (b) [below=3pt of a] {1};
    \node (c) [below=3pt of b] {2};    
    \node (d) [right=50pt of a] {0};
    \node (e) [below=1pt of d] {1};
    \node (f) [below=1pt of e] {2};
    \node (g) [below=1pt of f] {3};
    \foreach \i/\j in {a/d,b/e,c/f}
      \draw [->, shorten >=2pt] (\i) -- (\j);
    \node [fit=(a) (b) (c), group=$\Delta_3$] {};
    \node [fit=(d) (e) (f) (g), group=$\Delta_4$] {};
    \node[right =20pt of {(0.1,0.5)}, outer sep=1.5pt,fill=white] {$f$};
  \end{tikzpicture}\hspace{2cm}
    \begin{tikzpicture}
    [group/.style={ellipse, draw, minimum height=30pt, minimum width=20pt, label=above:#1},
      my dot/.style={circle, fill, minimum width=2.5pt, label=above:#1, inner sep=0pt}]
    \node (a)  {0};
    \node (b) [below=3pt of a] {1};
    \node (c) [below=3pt of b] {2};    
    \node (d) [right=50pt of a] {0};
    \node (e) [below=1pt of d] {1};
    \node (f) [below=1pt of e] {2};
    \node (g) [below=1pt of f] {3};
    \foreach \i/\j in {a/d,b/e,c/f,c/g}
      \draw [<-, shorten >=2pt] (\i) -- (\j);
    \node [fit=(a) (b) (c), group=$\Delta_3$] {};
    \node [fit=(d) (e) (f) (g), group=$\Delta_4$] {};
    \node[right =20pt of {(0.1,0.5)}, outer sep=1.5pt,fill=white] {$g$};
  \end{tikzpicture}
    \caption{Two set maps $f:\{0,1,2\}\to \{0,1,2,3\}$ and $g:\{0,1,2,3\}\to \{0,1,2\}$.}
    \label{fig:set maps}
\end{figure} 

For (2), apply a similar argument as in Part (1): since $\Delta_4$ has more elements than $\Delta_3$, any pointed tripod $R$ satisfies
$$\tfrac{\pi}{2}=\tfrac{\pi}{2}-0\leq\dis(R).$$
It follows that $\dgh\left(\left(\Delta_3,\mu^{(0)}_{\Delta_3}\right),\left(\Delta_4,\mu^{(0)}_{\Delta_4}\right)\right)\geq\tfrac{\pi}{4}$. To see that $\dgh\left(\left(\Delta_3,\mu^{(0)}_{\Delta_3}\right),\left(\Delta_4,\mu^{(0)}_{\Delta_4}\right)\right)=\tfrac{\pi}{4}$, we utilize the same set maps $f$ and $g$ given in Figure \ref{fig:set maps} and construct the same tripod $R$ as in (1). With the metrics $\mu^{(0)}_{\Delta_3}$ and $\mu^{(0)}_{\Delta_4},$ a direct calculation shows that $\dis(R)=\tfrac{\pi}{4}$.

For (3), we first apply Proposition \ref{prop:property-GH} to see that 
$$\dgh\left(\left( L(\Delta_3)\, ,\mu^{(1)}_{\Delta_3}\right),\left( L(\Delta_4)\, ,\mu^{(1)}_{\Delta_4}\right)\right)\geq \tfrac{1}{2}\cdot\left(\pi-\tfrac{2\pi}{3}\right)=\tfrac{\pi}{6}.$$ 
Next we construct a tripod $R$ whose distortion is $\tfrac{\pi}{3}$, by first defining two set maps $f$ and $g$ as in Figure \ref{fig:set maps for L}. We set $Z=L(\Delta_3)\sqcup L(\Delta_4)$, and define set maps $\phi=(\Id_{L(\Delta_3)},f)$ and $\psi=(g,\Id_{ L(\Delta_4)})$. These define a pointed tripod $R$, with $\dis(R)=\tfrac{\pi}{3}.$

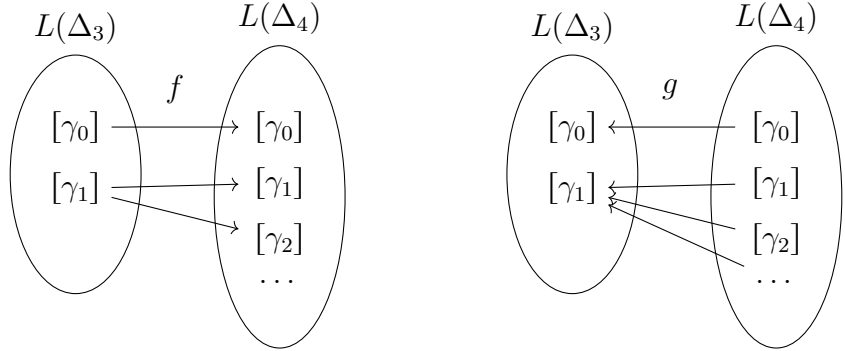
\begin{figure}[ht!]
\centering
  \begin{tikzpicture} 
    [group/.style={ellipse, draw, minimum height=30pt, minimum width=20pt, label=above:#1},
      my dot/.style={circle, fill, minimum width=2.5pt, label=above:#1, inner sep=0pt}]
    \node (a)  {$[\gamma_0]$};
    \node (b) [below=3pt of a] {$[\gamma_1]$};
    \node (d) [right=50pt of a] {$[\gamma_0]$};
    \node (e) [below=1pt of d] {$[\gamma_1]$};
    \node (f) [below=1pt of e] {$[\gamma_2]$};
    \node (g) [below=1pt of f] {$\dots$};
    \foreach \i/\j in {a/d,b/e,b/f}
      \draw [->, shorten >=2pt] (\i) -- (\j);
    \node [fit=(a) (b) (c), group=$L(\Delta_3)$] {};
    \node [fit=(d) (e) (f) (g), group=$L(\Delta_4$)] {};
    \node[right =20pt of {(0.3,0.5)}, outer sep=1.5pt,fill=white] {$f$};
  \end{tikzpicture}\hspace{2cm}
    \begin{tikzpicture}
    [group/.style={ellipse, draw, minimum height=30pt, minimum width=20pt, label=above:#1},
      my dot/.style={circle, fill, minimum width=2.5pt, label=above:#1, inner sep=0pt}]
    \node (a)  {$[\gamma_0]$};
    \node (b) [below=3pt of a] {$[\gamma_1]$};
    \node (d) [right=50pt of a] {$[\gamma_0]$};
    \node (e) [below=1pt of d] {$[\gamma_1]$};
    \node (f) [below=1pt of e] {$[\gamma_2]$};
    \node (g) [below=1pt of f] {$\dots$};
    \foreach \i/\j in {a/d,b/e,b/f,b/g}
      \draw [<-, shorten >=2pt] (\i) -- (\j);
    \node [fit=(a) (b) (c), group=$L(\Delta_3)$] {};
    \node [fit=(d) (e) (f) (g), group=$L(\Delta_4$)] {};
    \node[right =20pt of {(0.3,0.5)}, outer sep=1.5pt,fill=white] {$g$};
  \end{tikzpicture}
    \caption{Two set maps $f:L(\Delta_3)\to L(\Delta_4)$ and $g:L(\Delta_4)\to L(\Delta_3)$. Recall from page \pageref{para:gamma_r} that in the cycle graph $C_n$, we defined the discrete loop $\gamma_r$ for $r=1,\cdots,\floor*{\tfrac{n-1}{3}}$. As $\Delta_n$ and $C_n$ have the same underlying space $\{0,\cdots,n-1\}$, the notation $\gamma_r$ can be inherited. Although $\gamma_r$ depends on $n$, we will not specify $n$ in the notation when the concept is clear. The dots in $L(\Delta_4)$ represent $[\gamma_1^{\ast k}]$ for integer $k\neq 0,1$. }
    \label{fig:set maps for L}
\end{figure} 

For (4), because $\left( \Delta_n,d_n\right)\cong\left( C_n,\tfrac{2\pi}{n}d_{C_n}\right)$, we have
	\[\ppi_1(\Delta_3)=\bbmzero\text{ and }\ppi_1(\Delta_4)=\Z\left[ \tfrac{\pi}{2},\pi\right).\] 
Therefore, $\dgm_1(\Delta_3)=\emptyset$ and $\dgm_1(\Delta_4)=\{(\tfrac{\pi}{2},\pi)\}$. It follows immediately that
$$ \di(\ppi_1(\Delta_3),\ppi_1(\Delta_4))= \di(\pH_1^{\VR}(\Delta_3),\pH_1^{\VR}(\Delta_4))= \db(\dgm_1(\Delta_3),\dgm_1(\Delta_4))=\tfrac{\pi}{4}.$$
\end{proof}

\begin{remark}The symmetry property of $(\Delta_n,d_n)$ guarantees that $\dgh((\Delta_3,d_3),(\Delta_4,d_4))=\dgh^{\pt}((\Delta_3,0,d_3),(\Delta_4,0,d_4))$. In Theorem \ref{thm:stab-pi0}, Theorem \ref{thm:stab-mu1} and Theorem \ref{thm:stab-di-ppi_1}, we have proved that (2), (3) and (4) are lower bounds of (1), the Gromov-Hausdorff distance, under certain restrictions. 
\end{remark}


\end{document}